\documentclass[3p,preprint]{elsarticle}
\usepackage{graphicx}
\usepackage{amssymb}
\usepackage{amsthm}
\usepackage{amsmath}
\usepackage{revsymb}
\usepackage{empheq}
\usepackage{mathtools}
\usepackage{subcaption}
\usepackage{lineno} 
\usepackage{enumerate}
\usepackage{multirow} 
\usepackage{array} 
\usepackage{booktabs} 
\usepackage{tabularx}
\usepackage{longtable}
\usepackage{bm}
\usepackage{setspace}
\usepackage{float}
\usepackage{bbm}
\usepackage[hidelinks]{hyperref}
\usepackage{algorithm}
\usepackage{algorithmic}
\journal{arXiv}
\usepackage[figuresright]{rotating}

\newcounter{thmcounter}[section]
\theoremstyle{definition}
\newtheorem{definition}[thmcounter]{Definition}

\newtheorem{remark}[thmcounter]{Remark}
\theoremstyle{plain}
\newtheorem{theorem}[thmcounter]{Theorem}

\newtheorem{proposition}[thmcounter]{Proposition}

\newtheorem{corollary}[thmcounter]{Corollary}

\numberwithin{equation}{section}
\numberwithin{thmcounter}{section}

\allowdisplaybreaks[3]

\newcommand{\doublehookrightarrowalt}{%
    \lhook\joinrel\relbar\mspace{-12mu}\hookrightarrow
}

\begin{document}

\begin{frontmatter}

\title{On a nonlinear nonlocal reaction-diffusion system applied to image restoration}

\author{Yuhang Li$^{\ \textrm{a}}$}
\ead{mathlyh@stu.hit.edu.cn}
\author{Zhichang Guo$^{\ \textrm{a}}$}
\ead{mathgzc@hit.edu.cn}
\author{Jingfeng Shao$^{\ \textrm{a},\textrm{b},}\cormark[1]$}
\ead{sjfmath@163.com}
\author{Boying Wu$^{\ \textrm{a}}$}
\ead{mathwby@hit.edu.cn}
\cortext[1]{Corresponding author.}

\address{$^{\rm{a}}$The Department of Mathematics, Harbin Institute of Technology, Harbin 150001, China}
\address{$^{\rm{b}}$College of Mathematics and Information Science, Guangxi University, Nanning 530004, China}

\begin{abstract}
    This paper deals with a novel nonlinear coupled nonlocal reaction-diffusion system 
proposed for image restoration, characterized by the advantages of preserving 
low gray level features and textures. The gray level indicator in the proposed 
model is regularized using a new method based on porous media type equations, which is 
suitable for recovering noisy blurred images. The well-posedness, regularity, and other 
properties of the model are investigated, addressing the lack of theoretical analysis in 
those existing similar types of models. Numerical experiments conducted on texture and 
satellite images demonstrate the effectiveness of the proposed model in denoising 
and deblurring tasks.
\end{abstract}


\end{frontmatter}

\section{Introduction}\label{sec1}

\subsection{{Diffusions and regularizations for image restoration}}
In many image processing applications, the acquired images are often blurred 
and corrupted by noise \cite{WELK2005,CHAN2005}. 
The degradation model of an image can typically be formulated as 
$f=K{u^{\dagger}}+n$,
where $K$ is some linear operator, especially convolution operator, 
$n$ is noise, $f$ is the observed image, and {$u^{\dagger}$} is the original image. 
In this paper, we mainly focus on the non-blind image restoration problem, 
which aims to recover a clean image $u$ from a noisy blurred image $f$ 
when $K$ is known exactly. 
{As this is an ill-posed inverse problem, regularization techniques 
are often required, typically through the minimization of} the following energy functional:
\begin{equation}  \label{eqn:energy}
    E(u) = \varPhi(u) + \frac{\lambda}{2} \|Ku-f\|_{2}^2,  
\end{equation}
where $\|Ku-f\|_2^{{2}}$ is called the fidelity term and $\varPhi(u)$ is called the 
regularization term, $\lambda \geq 0$ is a parameter to balance the fidelity term 
and regularization term \cite{CHAN2005,AUBERT1997}.
{The regularization term} $\varPhi(u)$ {imposes a priori 
constraint on the unknown minimizer $u$ and} is generally chosen 
as {$\frac{1}{2}\int_{\Omega} G(|\nabla u|^2) dx$} in 
earlier methods, yielding the gradient flow
{(or the subgradient flow if $E$ is not Fr\'{e}chet differentiable)}
corresponding to \eqref{eqn:energy} as
\begin{equation}  \label{eqn:second-order-PDE}
    u_t = 
        \operatorname{div} \big(g(|\nabla u|^2) 
        \nabla u \big) - \lambda K'(Ku-f)
\end{equation}
is a reaction-diffusion equation, where $K'$ is the adjoint operator of $K$, 
{$G: \mathbb{R}_{\geq 0} \to \mathbb{R}$, $g(s) = \frac{d}{ds} G(s)$}. 
Many models have been proposed for image restoration.
Quadratic Tikhonov regularization \cite{TIKHONOV1977}, which amounts to taking $G(s) = s$ and 
minimizing \eqref{eqn:energy} in $H^1(\Omega)$, leads to the gradient descent flow 
\begin{equation}  \label{eqn:heat-reaction}
u_t = \Delta u - \lambda K'(Ku-f),
\end{equation}
and it is well known that sharp edges cannot be preserved under \eqref{eqn:heat-reaction}.
In the well-known ROF model {proposed by Rudin, Osher and Fatemi} 
\cite{RUDIN1992,RUDIN1994} {(see also \cite{CAI2012}), total variation (TV) regularization is employed 
by selecting $G(s)=2\sqrt{s}$ and minimizing \eqref{eqn:energy} in $BV(\Omega)$}, 
which results in {the subgradient flow}
\begin{equation}  \label{eqn:tv-reaction} 
    {u_t = \operatorname{div} \left(\frac{\nabla u}{|\nabla u|}\right) 
- \lambda K'(Ku-f).} 
\end{equation}
The use of TV can effectively preserve edges but easily cause a staircase effect. 
Various improved versions \cite{TAKEDA2008,HUANG2008,BREDIES2010} have been proposed 
to eliminate this effect. Another popular class of choices 
is to {employ} Perona-Malik 
diffusion \cite{PERONA1990,WELK2005}, {that is, using the flow} 
\begin{equation}  \label{eqn:pm-reaction} 
    {u_t = \operatorname{div} \left(\frac{1}{1+k_1 |\nabla u|^2} \nabla u\right) 
- \lambda K'(Ku-f)}
\end{equation}
{generated by \eqref{eqn:energy} with $G(s) = \frac{1}{k_1} \log(1+ k_1 s)$.}
Perona-Malik diffusion is a forward-backward diffusion process which is known to 
sharpen edges effectively but {it is mathematically ill-posed and can} cause 
staircase instability \cite{WEICKERT1997-2,GUIDOTTI2015}. 
To reduce the noise sensitivity 
of Perona-Malik diffusion \cite{WEICKERT1997}, various regularization 
methods \cite{CATTE1992,GUIDOTTI2009,GUIDOTTI2012,GUIDOTTI2013} 
can be utilized to enhance its stability. Higher order regularization terms and PDEs 
have been proposed to overcome the staircase effect of edge-preserving second-order PDEs. 
Some early works \cite{YOU2000,LYSAKER2003} indicated that methods based on 
fourth-order PDEs can effectively reduce the staircase effect but encourage
piecewise planar solutions, with modification can be found 
in \cite{BERTOZZI2004,GUIDOTTI2011,WEN2022}.  

{Besides being interpreted as a gradient descent flow of an energy functional, 
a PDE of the form \eqref{eqn:second-order-PDE} can also be directly understood 
from the perspective of reaction-diffusion equations. From the physical point of view, 
reaction-diffusion equation describes the evolution of the concentration of a substance 
in space and time. The diffusion represents the transport of the substance from regions 
of high concentration to regions of low concentration,
and the reaction process accounts for transformation or other 
dynamic processes \cite{VOLPERT2014,ROUSSEL2019}. 
The classical diffusion process for a density function $u$ originates from the principle of mass 
conservation \cite{MEHRER2007}, expressed as 
\begin{equation} \label{eqn:mass-conservation}
    u_t +\operatorname*{div}(\vec{J}) = 0,
\end{equation}
where $\vec{J}$ denotes the diffusion flux. According to Fick’s \textit{first law of diffusion} \cite{MEHRER2007}, 
\begin{equation} \label{eqn:fick-law}
    \vec{J} = -c \nabla u,
\end{equation}
where $c$ is the diffusion coefficient measuring the rate of diffusion. 
Diffusion has a smoothing effect on spatial heterogeneity, and when the concentration is treated as
the gray level at a given location, the diffusion process can be applied to suppress noise in
images \cite{WEICKERT1998-2,BARBU2019}. 
If $c$ is a constant, \eqref{eqn:mass-conservation} reduces to the heat equation, performing 
isotropic diffusion that smooths the image and thereby eliminates noise.
However, it is well-established that smoothing an image using the heat equation is equivalent 
to convolving it with a Gaussian kernel (indeed, $u(t)$ belongs to the $C^{\infty}$ class for $t>0$), 
which leads to the loss of edge information in the image. In practice, $c$ is usually chosen 
as a function $c(u)$ depending on $u$, resulting in \textit{anisotropic} diffusion, 
so that the diffusion rate varies between certain important image features and other regions in the image,
thereby enabling both noise removal and the preservation of these features \cite{BARBU2019}. 
By adding the reaction term $-\lambda K'(Ku-f)$ to the right-hand 
side of \eqref{eqn:mass-conservation} results in the reaction-diffusion equation
\begin{equation} \label{eqn:second-order-PDE-c}
u_t = \operatorname{div} (c(u) \nabla u) - \lambda K'(Ku-f),
\end{equation}
which can be applied to image restoration problem introduced at 
the beginning \cite{WELK2005,MORIGI2008,ZHAO2018}. 

In \eqref{eqn:second-order-PDE-c}, the reaction term $-\lambda K'(Ku-f)$ drives $u$ so 
that $Ku$ approaches $f$, playing the role of the fidelity term in \eqref{eqn:energy},
and the diffusion term $\operatorname{div} (c(u) \nabla u)$ smooths $u$ and endows it with some regularity 
(or makes it belong to certain function spaces), serving a role similar to the regularization 
term in \eqref{eqn:energy}. 
When $c$ is constant, \eqref{eqn:second-order-PDE-c} corresponds to the 
gradient descent flow of the Tikhonov regularization \eqref{eqn:heat-reaction}, leading to 
over-smoothing of the image and blurring of edges and other features. 
Choosing $c(u) = \frac{1}{|\nabla u|}$ yields \eqref{eqn:tv-reaction}, 
the subgradient flow of the ROF model, where diffusion is slower in regions with large 
gradients and faster in regions with small gradients. Such gradient-dependent diffusion 
coefficients $c(u) = b(|\nabla u|)$, which decrease monotonically with $|\nabla u|$, 
are also referred to as ``edge detectors.'' Typical choices 
include $b(|\nabla u|) = \frac{1}{1+k_1 |\nabla u|^2}$ and 
$b(|\nabla u|) = |\nabla u|^{p-2}$ ($1<p<2$), corresponding respectively to 
Perona-Malik diffusion \eqref{eqn:pm-reaction} and $p$-Laplacian-type diffusion. 
All these choices help slow down diffusion in regions with large gradients, 
preventing the smoothing effect from destroying edges in 
the image \cite{BARBU2019}. The above analysis also shows
that, besides choosing different regularization terms in \eqref{eqn:energy}, one can directly 
impose a priori information on the solution by designing an appropriate diffusion coefficient.} 
In \cite{ZHOU2015}, a doubly degenerate (DD) diffusion model was proposed as a novel image 
restoration framework, originally utilized for multiplicative noise removal. The 
framework was described as 
\begin{equation}  \label{eqn:DD}
    u_t = 
        \operatorname{div} \left(c\left(\left|\nabla u\right|,u\right) 
        \nabla u\right) - \lambda h(f, u).
\end{equation}
{Here} $c(\left|\nabla u\right|,u)=a(u)b(|\nabla u|)$, {where 
$a(u)$ is chosen to be monotonically increasing in $u$ and is referred to as 
a gray-level indicator, $b(|\nabla u|)$ is the edge detector described above, and} ${- \lambda} h(u,f)$ 
represents certain {reaction} terms derived from variational models. 
Numerous PDE-based image restoration models can be incorporated into this framework. 
The gray level indicator is used to control the diffusion {rate} at different 
gray levels and the edge detector 
is employed to preserve image edges; various options are available for selecting both. 
{Typical choices are 
\begin{equation} \label{eqn:DD-a-b} 
    a(u) = \left(\frac{|u|}{M}\right)^{\gamma},\quad 
    b(|\nabla u|) = \frac{1}{1 + k_1 |\nabla u|^{\beta}}, 
\end{equation}
where $M=\sup_{x\in\Omega}u(x)$ and $\gamma,k_1,\beta>0$. With the aid of the 
gray level indicator and the edge detector, the diffusion 
rate is reduced in regions with low gray level or large gradients, thereby lessening 
the smoothing effect that would otherwise destroy faint structures and sharp edges. 
Taking the reaction term $- \lambda h(f,u)$ in \eqref{eqn:DD} to be $-\lambda K'(Ku-f)$,
the DD model \eqref{eqn:DD} becomes applicable to the restoration of noisy blurred image   
discussed in this paper.

In summary, reaction-diffusion equations for image restoration can be derived either as the 
gradient flow of an energy functional or constructed directly by designing an appropriate 
diffusion coefficient. The latter approach may yield a model that no longer possesses a variational 
structure like \eqref{eqn:energy}, but it brings two advantages. Firstly, this approach offers 
greater flexibility in modeling: 
by carefully designing the diffusion coefficient one can precisely manipulate the diffusion 
rate at each pixel, allowing better preservation of desired image features and improved 
restoration performance. The model also remains interpretable from the reaction-diffusion 
viewpoint. Secondly, by modifying certain terms in the diffusion coefficient 
one can turn mathematically ill-posed equations derived from \eqref{eqn:energy} into well-posed ones; 
a typical example is the regularization of the Perona-Malik model \eqref{eqn:pm-reaction}. 
We will focus mainly on this approach in what follows.}

\subsection{{Nonlocal image restoration models}}

In natural images, there are often more complex structures such as textures and other 
repetitive features, which impose higher demands on image restoration models. Nonlocal 
operators have been defined in \cite{GILBOA2007,GILBOA2009} to extend the nonlocal method 
to the variational framework. The nonlocal total variation (NLTV) 
regularization \cite{GILBOA2009,ZHANG2010} remains one of the most popular deblurring 
methods for recovering texture to date. Other image restoration methods based on nonlocal 
functionals or PDEs \cite{KINDERMANN2005,SHI2021,WEN2023} 
have also achieved success in preserving textures. Fractional derivative-based variational 
and PDE models were proposed in the past two decades for image processing due to the nonlocal 
property of the fractional derivative. Bai {and Feng} \cite{BAI2007} proposed a 
fractional diffusion equation formulated as
\begin{equation}  \label{eqn:fractional-order-PDE}
    u_t = 
        \operatorname{div}^{\alpha} \big(g(|\nabla^{\alpha} u|^2) 
        \nabla^{\alpha} u\big), 
\end{equation}
where $1<\alpha<2$, $\nabla^{\alpha}$ denotes the fractional gradient operator 
and $-\operatorname{div}^{\alpha}$ denotes the adjoint operator of $\nabla^{\alpha}${, and  
the edge detector is chosen as $g(s)=\frac{1}{1+s}$ as in the Perona-Malik model.
\eqref{eqn:fractional-order-PDE} can be regarded as the gradient flow of the energy 
involving a fractional-order gradient:
\begin{equation} \label{eqn:energy-fractional}
E(u)=\frac{1}{2} \int_{\Omega}G(|\nabla^\alpha u|^2) dx
\end{equation}
with $g(s) = \frac{d}{ds} G(s)$}. 
This model was proposed for image denoising and can be viewed as generalizations 
of second-order and fourth-order anisotropic diffusion equations. It has been demonstrated 
to effectively eliminate the staircase effect. Zhang et al. 
\cite{ZHANG2015} established a comprehensive total $\alpha$-order variation 
framework for image restoration, where $\alpha$ can take any positive value. 
This framework along with other fractional derivative-based models 
\cite{CHAN2013,YIN2016} demonstrates the texture preservation 
capability in image restoration.

Inspired by the model \eqref{eqn:fractional-order-PDE} of Bai and Feng, 
Yao et al. generalized the DD model \eqref{eqn:DD} to fractional order, initially applied it for 
multiplicative noise removal \cite{YAO2019}, and later extended its application 
to deblurring tasks \cite{GUO2019}. The model takes the form
\begin{equation}  \label{eqn:fractional-DD}
    u_t = 
        \operatorname{div}^{\alpha} \left(c\left(\left|\nabla^{\alpha} u\right|,u\right) 
        \nabla^{\alpha} u\right) - \lambda K'(Ku-f), 
\end{equation}
where $0<\alpha<2$, $c(\left|\nabla^{\alpha} u\right|,u)=a(u)b(|\nabla^{\alpha} u|)$. The 
gray level indicator is chosen as
\begin{equation} \label{eqn:fractional-DD-a}
    a(u) = \left(\frac{|u|}{M}\right)^{\gamma},
\end{equation}
where $M = \sup_{x \in \Omega} u(x)$, $\gamma>0$. On 
the other hand,
$b(|\nabla^{\alpha} u|) = \frac{1}{1 + k_1 \left|\nabla^{\alpha} u\right|^{\beta}}$
can be seen as a texture detection function, where $k_1 > 0$, $\beta > 0$. 
The design of the diffusion coefficient in this model emphasizes the preservation 
of low gray level features and texture information. Experimental results have 
demonstrated the effectiveness of the model in restoring texture-rich images.
However, the model \eqref{eqn:fractional-DD} still faces two challenges: Firstly, 
since the observed image is always noisy or noisy blurred, the gray level 
indicator \eqref{eqn:fractional-DD-a} may result in 
a large difference in diffusion {rates} at neighboring pixels in homogeneous regions, 
thereby reducing the visual quality of the restored image. Secondly, {the model 
no longer complies with the fundamental diffusion 
laws \eqref{eqn:mass-conservation} and \eqref{eqn:fick-law}; 
in particular,} the presence of the fractional-order gradient {violates mass conservation}, 
which complicates its physical interpretation and increases computational costs. 
In \cite{SHAN2022}, a model based on \eqref{eqn:mass-conservation} and 
fractional Fick's law \cite{SCHUMER2001}
\begin{equation} \label{eqn:fractional-fick-law}
    {\vec{J} = -c \nabla^{\alpha} u}
\end{equation}
was proposed for multiplicative noise removal, 
attempting to address the challenges mentioned above. The model takes on the form
\begin{equation} \label{eqn:Shan}
    u_t = 
        \operatorname{div} \left(c\left(\left|\nabla^{\alpha} u\right|,u_{\sigma}\right) 
        \nabla^{\alpha} u\right),
\end{equation}
where $0<\alpha \leq 3$, 
$c(\left|\nabla u\right|,u_{\sigma})=a(u_{\sigma})b(|\nabla^{\alpha} u|)$.
The authors chose the gray level indicator as 
\begin{equation} \label{eqn:Shan-a}
    a(u_{\sigma}) = \left(\frac{|u_{\sigma}|}{M_{\sigma}}\right)^{\gamma},
\end{equation}
where $u_{\sigma} = G_{\sigma} * u$, $G_{\sigma}$ is a Gaussian convolution
kernel, $M_{\sigma} = \sup_{x \in \Omega} u_{\sigma}(x)$, 
$\sigma, \gamma>0$. Gaussian filters can reduce the interference of noise on 
image features, enabling the gray level indicator to perform more effectively 
in controlling the diffusion {rate}. Similar operations have also been observed 
in previous works \cite{MAJEE2020,SHAN2019,ZHOU2018}.
The selection of the texture detection function is similar to that 
in \eqref{eqn:fractional-DD}. This model possesses has a more comprehensive 
physical background and experimental results have demonstrated that it achieves 
superior image restoration performance and fast computational speeds.

\subsection{The motivation of the paper}
The motivation of this paper is to propose an anisotropic diffusion model from a clear 
physical background, with the aim of preserving low gray level image features and texture 
details in denoising and deblurring tasks on both texture-rich images 
and {detailed satellite} images. 
A simple idea is to consider the equation 
\begin{equation} \label{eqn:modified-Shan}
    u_t =
\operatorname{div} \left(c\left(\left|\nabla^{\alpha} u\right|,u_{\sigma}\right)
\nabla u\right) - \lambda K'(Ku-f),
\end{equation}
and select the same gray level indicator as in \eqref{eqn:Shan}. However, this is not 
suitable for restoring noisy blurred images because the Gaussian filter makes the 
already blurred image \emph{over smooth} ($u_{\sigma}$ belongs to the $C^{\infty}$ class), 
which would result in a rough gray level indication and decrease the 
visual quality of the restored image. This requires us to adopt a ``gentler" approach 
to regularize $u$ in \eqref{eqn:fractional-DD-a}. The gray level indicator 
in \eqref{eqn:fractional-DD-a} was initially inspired by the gamma 
correction \cite{ZHOU2015,POYNTON2012}, 
but we now provide another interpretation. In DD model \eqref{eqn:DD}, the gray level indicator 
ensures that the diffusion is slow when $u$ is small and fast when $u$ is large, 
which is the characteristic of porous medium equations (PMEs), also known as Newtonian 
filtration slow diffusion \cite{VAZQUEZ2006}. In fact, if we 
set $a(u)=\frac{|u|^{\gamma}}{M^{\gamma}}$, $b(\cdot) \equiv 1$, 
$h(f,u)=0$ in the DD model \eqref{eqn:DD}, and taking $M$ as a positive constant, 
such as $M = \sup_{x \in \Omega} f(x)$, \eqref{eqn:DD} becomes 
a PME:
\begin{equation} \label{eqn:DD-PME}
    u_t 
    = \frac{M^{-\gamma}}{\gamma+1} \operatorname{div} \left(|u|^{\gamma}\nabla u\right) 
    = \frac{M^{-\gamma}}{\gamma+1} \Delta |u|^{\gamma+1}.
\end{equation}
This type of equation appears in different contexts, such as 
in the flow of gases through porous media \cite{BARENBLATT1990,MUSKAT1937}, 
in the heat conduction at high temperatures \cite{ZEL2002} and in groundwater flow 
\cite{BOUSSINESQ1904}. For the case where $b(\cdot)$ is chosen as a general edge 
detection operator, the equation 
\[u_t =
\operatorname{div} \left(\frac{|u|^{\gamma}}{M^{\gamma}} b(\left|\nabla u\right|) 
\nabla u\right)\] 
can be viewed as an \emph{anisotropic} PME. 

Inspired by the analysis above, we consider introducing a variable $v$ into the gray value 
indicator, designing it to satisfy a type of slow diffusion equation
and ensuring it has the same initial datum as $u$, 
so that it gradually becomes more regular during the diffusion process, rather than always 
remaining smooth like $u_{\sigma}$ in \eqref{eqn:Shan-a}. The diffusion equation for 
denoising and deblurring is written as
\begin{equation} \label{eqn:Ours-1} 
u_t =
\operatorname{div} \left(\frac{|v|^{\gamma}}{M^{\gamma}} \frac{1}{1 + k_1 \left|\nabla^{\alpha} u\right|^{\beta}}
\nabla u\right) - \lambda K'(Ku-f).
\end{equation}
To emphasize mutual transfer of information between the two equations, we design the 
diffusion equation for $v$ as
\begin{equation} \label{eqn:Ours-2} 
    {v}_t = \lambda_1 \operatorname{div}
    \left(\frac{|u|^{\mu}}{M^{\mu}} \nabla v\right) + (1-\lambda_1) \operatorname{div}
    \left(\frac{|v|^{\gamma}}{M^{\gamma}} \nabla v\right),
\end{equation}
which exhibits characteristics of a PME, where $0<\lambda_1<1$ and $\mu>0$, 
the diffusion {rate} is simultaneously controlled by both $u$ and $v$. 
{From the viewpoint of PDE evolution, although a rigorous 
characterization of the asymptotic behavior as $t \to \infty$ would require further 
mathematical analysis, it can be qualitatively stated that under} the coupled 
reaction-diffusion system formed by \eqref{eqn:Ours-1} and \eqref{eqn:Ours-2}, 
{$u$ gradually approaches a clearer restored image, 
while $v$ tends to a smoother function that partially reflects the structural information of $u$. 
Specifically, on the right-hand side of \eqref{eqn:Ours-2}, the slow diffusion term acting 
on $v$ itself ensures that $v$ becomes smoother over time, while
the coupled diffusion term whose coefficient depends on the current image $u$, 
locally modulates the diffusion rate of $v$ in a $u$-weighted manner. Moreover, the regularity 
of $v$ does not become $C^\infty$ immediately as in the heat 
equation (for general initial data one typically obtains at most some H\"{o}lder regularity \cite{VAZQUEZ2006}). 
In this way, $v$ receives the desired ``gentler'' regularization while gradually absorbing structural 
information from the gradually clearer $u$, which in turn provides more reliable gray-level 
indication for the restoration of $u$, facilitating} positive mutual transfer of information within the system. 
Besides, following the recommendation 
in \cite{WELK2005}, we adopt periodic boundary conditions for this system. {Henceforth, 
for an $N$-dimensional open cube $\Omega$ with side length $L>0$, we say that a measurable
function $u:\Omega \to \mathbb{R}$ is periodic on $\Omega$ if admits an
$L$-periodic extension in each coordinate direction. That is, 
there exists a measurable function $\tilde{u}:\mathbb{R}^N \to \mathbb{R}$ such 
that $\tilde{u}|_{\Omega} = u$ a.e. in $\Omega$ and 
\[\tilde{u}(x+ Le_j) = \tilde{u}(x), \quad \textrm{for~a.e.~} x \in \mathbb{R}^N, \ j=1, \dots, N,\] 
where $e_j$ denotes the $j$-th standard basis vector of $\mathbb{R}^N$.}\\[1ex]
\textbf{The proposed model.} Taking everything into account, we propose the following 
reaction-diffusion system as an image restoration model for some fixed positive $T_0$:
\begin{equation}
    \left\{
    \begin{array}{ll}
        \displaystyle \vspace{0.2em} 
        u_t = 
        \operatorname{div} \left(c\left(\left|\nabla^{\alpha} u\right|,v\right) 
        \nabla u\right) - \lambda K'(Ku-f), 
        & \textrm{in~~} (0,T_0) \times \Omega,  \\ 
        v_t = 
        \operatorname{div} \left( a(u,v) \nabla v\right)
        , & \textrm{in~~} (0,T_0) \times \Omega,   \\
        u, v \textrm{~~periodic on} \Omega & \textrm{for~~} (0,T_0), \\
        u(0,\cdot) = v(0,\cdot) = f, & \textrm{in~~} \Omega, 
    \end{array}
    \right. \label{eqn:main}
\end{equation}
where 
\[\begin{array}{c}
    c\left(\left|\nabla^{\alpha} u\right|,v\right) 
    = \dfrac{|v|^{\gamma}}{M^{\gamma}} b(\left|\nabla^{\alpha} u\right|), \quad 
    b(\left|\nabla^{\alpha} u\right|) = 
\dfrac{1}{1 + k_1 \left|\nabla^{\alpha} u\right|^{\beta}}, \\ \vspace{0.2em}
a\left(u,v\right) = 
\lambda_1 \dfrac{|u|^{\mu}}{M^{\mu}} + 
\left(1-\lambda_1\right) \dfrac{|v|^{\gamma}}{M^{\gamma}},    
\end{array}\]
$\Omega \subset \mathbb{R}^{N}$ is an $N$-dimensional open cube 
for $N \geq 1$, $K \in \mathcal{L}\left(L^1(\Omega),L^2(\Omega)\right)$, 
$M = \sup_{x \in \Omega} f(x)$, and $0<\alpha<1$, $\beta>0$, 
$\gamma>0$, $\mu>0$, $\lambda>0$, $0<\lambda_1<1$, $k_1>0$ are some given 
constants. The proposed model maintains the characteristics similar to those of 
existing models: $v$ can be regarded as a regularized version of the observed image, 
when $v$ is small, the diffusion at low gray value regions will be slow. The 
diffusion {rate} is also controlled by the texture detection 
function $b(\left|\nabla^{\alpha} u\right|)$,
which becomes small when $\left|\nabla^{\alpha} u\right|$ is large, leading to 
the protection of the textures. Our model also possesses the following advantages: 
the gentler regularization of $u$ in the gray level indicator aligns the model more 
closely with the task of recovering noisy blurred images; the model is based on a 
more classical diffusion framework, providing a \emph{clearer} physical background. 

The main contribution of this paper lies in the theoretical analysis of the proposed model.
Due to the low regularity of the diffusion coefficient in system \eqref{eqn:main},
it is challenging to establish its well-posedness using classical fixed-point methods 
as done in \cite{MAJEE2020,SHAN2019,ZHOU2018}. Therefore, our work also serves as a 
\emph{theoretical complement} to equations of the type in \cite{ZHOU2015,YAO2019,GUO2019}. 
In this paper, we employ Maximal Regularity Theory developed by 
Amann et al. \cite{AMANN1995,AMANN2004,AMANN2005,KUNSTMANN2004} to 
investigate the well-posedness of \eqref{eqn:main}. The regularity and other properties 
of solutions are also concerned. Additionally, we present a semi-implicit finite 
difference scheme for the proposed model and validate its effectiveness in image 
restoration tasks using both texture and satellite images.\\[1ex]
\textbf{Organization of the paper.} The rest of this paper is organized as follows. 
Some mathematical preliminaries and our 
main result are stated in Sect. \ref{sec2}. The main result and 
some properties of weak solutions are proven in Sect. \ref{sec3}. Regularity results 
for the proposed model are provided in Sect. \ref{sec4}. In Sect. \ref{sec5},
some numerical examples are presented to demonstrate the effectiveness of our model. 
{The concluding remarks are drawn in Sect. \ref{sec6}.}

\section{Mathematical Preliminaries and main result}\label{sec2}

In this section, we state some necessary preliminaries of the fractional gradient, 
fractional order Sobolev spaces and maximal $L^p$-regularity which will be used 
below. {For simplicity of discussion, throughout the theoretical analysis in this paper, we 
take $\Omega$ to be the normalized unit open cube $(0,1)^N$.} 
In the following we will make essential use of two types of periodic Sobolev spaces 
of fractional order. We refer to \cite{TRIEBEL1978,LEONI2023} as well as 
\cite[Section 5]{AMANN1993} and the references therein for more basic results of 
general fractional order Sobolev spaces. Let $U$ be an open set in $\mathbb{R}^N$ 
and $p \in [1,\infty)$. Denote by 
\[[f]_{W^{s,p}(U)} = \left(\int_{U} \int_{U} 
\frac{|f(y)-f(x)|^p}{|x-y|^{N+sp}} dxdy\right)^{\frac{1}{p}}\]
the Gagliardo seminorm of a measurable function $f$ in $U$ for $s \in (0,1)$.
Then for $s \in (0,\infty) \setminus \mathbb{N}$, the periodic 
Slobodeckij spaces are the Banach spaces defined by 
\[W_{\pi}^{s,p}(\Omega) = \left\{u \in W^{\lfloor s\rfloor,p}(\Omega): 
u \textrm{~is periodic on~} \Omega,\ {
\left[D^{\alpha} u\right]_{W^{s-\lfloor s\rfloor,p}(\Omega)} < \infty, 
\forall |\alpha|=\lfloor s\rfloor} \right\}\]
equipped with the norm 
\[\|u\|_{W_{\pi}^{s,p}(\Omega)} = \|u\|_{L_{\pi}^{p}(\Omega)} + 
\sum_{|\alpha|=\lfloor s\rfloor} 
\left[D^{\alpha} u\right]_{W^{s-\lfloor s\rfloor,p}(\Omega)},\] 
{where $\lfloor s\rfloor$ denotes the greatest integer less than or equal to $s$, and} 
the subscript $\pi$ indicates that we are in a periodic function space.
{We will later use Morrey's embedding theorem: for $p \in [1,\infty)$ and 
$s \in (0,\infty) \setminus \mathbb{N}$ such that $(s-\lfloor s\rfloor)p > N$, 
\[W_{\pi}^{s,p}(\Omega) \hookrightarrow C_{\pi}^{\lfloor s\rfloor,s-\lfloor s\rfloor-\frac{N}{p}}(\overline{\Omega}).\]
Readers are referred to \cite{LEONI2023} for details. 

Given $u \in L_{\pi}^1(\Omega)$, its Fourier coefficients are} 
\[{\widehat{u}}(k) = \int_{\Omega} u(x) e^{-2 \pi \mathbbm{i} k \cdot x} dx\]
{for each $k \in \mathbb{Z}^N$.} 
For $p \in (1,\infty)$ and $s \in (0,\infty)$, {we put} 
\[{\|u\|_{H_{\pi}^{s,p}(\Omega)} = \| (I - \Delta)^{\frac{s}{2}} u\|_{L_{\pi}^p(\Omega)},}\]
where 
\[{(I - \Delta)^{\frac{s}{2}} u(x) = \sum_{k \in \mathbb{Z}^N} 
\left(1+4 \pi^2 |k|^2\right)^{\frac{s}{2}} \widehat{u}(k) e^{2 \pi \mathbbm{i} k \cdot x}.}\]
Then the Bessel potential spaces are the Banach spaces defined by
\[H_{\pi}^{s,p}(\Omega) = \left\{u \in L_{\pi}^p(\Omega): 
\|u\|_{H_{\pi}^{s,p}(\Omega)} < \infty\right\}.\]
{It is known from the almost reiteration property \cite[Theorem V.1.5.3]{AMANN1995} that 
Slobodeckij spaces can be embedded into Bessel potential 
spaces (see also \cite[Theorems 1.3.3(e) and 1.10.3.2]{TRIEBEL1978}). 
Specifically, for any $\delta>0$ and $p\in(1,\infty)$, the embedding 
\begin{equation} \label{eqn:slobodeckij-to-bessel}
  W_{\pi}^{s,p}(\Omega) \hookrightarrow H_{\pi}^{s-\delta,p}(\Omega)  
\end{equation}
holds for all $s \in (\delta,\infty) \setminus \mathbb{N}$.
Furthermore, once $s-\frac{N}{p} \in (0,\infty) \setminus \mathbb{N}$, 
\begin{equation} \label{eqn:bessel-to-holder}
   H_{\pi}^{s,p}(\Omega) \hookrightarrow C_{\pi}^{\kappa,\varrho}(\overline{\Omega}) 
\end{equation}
follows from classical embedding 
theorem \cite[Theorems 2.8.1(e) and 4.6.1(e)]{TRIEBEL1978} and Rychkov's extension 
theorem \cite[Section 5.1.4]{SAWANO2018} (see also \cite[Theorem 4.1]{RYCHKOV1999}), 
where $\kappa$ and $\varrho$ denote 
the integer part and fractional part of $s-\frac{N}{p}$, respectively.}

{For problem \eqref{eqn:main} to make sense, 
the fractional gradient $\nabla^{\alpha}$ of $u$ needs to be defined. 
Although certain alternative definitions would work as well, in this paper, we use 
the frequency domain definition of the fractional gradient, consistent with those 
in \cite{BAI2007,GUIDOTTI2009,GUIDOTTI2011,YIN2016}.} For $\alpha \in (0,1)$, the 
fractional gradient {\[\nabla^{\alpha} =
(\partial_{x_1}^{\alpha}, \dots, \partial_{x_N}^{\alpha})\]}
of $u$ {can be represented by the Fourier series}
\[{\partial_{x_j}^{\alpha} u(x) = \sum_{k \in \mathbb{Z}^N} 
(2 \pi \mathbbm{i} k_j)^{\alpha} \widehat{u}(k) e^{2 \pi \mathbbm{i} k \cdot x}.}\]
Here, $\partial_{x_j}^{\alpha} u$ denotes the $j$-th component of $\nabla^{\alpha} u$, 
referred to as the partial derivative with order $\alpha$ of $u$ with respect to $x_j${, 
and $(2 \pi \mathbbm{i} k_j)^{\alpha} = (2 \pi |k_j|)^{\alpha} e^{\mathbbm{i} \alpha \frac{\pi}{2} \operatorname{sgn}(k_j)}$. 
The fractional gradient defined in this way is generally not rotation-invariant,
in contrast to some other possible definitions such as the Riesz fractional gradient (see \cite{SILHAVY2020}). 
Nevertheless, this definition has the advantage of allowing a straightforward 
characterization of the mapping properties of fractional derivatives between Bessel 
potential spaces. Moreover, it is widely employed in image processing because it is easy to
implement using fast Fourier transform \cite{BAI2007,GUIDOTTI2009}.}

We will apply Maximal Regularity Theory to the study of problem \eqref{eqn:main}. For 
more details about this theory, see \cite{AMANN1995,AMANN2004,AMANN2005,KUNSTMANN2004} and 
the references therein.
Let $E_0$ and $E_1$ be Banach spaces such that $E_1 \xhookrightarrow[]{d} E_0$, 
i.e. $E_1$ is densely embedded in $E_0$. Suppose that $1<p<\infty$. 
Set $E_{1-\frac{1}{p}} = (E_1,E_0)_{1-\frac{1}{p},p}$, 
$(\cdot,\cdot)_{\theta,p}$ denoting the standard real interpolation functor.
Given $T_0 \in (0,\infty)$, for any $T \in (0,T_0]$, put 
\[\mathcal{W}^{1,p}(0,T;(E_1,E_0)) = L^p(0,T;E_1) \cap W^{1,p}(0,T;E_0).\]
An operator ${A^{\sharp}} \in L^{\infty}(0,T;\mathcal{L}(E_1,E_0))$ is said to possess the property 
of maximal $L^p$-regularity on $[0,T)$ with respect to $(E_1,E_0)$ if the map 
\[\mathcal{W}^{1,p}(0,T;(E_1,E_0)) \to L^p(0,T;E_0) \times E_{1-\frac{1}{p}}, \quad
    u \mapsto \left(\dfrac{du}{dt} + {A^{\sharp}}u, u(0)\right)\]
is a bounded isomorphism. The set of all such operators ${A^{\sharp}}$ is denoted 
by $\mathcal{MR}^{p}(0,T;(E_1,E_0))$. Denote by $\mathcal{MR}^{p}(E_1,E_0)$ the set of
all ${A^{\flat}} \in \mathcal{L}(E_1,E_0)$ such that the map 
$[0,T) \to \mathcal{L}(E_1,E_0), t \mapsto {A^{\flat}}$
belongs to $\mathcal{MR}^{p}(0,T;(E_1,E_0))$. Let $X$ and $Y$ be metric spaces
{of functions defined on $[0,T)$. A function $f: X \to Y$ is called a Volterra map 
(or to possess the Volterra property) if}, for each $T \in (0,T_0]$ and each $u,v \in X$ with $u|_{[0,T)} = v|_{[0,T)}$, 
it follows that $f(u)|_{[0,T)} = f(v)|_{[0,T)}$.
Denote by {$\mathcal{C}^{0,1}(X,Y)$} the space of all maps
$f:X \to Y$ which are bounded on bounded sets and uniformly Lipschitz continuous on 
such sets. {We write $\mathcal{C}_{\mathrm{Volt}}^{0,1}(X,Y)$ for the subset of all Volterra
maps in $\mathcal{C}^{0,1}(X,Y)$.}
Let $Y_1$ and $Y_0$ be Banach spaces {of functions defined on $[0,T)$} such that $Y_1 \hookrightarrow Y_0$. Denote by 
$\mathcal{C}_{\mathrm{Volt}}^{0,1}(X;Y_1,Y_0)$ the set of all {Volterra} maps $g:X \to Y_0$ such 
that $g-g(0) \in \mathcal{C}^{0,1}(X,Y_1)$. {Note that the Volterra property is automatic 
for mappings that are local in time \cite{AMANN2005}. Indeed, this property is 
relevant only for problems in which nonlocalities in time are present \cite{GUIDOTTI2010}.}

Consider a quasilinear abstract Cauchy problem
\begin{equation} \label{eqn:amann-quasilinear}
    \left\{
        \begin{array}{l} 
            \dfrac{d u}{d t} + A(u)u = F(u), \quad t \in (0,T_0), \vspace{0.2em} \\ 
            u(0) = u_0. 
    \end{array} \right.
\end{equation}
A function $u \in \mathcal{W}_{\mathrm{loc}}^{1,p}(0,T;(E_1,E_0))$ 
is said to be a solution to \eqref{eqn:amann-quasilinear} on $[0,T)$ 
if it satisfies \eqref{eqn:amann-quasilinear} in the a.e. sense on $(0,T)$. A solution 
is said to be maximal if it cannot be extended to a solution on a strictly larger interval. 

The following existence and uniqueness result is applied to the study 
of problem \eqref{eqn:main}.
\begin{theorem}(See \cite[Theorem 2.1]{AMANN2005}) \label{thm:amann-main}
    Suppose that 
    \begin{enumerate}[(i)]
        \item $A \in \mathcal{C}_{\mathrm{Volt}}^{0,1}\left(\mathcal{W}^{1,p}(0,T;(E_1,E_0)),
        \mathcal{MR}^p(0,T;(E_1,E_0))\right)$;
        \item there exists $q \in (p,\infty]$ such that 
        \[F \in \mathcal{C}_{\mathrm{Volt}}^{0,1}\left(\mathcal{W}^{1,p}(0,T;(E_1,E_0)), 
        L^q(0,T;E_0), L^p(0,T;E_0)\right);\]
        \item $u_0 \in E_{1-\frac{1}{p}}$.
    \end{enumerate}
    Then there exist a maximal $T_{\max} \in (0,T_0]$ and a unique solution $u$ of 
    \eqref{eqn:amann-quasilinear} on $[0, T_{\max})$.
\end{theorem}

{Henceforth, we always suppose that $0 < T \leq T_0$.} Reformulating problem \eqref{eqn:main} as a weak $L^p$-formulation, the weak solution 
of \eqref{eqn:main} can be defined as follows.

\begin{definition} \label{def:main-weak}
    {Let $p>2$.} Given $f \in W_{\pi}^{1-\frac{2}{p},p}(\Omega)$, a couple of 
    functions $(u,v)$ is called a weak solution to problem \eqref{eqn:main}  
    on $[0,T)$, if it satisfies the following conditions:
    \begin{enumerate}[(i)]
        \item $u, v \in L_{\mathrm{loc}}^p\left(0,T;W_{\pi}^{1,p}(\Omega)\right)   
        \cap W_{\mathrm{loc}}^{1,p}\left(0,T;W_{\pi}^{-1,p}(\Omega)\right)$;
        \item $u(0,\cdot)=v(0,\cdot)=f$; 
        \item for any $\varphi, \psi \in W_{\pi}^{1,p'}(\Omega)$, 
        the following integral equalities hold:
        \begin{equation} \label{eqn:main-weak}
        \begin{array}{c} 
        \displaystyle \int_{\Omega} u_t \varphi d x + 
        \int_{\Omega} c\left(\left|\nabla^{\alpha} u\right|,v\right)
        \nabla u \cdot \nabla \varphi d x 
        + \lambda \int_{\Omega} (Ku-f) K \varphi d x = 0, \\[1em]
        \displaystyle \int_{\Omega} v_t \psi d x 
        + \int_{\Omega} a\left(u, v\right) \nabla v \cdot \nabla \psi d x = 0,
        \end{array}
        \end{equation}
        for almost all $t \in (0,T)$. {Here $\frac{1}{p} + \frac{1}{p'} =1$.}
    \end{enumerate}
\end{definition}

{We will reduce the study of weak solutions to \eqref{eqn:main} to that of solutions to 
a quasilinear abstract Cauchy problem of the form \eqref{eqn:amann-quasilinear}.} The main result is stated as follows.

\begin{theorem} \label{thm:well-posedness-weak}
Assume that $p \in \left(\frac{2+N}{1-\alpha}, \infty\right)$ and $\beta \geq 1$. 
Let $f \in W_{\pi}^{1-\frac{2}{p},p}(\Omega)$ {with $f \geq 0$ and 
bounded away from zero a.e. in $\Omega$.} Then there exist a unique 
maximal $T_{\max} \in (0,T_0]$ and a unique weak solution $(u, v)$ of
problem \eqref{eqn:main} on $[0, T_{\max})$.
\end{theorem}

{The proof of Theorem \ref{thm:well-posedness-weak} is carried out within the framework of} Maximal Regularity Theory, 
incorporating the techniques {developed in} \cite{GUIDOTTI2009, GUIDOTTI2010, GUIDOTTI2011}, 
{with details} given in the next section.

\section{Local well-posedness and some properties of weak solutions}\label{sec3}

In this section, we first establish the local well-posedness of 
problem \eqref{eqn:main} via maximum regularity.
For all $\varepsilon >0$, let
$\overline{v}_{\varepsilon} = \max \{v, \varepsilon\}$. We consider the following 
auxiliary problem:
\begin{equation}
    \left\{
    \begin{array}{ll}
        \displaystyle \vspace{0.2em} 
        u_t = 
        \operatorname{div} \left(
        c\left(\left|\nabla^{\alpha} u\right|, \overline{v}_{\varepsilon}\right) \nabla u\right) 
        - \lambda K'(Ku-f), & \textrm{in~~} (0,T_0) \times \Omega,  \\ 
        v_t = 
        \operatorname{div} \left(a\left(u, \overline{v}_{\varepsilon}\right)
         \nabla v\right), & \textrm{in~~} (0,T_0) \times \Omega,   \\
        u, v \textrm{~~periodic on} \Omega & \textrm{for~~} (0,T_0), \\
            u(0,\cdot) = v(0,\cdot) = f, & \textrm{in~~} \Omega. 
    \end{array}
    \right. \label{eqn:auxiliary}
\end{equation}

Similar to Definition \ref{def:main-weak}, we define the weak solution to 
the auxiliary problem \eqref{eqn:auxiliary}.
\begin{definition}
    {Let $p>2$.} Given $f \in W_{\pi}^{1-\frac{2}{p},p}(\Omega)$, a couple of 
    functions $(u,v)$ is called a weak solution to
    auxiliary problem \eqref{eqn:auxiliary}  
    on $[0,T)$, if it satisfies the following conditions:
    \begin{enumerate}[(i)]
        \item $u, v \in L_{\mathrm{loc}}^p\left(0,T;W_{\pi}^{1,p}(\Omega)\right) \cap 
        W_{\mathrm{loc}}^{1,p}\left(0,T;W_{\pi}^{-1,p}(\Omega)\right)$;
        \item $u(0,\cdot)=v(0,\cdot)=f$; 
        \item for any $\varphi, \psi \in W_{\pi}^{1,p'}(\Omega)$, 
        the following integral equalities hold:
        \begin{subequations}
        \begin{gather}
        \displaystyle \int_{\Omega} u_t \varphi d x + 
        \int_{\Omega} c\left(\left|\nabla^{\alpha} u\right|, \overline{v}_{\varepsilon}\right)
        \nabla u \cdot \nabla \varphi d x + \lambda \int_{\Omega} (Ku-f) K \varphi d x = 0
        , \label{eqn:auxiliary-weak-1} \\
        \displaystyle \int_{\Omega} v_t \psi d x 
        + \int_{\Omega} a\left(u, \overline{v}_{\varepsilon}\right)
        \nabla v \cdot \nabla \psi d x = 0, \label{eqn:auxiliary-weak-2}
        \end{gather}
        \end{subequations}
        for almost all $t \in (0,T)$.
    \end{enumerate}
\end{definition}

{
\begin{remark} 
If $(u,v)$ is a weak solution to \eqref{eqn:auxiliary} on $[0,T)$, 
it then follows that $v$ satisfies the minimum principle 
\begin{equation} \label{eqn:auxiliary-minimum}
v(t,x) \geq \underset{x \in \Omega}{\operatorname*{ess} \inf}~f(x) \eqqcolon \gamma_{f} 
\end{equation}
holds for every $t\in(0,T)$ and a.e. $x \in \Omega$, as long as $\gamma_f>-\infty$. 
\eqref{eqn:auxiliary-minimum} can be immediately proved by Stampacchia's truncation 
technique. Namely, one may take $\psi=-(v-\gamma_f)_-$ in \eqref{eqn:auxiliary-weak-2}, 
where $r_- \coloneqq \max\{-r,0\}$ for $s \in \mathbb{R}$. 
Indeed, for a.e. $t \in (0,T)$, $-(v(t,\cdot)-\gamma_f)_-$ belongs to 
$W_{\pi}^{1,p}(\Omega) \hookrightarrow W_{\pi}^{1,p'}(\Omega)$ (note that $p > 2$), making $\psi = -(v-\gamma_f)_-$ an 
admissible test function. The proof is standard and is therefore omitted. \label{rmk:auxiliary}
\end{remark}

Hereafter, we will abbreviate $\operatorname*{ess}\inf$ (resp., $\operatorname*{ess}\sup$) 
as $\inf$ (resp., $\sup$). Let $B_X$ denote the unit ball in the Banach space $X$, and 
write $B_X(x,R)\coloneqq x+RB_X$ for the ball centered at $x$ with radius $R>0$. 
We now present the proof of Theorem \ref{thm:well-posedness-weak}.}

\begin{proof}[\bf Proof of Theorem \ref{thm:well-posedness-weak}]
Firstly, we reformulate problem \eqref{eqn:auxiliary} as a quasilinear abstract 
Cauchy problem. {Since $p \in \left(\frac{2+N}{1-\alpha}, \infty\right) \subset \left(2+N, \infty\right)$, 
it follows from Morrey's embedding theorem together with the assumptions 
on $f$ that} 
{\[f\in C_{\pi}^{0, 1-\frac{2+N}{p}}(\overline{\Omega}) \quad 
\textrm{and} \quad \inf_{x \in \Omega} f(x) > 0.\]} In this section, we always choose
\[E_0 = W_{\pi}^{-1,p}(\Omega)^2, \quad E_1 = W_{\pi}^{1,p}(\Omega)^2.\]
For this choice it is well-known that $E_1 \xhookrightarrow[]{d} E_0$. We denote by 
\[E_{1-\frac{1}{p}} = \left(E_0, E_1\right)_{1-\frac{1}{p},p} 
= W_{\pi}^{1-\frac{2}{p},p}(\Omega)^2\]
the respective real interpolation space. To simplify the notation, we set
\[E_0^{1/2} = W_{\pi}^{-1,p}(\Omega), \quad E_1^{1/2} = W_{\pi}^{1,p}(\Omega), \quad
E_{1-\frac{1}{p}}^{1/2} = W_{\pi}^{1-\frac{2}{p},p}(\Omega).\]
For fixed $w_0=(u_0,v_0)^{\mathrm{T}} \in \mathcal{W}^{1,p}(0,T;(E_1,E_0))$, let a bilinear 
form $a_{11}[w_0](\cdot,\cdot)$ be given by
\[a_{11}[w_0](u,\varphi) \coloneqq \int_{\Omega} 
c\left(\left|\nabla^{\alpha} u_0\right|, \overline{v_0}_{\varepsilon}\right) 
\nabla u \cdot \nabla \varphi d x, \quad u \in W_{\pi}^{1,p}(\Omega),
\varphi \in W_{\pi}^{1,p'}(\Omega).\]
Then a linear differential operator $A_{11}(w_0)$ is
naturally induced by $a_{11}[w_0]$, i.e. 
for almost all $t \in (0,T)$,
\[\left \langle A_{11}(w_0)u, \varphi \right \rangle 
_{\left(W_{\pi}^{-1,p}(\Omega),W_{\pi}^{1,p'}(\Omega)\right)} = a_{11}[w_0](u,\varphi)
, \quad \forall \varphi \in W_{\pi}^{1,p'}(\Omega),.\]
Actually, we have already defined an operator
\[A_{11} : \mathcal{W}^{1,p}(0,T;(E_1,E_0)) 
\to L^{\infty} \left(0,T;\mathcal{L}\left(E_1^{1/2},E_0^{1/2}\right)\right), 
\quad w \mapsto A_{11}(w).\]
In the same way, for fixed 
$w_0=(u_0,v_0)^{\mathrm{T}} \in \mathcal{W}^{1,p}(0,T;(E_1,E_0))$, let 
$a_{22}[w_0](\cdot,\cdot)$ be given by
\[a_{22}[w_0](v,\psi) \coloneqq \int_{\Omega} 
a\left(u, \overline{v}_{\varepsilon}\right)
\nabla v \cdot \nabla \psi d x, \quad v \in W_{\pi}^{1,p}(\Omega),
\psi \in W_{\pi}^{1,p'}(\Omega),\]
then we can define $A_{22}(w_0)$ by
\[\left \langle A_{22}(w_0)v, \psi \right \rangle 
_{\left(W_{\pi}^{-1,p}(\Omega),W_{\pi}^{1,p'}(\Omega)\right)}
= a_{22}[w_0](v,\psi), \quad 
\forall \psi \in W_{\pi}^{1,p'}(\Omega)\] 
for almost all $t \in (0,T)$, and introduce another operator
\[A_{22} : \mathcal{W}^{1,p}(0,T;(E_1,E_0)) 
\to L^{\infty} \left(0,T;\mathcal{L}\left(E_1^{1/2},E_0^{1/2}\right)\right), 
\quad w \mapsto A_{22}(w).\]
Now we define the operator matrix $A$ given by
\[A: \mathcal{W}^{1,p}(0,T;(E_1,E_0)) 
\to L^{\infty} \left(0,T;\mathcal{L}\left(E_1,E_0\right)\right), \quad
w \mapsto  \left(\begin{matrix}
    A_{11}(w) & 0 \\
    0         & A_{22}(w)
\end{matrix}\right).\]
For fixed $w=(u,v)^{\mathrm{T}} \in \mathcal{W}^{1,p}(0,T;(E_1,E_0))$, we define 
\begin{equation} \label{eqn:source-matrix}
F: \mathcal{W}^{1,p}(0,T;(E_1,E_0)) 
\to L^p \left(0,T;E_0\right), \quad
w \mapsto \left(\begin{matrix}
    -\lambda K'(Ku-f)  \\
    0      
\end{matrix}\right).   
\end{equation}
Then problem \eqref{eqn:auxiliary} can be written as a quasilinear abstract 
Cauchy problem
\begin{equation} \label{eqn:quasilinear-abstract}
    \left\{
        \begin{array}{l} 
            \dfrac{d w}{d t} + A(w)w = F(w), \quad t \in (0,T_0), \vspace{0.2em} \\ 
            w(0) = (f,f)^{\mathrm{T}}, 
    \end{array} \right.
\end{equation}

Let us now see that the linearized operator matrix has the property 
of maximal $L^p$-regularity. It is known that $\partial^{\alpha}_{x_j}$ has 
the mapping property between Bessel potential spaces \cite{GUIDOTTI2009,GUIDOTTI2011}, 
i.e. for any $s \in (\alpha, 1)$,
\[\partial^{\alpha}_{x_j} \in 
\mathcal{L}\left(H^{s,p}_{\pi}(\Omega), H^{s-\alpha,p}_{\pi}(\Omega)\right), 
\quad j = 1,2, \cdots, N.\]
{By combining embeddings \eqref{eqn:slobodeckij-to-bessel} and \eqref{eqn:bessel-to-holder} we know} that  
when $(s-\alpha)p>N$, for any $\delta \in \left(0, s-\alpha-\frac{N}{p}\right)$, 
\begin{equation} \label{eqn:fractional-differential-map}
\partial^{\alpha}_{x_j} \in \mathcal{L}\left(W^{s,p}_{\pi}(\Omega), 
C_{\pi}^{s-\alpha-\frac{N}{p}-\delta}(\overline{\Omega})\right), \quad j = 1,2, \cdots, N.
\end{equation}
Taking $s = 1-\frac{2}{p}$ in \eqref{eqn:fractional-differential-map} and using 
embedding theorem \cite[Theorem III.4.10.2]{AMANN1995}
\begin{equation} \label{eqn:amann-embedding}
\mathcal{W}^{1,p}\left(0,T;\left(E_1^{1/2},E_0^{1/2}\right)\right) \hookrightarrow 
C\left([0,T],E_{1-\frac{1}{p}}^{1/2}\right)
\end{equation}
we know that for any $u \in \mathcal{W}^{1,p}\left(0,T;\left(E_1^{1/2},E_0^{1/2}\right)\right)$,
there exists $\rho = 1-\alpha-\frac{N+2}{p}-\delta \in (0,1)$ such that 
\[\left|\nabla^{\alpha}u\right|
\in C\left([0,T],C^{0,\rho}_{\pi}(\overline{\Omega})\right).\]
Since $k_1>0$ and $\beta \geq 1$, the function $b \in BUC^{\infty}[0,\infty)$. 
It follows from the property of Nemytskii operators \cite[Theorem 3.1]{CHIAPPINELLI1995} that
\[g \mapsto b^{(i)}(g), \quad C^{0,\rho_0}(\overline{\Omega}) \to C^{0,\rho_0}(\overline{\Omega}), \quad i = 0,1\]
is bounded and uniformly Lipschitz continuous on bounded sets for any $\rho_0 \in (0, 1)$. Therefore,
\begin{equation} \label{eqn:regular-texture}
b\left(\left|\nabla^{\alpha}u\right|\right) \in 
C\left([0,T],C^{0,\rho}_{\pi}(\overline{\Omega})\right).
\end{equation}
For any $v \in \mathcal{W}^{1,p}\left(0,T;\left(E_1^{1/2},E_0^{1/2}\right)\right)$, by 
\eqref{eqn:amann-embedding} again we know that 
\begin{equation} \label{eqn:regular-graylevel-1}
    \frac{\overline{v}_{\varepsilon}^{\gamma}}{M^{\gamma}} \in 
C\left([0,T],C^{0,\gamma\left(1-\frac{2+N}{p}\right)}_{\pi}(\overline{\Omega})\right).
\end{equation}
Let $\rho_1 = \min \left\{\rho, \min\{1,\gamma\}\left(1-\frac{2+N}{p}\right)\right\}$, 
then $\rho_1 \in (0,1)$. Combining 
\eqref{eqn:regular-texture} with \eqref{eqn:regular-graylevel-1}, it holds that 
\begin{equation} \label{eqn:holder-continuity-1}
c\left(\left|\nabla^{\alpha} u\right|, \overline{v}_{\varepsilon}\right) 
= \frac{\overline{v}_{\varepsilon}^{\gamma}}{M^{\gamma}} b(\left|\nabla^{\alpha} u\right|)
\in C\left([0,T],C^{0,\rho_1}_{\pi}(\overline{\Omega})\right)
\end{equation}
since H{\"o}lder spaces {on $\overline{\Omega}$} are algebras. It follows that 
\begin{equation}  \label{eqn:continuous-weak-1}
    \left[t \mapsto A_{11}(w(t)) \right] \in 
    C\left([0,T], \mathcal{L}\left(E_1^{1/2},E_0^{1/2}\right)\right)
\end{equation}
for any $w=(u,v)^{\mathrm{T}} \in \mathcal{W}^{1,p}(0,T;(E_1,E_0))$. 
On the other hand, {for each fixed $t \in (0,T)$, the principle symbol of 
$-A_{11}(w(t)) \eqqcolon \mathbf{A}_{11}^t$ is} 
\[{\mathcal{A}_{11,\pi}^t(x,\xi) \coloneqq
    c\left(\left|\nabla^{\alpha} u(t)\right|, \overline{v}_{\varepsilon}(t)\right) |\xi|^2, 
    \quad x \in \Omega, \ \xi \in \mathbb{R}^N \setminus \{0\}.}\] 
From \eqref{eqn:holder-continuity-1}, it follows that there exists
 a constant $C_1>1$ such that for all $x \in \overline{\Omega}$ and $|\xi|=1$,
\begin{equation} \label{eqn:strongly-elliptic-1}
0< \frac{1}{C_1} < {\mathcal{A}_{11,\pi}^t(x,\xi)} < C_1,
\end{equation} 
{which implies that 
$\sigma(\mathcal{A}_{11,\pi}^t) \subset \left\{\lambda \in \mathbb{C}: \mathrm{Re} \lambda > 0 \right\}$
for all $x \in \overline{\Omega}$ and $|\xi|=1$. Therefore, for every $t \in (0,T)$, $-A_{11}(w(t))$ is 
normally elliptic in the sense of Amann \cite{AMANN1990,AMANN1993}.} 
It follows from \eqref{eqn:continuous-weak-1} 
and \eqref{eqn:strongly-elliptic-1} and the weak generation 
theorem \cite[Theorem 8.5]{AMANN1993} that for 
every $t \in (0,T)$, $-A_{11}(w(t))$ is sectorial in $E_0^{1/2}$. 
To obtain the maximal $L^p$-regularity of $A_{11}(w(t))$, 
as pointed out in \cite[Section 13]{KUNSTMANN2004}, the only 
requirement in the end is sufficient regularity of the coefficients, with H{\"o}lder 
continuity being adequate. Thus, for every $t \in (0,T)$,
\begin{equation} \label{eqn:maximal-regularity-1}
    A_{11}(w(t)) \in \mathcal{MR}^p\left(E_1^{1/2},E_0^{1/2}\right).
\end{equation} 
Similarly, letting $\rho_2 = \min\{\gamma, \mu, 1\} \left(1-\frac{2+N}{p}\right)$, 
then we have $\rho_2 \in (0,1)$. For any 
$u \in \mathcal{W}^{1,p}\left(0,T;\left(E_1^{1/2},E_0^{1/2}\right)\right)$,
using \eqref{eqn:amann-embedding} again we deduce that
\begin{equation} \label{eqn:regular-graylevel-2}
    a\left(u,\overline{v}_{\varepsilon}\right) = 
    \lambda_1 \frac{|u|^{\mu}}{M^{\mu}} + 
    \left(1-\lambda_1\right) \frac{\overline{v}_{\varepsilon}^{\gamma}}{M^{\gamma}} \in 
    C\left([0,T],C^{0,\rho_2}_{\pi}(\overline{\Omega})\right)
    \end{equation}
and that
\begin{equation} \label{eqn:continuous-weak-2}
    \left[t \mapsto A_{22}(w(t)) \right] \in 
    C\left([0,T], \mathcal{L}\left(E_1^{1/2},E_0^{1/2}\right)\right)
\end{equation}
for any $w=(u,v)^{\mathrm{T}} \in \mathcal{W}^{1,p}(0,T;(E_1,E_0))$.  
{For each fixed $t \in (0,T)$, the principle symbol of 
$-A_{22}(w(t)) \eqqcolon \mathbf{A}_{22}^t$ is} 
\[{\mathcal{A}_{22,\pi}^t(x,\xi) = 
    a\left(u(t),\overline{v}_{\varepsilon}(t)\right) |\xi|^2, 
    \quad x \in \Omega, \ \xi \in \mathbb{R}^N \setminus \{0\}.}\]
There also exists a constant $C_2>1$ such that for all $x \in \overline{\Omega}$ and $|\xi|=1$,
\begin{equation} \label{eqn:strongly-elliptic-2}
0< \frac{1}{C_2} < {\mathcal{A}_{22,\pi}^t(x,\xi)} < C_2,
\end{equation}
{which ensures that $-A_{22}(w(t))$ is normally elliptic for every $t \in (0,T)$. 
Combining \eqref{eqn:regular-graylevel-2} and \eqref{eqn:strongly-elliptic-2} similarly yields}
\begin{equation} \label{eqn:maximal-regularity-2}
    A_{22}(w(t)) \in \mathcal{MR}^p\left(E_1^{1/2},E_0^{1/2}\right)
\end{equation} 
for every $t \in (0,T)$. It follows from \eqref{eqn:maximal-regularity-1} and 
\eqref{eqn:maximal-regularity-2} that for all $t \in (0,T)$ and 
each $G \in L^p(0,T)$, the linear abstract Cauchy problem 
\[\left\{
    \begin{array}{l} 
        \dfrac{d z}{d \tau} + A(w(t))z = G(\tau), \quad \tau \in (0,T), \vspace{0.2em} \\ 
        z(0) = (0,0)^{\mathrm{T}}
\end{array} \right.\] 
has exactly one strong $L^p$ solution on $(0,T)$, that is 
\begin{equation} \label{eqn:maximal-regularity-each-t}
A(w(t)) \in \mathcal{MR}^p\left(E_1,E_0\right).
\end{equation} 
Using \cite[Theorem 7.1]{AMANN2004} in combination with \eqref{eqn:continuous-weak-1},
\eqref{eqn:continuous-weak-2} and \eqref{eqn:maximal-regularity-each-t} we obtain that  
\[A(w) \in \mathcal{MR}^p(0,T;(E_1,E_0))\]
for any $w \in \mathcal{W}^{1,p}(0,T;(E_1,E_0))$. 

Next, we establish Lipschitz estimates of the operator matrix. For $R>0$, take $w_1=(u_1,v_1)^{\mathrm{T}}, w_2=(u_2,v_2)^{\mathrm{T}} 
 \in B_{\mathcal{W}^{1,p}(0,T;(E_1,E_0))}(0,R)$, then 
\begin{align*}
    & \|A(w_1)-A(w_2)\|_{L^{\infty}(0,T;\mathcal{L}(E_1,E_0))} \\ 
    = \, & \sup_{t \in [0,T]} 
    \sup_{\substack{w=(u,v)^{\mathrm{T}} \\ \|w\|_{E_1}=1}} 
    \|A(w_1(t))w-A(w_2(t))w\|_{E_0} \\
    = \, & \sup_{t \in [0,T]} \sup_{\|u\|_{E_1^{1/2}}=1}
     \|A_{11}(w_1(t))u-A_{11}(w_2(t))u\|_{E_0^{1/2}} \\
     &+ \sup_{t \in [0,T]} \sup_{\|v\|_{E_1^{1/2}}=1}
    \|A_{22}(w_1(t))v-A_{22}(w_2(t))v\|_{E_0^{1/2}} \\
    = \, & \sup_{t \in [0,T]} \sup_{\substack{\|u\|_{W^{1,p}_{\pi}(\Omega)}=1, \\
    \|\varphi\|_{W^{1,p'}_{\pi}(\Omega)}=1}} \int_{\Omega} 
    \left(c\left(\overline{v_1}_{\varepsilon}(t), \left|\nabla^{\alpha} u_1(t)\right|\right) 
    - c\left(\overline{v_2}_{\varepsilon}(t), \left|\nabla^{\alpha} u_2(t)\right|\right) \right)
    \nabla u \cdot \nabla \varphi d x \\
    &+ \sup_{t \in [0,T]} \sup_{\substack{\|v\|_{W^{1,p}_{\pi}(\Omega)}=1, \\ 
    \|\psi\|_{W^{1,p'}_{\pi}(\Omega)}=1}} \int_{\Omega} 
    \left(a\left(u_1(t),\overline{v_1}_{\varepsilon}(t)\right)
    -a\left(u_2(t),\overline{v_2}_{\varepsilon}(t)\right)\right)
    \nabla v \cdot \nabla \psi d x \\
    \leq \, & \sup_{t \in [0,T]} 
    M^{-\gamma} \left\| \overline{v_1}_{\varepsilon}(t)^{\gamma}b\left(\left|\nabla^{\alpha} u_1(t)\right|\right) 
    - \overline{v_2}_{\varepsilon}(t)^{\gamma} b\left(\left|\nabla^{\alpha} u_2(t)\right|\right)\right\|_{\infty} \\
    &+ \lambda_1M^{-\mu} \sup_{t \in [0,T]} \left\| \left|u_1(t)\right|^{\mu}  
    - \left|u_2(t)\right|^{\mu} \right\|_{\infty} 
    + (1-\lambda_1)M^{-\gamma} \sup_{t \in [0,T]} \left\| \overline{v_1}_{\varepsilon}(t)^{\gamma} 
    - \overline{v_2}_{\varepsilon}(t)^{\gamma}\right\|_{\infty} \\
    \coloneqq \, & I_1 + I_2 + I_3.
\end{align*}
We estimate $I_1, I_2$ and $I_3$ separately. From now on $C=C(R)$ denotes a positive constant which can take different value in 
different places. 
\begin{align*}
    I_1 \leq \, & C \sup_{t \in [0,T]} 
    \left\| \overline{v_1}_{\varepsilon}(t)^{\gamma} \right\|_{\infty}
    \left\| b\left(\left|\nabla^{\alpha} u_1(t)\right|\right) 
    - b\left(\left|\nabla^{\alpha} u_2(t) \right|\right) \right\|_{\infty} \\
    & +C \sup_{t \in [0,T]} \left\| b\left(\left|\nabla^{\alpha} u_2(t)\right|\right) \right\|_{\infty}  
    \left\| \overline{v_1}_{\varepsilon}(t)^{\gamma}
    - \overline{v_2}_{\varepsilon}(t)^{\gamma}\right\|_{\infty} \\
    \leq \, & C \sup_{t \in [0,T]}  \left\| \left|\nabla^{\alpha} u_1(t)\right| 
    - \left|\nabla^{\alpha} u_2(t)\right|\right\|_{\infty} +
    C \sup_{t \in [0,T]} \left\| \overline{v_1}_{\varepsilon}(t) - \overline{v_2}_{\varepsilon}(t)\right\|_{\infty} \\
    \leq \, & C \sup_{t \in [0,T]}  \sum_{j=1}^N 
    \left\| \partial_{x_j}^{\alpha} \left(  u_1(t) - u_2(t) \right) 
    \right\|_{C^{0,\rho}_{\pi}(\overline{\Omega})} 
    + C \sup_{t \in [0,T]} \left\| v_1(t) - v_2(t)\right\|_{C^{0,\rho}_{\pi}(\overline{\Omega})} \\ 
    \leq \, & C \sum_{j=1}^N 
    \left\|\partial_{x_j}^{\alpha}\right\|_{
    \mathcal{L}\left(H^{1-\frac{2}{p}-\delta,p}_{\pi}(\Omega), 
    H^{1-\alpha-\frac{2}{p}-\delta,p}_{\pi}(\Omega)\right)}
    \sup_{t \in [0,T]}  \left\| u_1(t) - u_2(t)\right\|_{W^{1-\frac{2}{p}, p}_{\pi}(\Omega)} \\
    &+ C \sup_{t \in [0,T]} \left\| v_1(t) - v_2(t)\right\|_{W^{1-\frac{2}{p}, p}_{\pi}(\Omega)} \\
    \leq \, & C \left\| u_1 - u_2\right\|_{C\left([0,T],W^{1-\frac{2}{p}, p}_{\pi}(\Omega)\right)} 
    + C \left\| v_1 - v_2\right\|_{C\left([0,T],W^{1-\frac{2}{p}, p}_{\pi}(\Omega)\right)}.
\end{align*}
{Here we use the fact that $r\mapsto\max\{r,\varepsilon\}$ is a Lipschitz 
mapping with constant $1$.} Similarly, we have 
\[I_2 \leq C \left\| u_1 - u_2\right\|_{C\left([0,T],W^{1-\frac{2}{p}, p}_{\pi}(\Omega)\right)}, \
I_3 \leq C \left\| v_1 - v_2\right\|_{C\left([0,T],W^{1-\frac{2}{p}, p}_{\pi}(\Omega)\right)}.\]
Combining the estimates of $I_1, I_2$ and $I_3$, we obtain that 
\begin{align*}
    & \|A(w_1)-A(w_2)\|_{L^{\infty}(0,T;\mathcal{L}(E_1,E_0))} \\ 
    & \leq  C \left(\left\| u_1 - u_2\right\|_{C\left([0,T],E_{1-\frac{1}{p}}^{1/2}\right)} 
    + \left\| v_1 - v_2\right\|_{C\left([0,T],E_{1-\frac{1}{p}}^{1/2}\right)}\right) \\
    & =  C \left\| w_1 - w_2\right\|_{C\left([0,T], E_{1-\frac{1}{p}}\right)} \\ 
    & \leq  C \left\| w_1 - w_2\right\|_{\mathcal{W}^{1,p}(0,T;(E_1,E_0))},
\end{align*}
from which it follows that 
\begin{equation} \label{eqn:operator-matrix-bounded-lip}
A \in \mathcal{C}_{\mathrm{Volt}}^{0,1}\left(\mathcal{W}^{1,p}(0,T;(E_1,E_0)),
\mathcal{MR}^p(0,T;(E_1,E_0))\right).
\end{equation}

Now we check the Lipschitz property of the source term. It is clear 
that $F(0) = \left(\lambda K'f, 0\right)^{\mathrm{T}} \in L^{\infty}(0,T;E_0)$. 
Set $\widetilde{F} = F - F(0) = \left(-\lambda K'K, 0\right)^{\mathrm{T}}$. For $R>0$, 
take $w_1=(u_1,v_1)^{\mathrm{T}}, w_2=(u_2,v_2)^{\mathrm{T}} 
\in B_{\mathcal{W}^{1,p}(0,T;(E_1,E_0))}(0,R)$, we have that 
\begin{align*}
    \left\|\widetilde{F}(w_1)-\widetilde{F}(w_2)\right\|_{L^{\infty}(0,T;E_0)}  
    \leq \, & \sup_{t \in [0,T]} \left\|\widetilde{F}(w_1(t))-\widetilde{F}(w_2(t))\right\|_{E_0} \\ 
    \leq \, & \lambda \sup_{t \in [0,T]} \left\| K'K\left(u_1(t) - u_2(t)\right) \right\|_{\infty} \\ 
    \leq \, & C \sup_{t \in [0,T]} \left\| u_1(t) - u_2(t) \right\|_{W^{1-\frac{2}{p}, p}_{\pi}(\Omega)} \\
    \leq \, & C \left\| w_1 - w_2\right\|_{C\left([0,T],E_{1-\frac{1}{p}}\right)} \\
    \leq \, & C \left\| w_1 - w_2\right\|_{\mathcal{W}^{1,p}(0,T;(E_1,E_0))},
\end{align*}
which implies that 
\begin{equation} \label{eqn:source-term-bounded-lip}
F \in \mathcal{C}_{\mathrm{Volt}}^{0,1}\left(\mathcal{W}^{1,p}(0,T;(E_1,E_0)); 
L^{\infty} \left(0,T;E_0\right), L^{p} \left(0,T;E_0\right)\right).
\end{equation}

It follows from \eqref{eqn:operator-matrix-bounded-lip}, \eqref{eqn:source-term-bounded-lip} and 
Theorem \ref{thm:amann-main} that there exist a maximal $T_{\max} \in \left(0,T_{0}\right]$ and 
a unique solution $w=(u,v)^{\mathrm{T}} \in \mathcal{W}_{\mathrm{loc}}^{1,p}(0,T_{\max};(E_1,E_0))$ 
of the quasilinear abstract Cauchy problem \eqref{eqn:quasilinear-abstract}. Then $(u,v)$ is 
the unique weak solution to auxiliary problem \eqref{eqn:auxiliary} on $\left[0,T_{\max}\right)$. 

{Now set $\varepsilon = \inf_{x \in \Omega} f(x)$. Recall the minimum principle mentioned 
in Remark \ref{rmk:auxiliary} (ii):
\[v(t) \geq \inf_{x \in \Omega} f(x) = \varepsilon 
\quad \mathrm{a.e.~in~}  \Omega\] 
for every $t \in \left(0,T_{\max}\right)$, which implies 
that $v = \max\{v, \varepsilon\} = \overline{v}_{\varepsilon}$. At this point, 
the concept of a weak solution to problem \eqref{eqn:auxiliary} coincides with that 
of a weak solution to \eqref{eqn:main}}, which concludes the proof.
\end{proof}

{
\begin{remark} For the well-posedness analysis of \eqref{eqn:main} we require that $f$ be 
bounded away from zero a.e. in $\Omega$ to ensure $\inf_{x \in \Omega} f(x) > 0$. 
This allows us to rule out the possible degeneracy 
induced by the factor $|v|^{\gamma}$ in the diffusion coefficient through the construction of the 
auxiliary problem \eqref{eqn:auxiliary}. Introducing such a 
technical assumption is common in the analysis of image processing models involving gray level 
indicators; see for example \cite{SHAN2019,MAJEE2020}. 
At the same time, from a modeling viewpoint, a degraded image with additive Gaussian noise 
may contain negative samples (as Gaussian noise is unbounded), and potential negative values are clamped to 
$0$ before processing. We emphasize that our theoretical development does not cover the case
$\left|\{x \in \Omega : f(x) = 0\}\right| > 0$; this limitation should not be read as a claim that 
solutions fail to exist for such initial data; rather, allowing a region of positive measure where $f$ vanishes 
introduces substantial technical challenges (for instance, loss of normally ellipticity) 
that are beyond the scope of this paper. Empirically, zero-valued (black) pixels
have shown no observable impact on the numerical results reported in Sect. \ref{sec5}. A rigorous 
theoretical analysis of this more general case is an interesting and challenging open problem 
left for future work.
\end{remark}
}

A general continuity result {\cite[Theorem 3.1]{AMANN2005}} for quasilinear parabolic 
problems implies the continuous dependence of the maximal existence time of solutions 
on both the initial data and the right-hand side. Utilizing this continuity result, 
we can derive the following proposition of problem \eqref{eqn:main}. 

\begin{proposition} \label{prop:continuity}
    Let the assumptions of Theorem \ref{thm:well-posedness-weak} hold. 
    {Assume that $K$ satisfies the DC-condition, i.e. $K(1)=1$, 
    treating $1\in L^{\infty}(\Omega)$.}
    Given $\tau \in (0,T_0]$ and $c_*>0$, there exists $r_0 = r_0(\tau) > 0$ such that for 
    each $f \in B_{W^{1-\frac{2}{p}, p}_{\pi}(\Omega)}(c_*,r_0)$,
    the weak solution to problem \eqref{eqn:main} exists on $[0,\tau]$.
\end{proposition}

\begin{proof}
    {Note that when $f$ in \eqref{eqn:main} is taken to be the
    constant $c_* > 0$, the couple $(u,v) \equiv (c_*,c_*)$ is a weak solution to 
    problem \eqref{eqn:main} on $[0,T_0)$.}
    Set $w_0 = (c_*,c_*)^{\mathrm{T}}$, then for any $\tau \in (0,T_0]$, 
    there exists $R > 0$ such that 
    $\|w_0\|_{\mathcal{W}^{1,p}(0,\tau;(E_1,E_0))} < R$. It follows from 
    \cite[Theorem 3.1]{AMANN2005} that there exists $r_0>0$ such that for 
    any $(f,f)^{\mathrm{T}} \in B_{E_{1-\frac{1}{p}}}(w_0,2r_0)$, the corresponding 
    maximal interval of existence contains $[0,\tau]$. This fact concludes the proof.
\end{proof}

For a general $K$, it is challenging to employ the Stampacchia's truncation 
method to establish the maximum and minimum principles satisfied by $u$. However, 
Theorem \ref{thm:well-posedness-weak} and {Morrey's embedding theorem imply} 
the local solution $u$ is actually bounded.
Using the Moser's method, we establish a specific $L^{\infty}$ estimate for $u$.

\begin{proposition} 
    Let the assumptions of Theorem \ref{thm:well-posedness-weak} hold. 
    Let $(u,v)$ be the weak solution to  
    problem \eqref{eqn:main} on $[0,T)$ with 
    initial datum $f$, then for every $t \in (0,T)$, 
    \begin{equation} \label{prop:l-infty-estimate}
        \|u(t,\cdot)\|_{\infty} \leq e^{2 \lambda t} \left( \|f\|_{\infty} 
        + \|K' f\|_{\infty}\right).
    \end{equation}
\end{proposition}

\begin{proof}
    Since $u(t)$ is H{\"o}lder continuous, we can take 
    $\varphi = u^{2r-1}(t)$ in \eqref{eqn:main-weak} with $r \geq 1$. Then we have 
    \begin{align*}
    \frac{1}{2r} \frac{d}{dt} \int_{\Omega} u^{2r}(t) dx 
    &+ (2r-1)\int_{\Omega} c\left(\left|\nabla^{\alpha} u\right|, v\right) 
    u^{2(r-1)} |\nabla u|^2 dx \\ 
    +& \lambda \int_{\Omega} K'Ku \cdot u^{2r-1} dx 
    - \lambda \int_{\Omega} K'f \cdot u^{2r-1} dx=0.
    \end{align*}
    Applying Young's and H{\"o}lder's inequalities, we obtain that 
    \begin{align*}
        \frac{1}{2r} \frac{d}{dt} \|u(t,\cdot)\|_{L^{2r}(\Omega)}^{2r} 
    & \leq \frac{\lambda}{2r}  \int_{\Omega} |K'Ku|^{2r} dx + 
    \frac{2r-1}{2r} \lambda \int_{\Omega} |u|^{2r} dx \\
    & \quad + \frac{\lambda}{2r} \int_{\Omega} |K'f|^{2r} dx + 
    \frac{2r-1}{2r} \lambda \int_{\Omega} |u|^{2r} dx \\
    & \leq \frac{\lambda}{2r} \|K'f\|_{L^{2r}(\Omega)}^{2r}
    + \frac{4r-2+\|K\|^2|\Omega|}{2r} \lambda\|u(t,\cdot)\|_{L^{2r}(\Omega)}^{2r}.
    \end{align*}
    It follows from Gr{\"o}nwall's inequality that 
    \[\|u(t,\cdot)\|_{L^{2r}(\Omega)}^{2r} 
    \leq \exp\{(4r-2+\|K\|^2|\Omega|)\lambda t\} 
    \left(\|u(0,\cdot)\|_{L^{2r}(\Omega)}^{2r} + 
    \lambda t \|K'f\|_{L^{2r}(\Omega)}^{2r}\right).\]
    Using the monotonicity of $\ell^p$ norms, we arrive the estimate 
    {\small \[\|u(t,\cdot)\|_{L^{2r}(\Omega)} \leq 
    \exp\left\{\left(2+\frac{\|K\|^2|\Omega|-2}{2r}\right)\lambda t\right\} 
    \left(\|u(0,\cdot)\|_{L^{2r}(\Omega)} 
    + (\lambda t)^{\frac{1}{2r}} \|K'f\|_{L^{2r}(\Omega)}\right)\]}
    which yields, as $r \to \infty$, the estimate 
    \eqref{prop:l-infty-estimate} holds.
\end{proof}

Now we apply the well-known generalized principle of linearized stability to 
conclude the following stability result for the equilibria of \eqref{eqn:main}. 
This result pertains to the case of $\lambda = 0$, which leads to a pure diffusion 
system.

\begin{theorem} \label{prop:stability}
    Let the assumptions of Theorem \ref{thm:well-posedness-weak} hold. 
    For any given $c_*>0$, there exists $r_0>0$ such that the unique weak 
    solution $(u,v)$ to problem \eqref{eqn:main}  
    {with $\lambda = 0$} and initial datum
    $f \in B_{W^{1-\frac{2}{p}, p}_{\pi}(\Omega)}(c_*,r_0)$ exists globally
    and converges exponentially in the topology 
    of $W^{1-\frac{2}{p}, p}_{\pi}(\Omega)^2$ to some 
    equilibrium $(u_{\infty},v_{\infty})$ of \eqref{eqn:main} as $t \to \infty$.
\end{theorem}

\begin{proof}
    It is sufficient to prove that the equilibrium $(c_*, c_*)^{\mathrm{T}}$ is 
    normally stable in $W^{1-\frac{2}{p}, p}_{\pi}(\Omega)^2$ when $\lambda=0$. 
    For a definition of normally stable, see \cite{PRUSS2009}.
    Let $\mathcal{E}$ denote 
    the set of equilibrium of \eqref{eqn:quasilinear-abstract}, which 
    means that $w = (u,v)^{\mathrm{T}} \in \mathcal{E}$ 
    if and only if $w \in E_1$ and $A(w)w=0=F(w)$, i.e.
    for any $\varphi, \psi \in W_{\pi}^{1,p'}(\Omega)$, 
    \begin{equation} \label{eqn:equilibrium-weak}
        \int_{\Omega} c\left(\left|\nabla^{\alpha} u\right|,v\right)
        \nabla u \cdot \nabla \varphi d x =
        \int_{\Omega} a\left(u, v\right) \nabla v \cdot \nabla \psi d x = 0.
    \end{equation}
    It is clear that $\mathcal{E}_0 \coloneqq \mathbb{R}^2 \subset \mathcal{E}$. At each 
    $w^* = (c_*, c_*)^{\mathrm{T}} \in \mathcal{E}_0$, the linearization of $A$ is given by
    \[A_0 \coloneqq A(w^*) + \left.\frac{d}{dw} \left[A(w)w^*\right]\right|_{w=w^*}
    = \left(\begin{matrix}
        -c_{11} \Delta & 0 \\
        0         & -c_{22} \Delta
    \end{matrix}\right)\]
    where $c_{11}=c_{11}(c_*), c_{22}=c_{22}(c_*) > 0$, $\Delta: W_{\pi}^{1,p}(\Omega) 
    \subset W_{\pi}^{-1,p}(\Omega) \to W_{\pi}^{-1,p}(\Omega)$. Since the spectrum of
    $-\Delta$ is discrete and consists only of non-negative eigenvalues, we know that 
    \[\sigma(A_0) \setminus \{0\} \subset \mathbb{C}_+
    = \left\{z \in \mathbb{C}: \mathrm{Re} z > 0\right\}\]
    and $\{0\}$ is isolated in $\sigma(A_0)$. It follows from $W_{\pi}^{1,p}(\Omega) 
    \doublehookrightarrowalt W_{\pi}^{-1,p}(\Omega)$ that $A_0$ has 
    compact resolvent, which implies that $A_0$ generates an eventually compact 
    semigroup since $-\Delta$ is sectorial. By \cite[Corollary 5.3.2]{ENGEL2000}, we 
    can conclude that $0$ is a pole of $R(\cdot,A_0)$.

    To show 0 is a semi-simple eigenvalue, we will prove that 
    $N(A_0) = N(A_0^2)$. Clearly $N(A_0) = \mathcal{E}_0$, taking $z \in N(A_0^2)$, then 
    there exists $w \in \mathcal{E}_0$ such that $A_0 z = w$. The periodic condition 
    ensures that $w=0$, and it follows that $z \in N(A_0)$. Therefore,
    \cite[Remark A.2.4]{LUNARDI1995} yields that
    $0$ is semi-simple. As $\mathcal{E}_0$ is a subspace 
    of $W^{1-\frac{2}{p}, p}_{\pi}(\Omega)^2$, the tangent space at $w^*$ is isomorphic to
    $N(A_0)$.

    In \cite[Remark 2.2]{PRUSS2009} it is shown that all equilibria close to $w^*$ are contained 
    in a manifold $\mathcal{M}$ of dimension $\mathrm{dim}\left(N(A_0)\right) = 2$. Since 
    the dimension of $\mathcal{E}_0$ is also $2$, there exists an open neighborhood
    $U_0 \subset W^{1-\frac{2}{p}, p}_{\pi}(\Omega)^2$ of $w^*$ such that 
    $\mathcal{M} \cap U_0 = \mathcal{E}_0 \cap U_0$. Thus, $U_0$ contains no other 
    equilibria than the elements of $\mathcal{E}_0$, i.e. 
    $\mathcal{E}_0 \cap U_0 = \mathcal{E} \cap U_0$. 

    So all assumptions of the generalized principle of linearized stability 
    \cite[Theorem 2.1]{PRUSS2009} are satisfied. This principle concludes the proof.
\end{proof}

\begin{remark}
    {$ $}
    \begin{enumerate}[(i)] \item {In the general case ($\lambda \ne 0$), 
    the question of whether steady states even exist remains a highly challenging open problem,
    let alone the analysis of convergence of solutions to such states.
    In addition, the proposed model \eqref{eqn:main} is not of variational type, 
    making it generally difficult to define a corresponding Lyapunov functional. 
    However}, when $\lambda = 0$, {it possesses a natural} Lyapunov functional given by 
    \[E = \frac{1}{2} \int_{\Omega} (u^2+v^2) dx.\]
     
    \item Let the assumptions of Theorem \ref{thm:well-posedness-weak} hold and $(u,v)$ be 
    the weak solution to problem \eqref{eqn:main} on $[0,T)$ with 
    the initial datum $f$. {If $\lambda = 0$ or} $K=I$, the weak solution to 
    problem \eqref{eqn:main} has additionally the following properties:
    \begin{enumerate}[(a)]
        \item (Extremum principle) For every $t \in (0,T)$,
            \[\inf_{x \in \Omega} f(x) \leq u(t,x) 
                ,\ v(t,x) \leq \sup_{x \in \Omega} f(x), 
            \quad \mathrm{for~a.e.} \ x \in \Omega.\]
        \item (Average invariance) For every $t \in (0,T)$, it holds that
            \[\int_{\Omega} u(t,x) dx = \int_{\Omega} v(t,x) dx = \int_{\Omega} f(x) dx.\]
    \end{enumerate}
    The proofs are straightforward. 
\end{enumerate} 
\end{remark}

\section{Regularity results for the proposed model}\label{sec4}

In this section, we establish regularity results for the solutions of 
problem \eqref{eqn:main}. Firstly, when the regularity of the initial datum is 
improved, the local existence and uniqueness of strong solutions to 
problem \eqref{eqn:main} is obtained via Maximum Regularity Theory. 
{Let $p>2$. Given $f \in W_{\pi}^{2-\frac{2}{p},p}(\Omega)$, a couple of 
functions 
\[u, v \in L_{\mathrm{loc}}^p\left(0,T_{\max};W_{\pi}^{2,p}(\Omega)\right)   
    \cap W_{\mathrm{loc}}^{1,p}\left(0,T_{\max};L_{\pi}^{p}(\Omega)\right)\] 
is called a strong solution to problem \eqref{eqn:main}  
on $[0,T)$, if it satisfies \eqref{eqn:main} a.e. in $Q_T \coloneqq (0,T) \times \Omega$.}

\begin{theorem} \label{thm:well-posedness-strong}
    Assume that $p \in \left(\frac{2+N}{1-\alpha}, \infty\right)$, 
    $\beta \in \{1\} \cup [2,\infty)$, $\mu > 1$ and $\gamma \geq 1$. 
    Let $f \in W_{\pi}^{2-\frac{2}{p},p}(\Omega)$ with {$f \geq 0$ and 
    bounded away from zero a.e. in $\Omega$}.
    Then there exist a unique 
    maximal $T_{\max} \in (0,T_0]$ such that problem \eqref{eqn:main} possess 
    a unique strong solution $(u, v)$ on $[0, T_{\max})$.
\end{theorem}

\begin{proof}
    The method for proving the local existence and uniqueness of strong solutions is similar to the proof of 
    Theorem \ref{thm:well-posedness-weak}, with the main differences lying in establishing 
    the maximum $L^p$-regularity of the linearized operator matrix and the Lipschitz 
    estimate between the operator matrix and the source term. Firstly, we reformulate 
    an auxiliary problem 
    \begin{equation}
        \left\{
        \begin{array}{ll}
            \displaystyle \vspace{0.2em} 
            u_t = 
            c\left(\left|\nabla^{\alpha} u\right|, \overline{v}_{\varepsilon}\right) \Delta u
            + \nabla c\left(\left|\nabla^{\alpha} u\right|, v\right) 
            \cdot \nabla u
            - \lambda K'(Ku-f), & \textrm{in~~} (0,T_0) \times \Omega,  \\ 
            v_t = 
            a\left(u, \overline{v}_{\varepsilon}\right) \Delta v
            + \nabla a\left(u, v\right) \cdot \nabla v, 
            & \textrm{in~~} (0,T_0) \times \Omega, \\
            u, v \textrm{~~periodic on} \Omega & \textrm{for~~} (0,T_0), \\
            u(0,\cdot) = v(0,\cdot) = f, & \textrm{in~~} \Omega, 
        \end{array}
        \right. \label{eqn:auxiliary-strong}
    \end{equation}
    as a quasilinear abstract Cauchy problem. 
    In this section, we always choose
    \[E_0 = L_{\pi}^{p}(\Omega)^2, \quad E_1 = W_{\pi}^{2,p}(\Omega)^2,\]
    which is well-known to satisfy $E_1 \xhookrightarrow[]{d} E_0$. Then 
    \[E_0^{1/2} = L_{\pi}^{p}(\Omega), \quad E_1^{1/2} = W_{\pi}^{2,p}(\Omega), \quad 
    E_{1-\frac{1}{p}} = W_{\pi}^{2-\frac{2}{p},p}(\Omega)^2, \quad
    E_{1-\frac{1}{p}}^{1/2} = W_{\pi}^{2-\frac{2}{p},p}(\Omega).\]
    For fixed $w_0=(u_0,v_0)^{\mathrm{T}} \in \mathcal{W}^{1,p}(0,T;(E_1,E_0))$, 
    we define linear differential operators $A_{11}(w_0)$ and $A_{22}(w_0)$ as follows:
    \[\begin{array}{c}
        A_{11}(w_0)u = -c\left(\left|\nabla^{\alpha} u_0\right|, 
    \overline{v_0}_{\varepsilon}\right) \Delta u - 
    \nabla c\left(\left|\nabla^{\alpha} u_0\right|, v_0\right) 
    \cdot \nabla u, \quad u \in W_{\pi}^{1,p}(\Omega), \\[1ex]
        A_{22}(w_0)v = -a\left(u_0, \overline{v_0}_{\varepsilon}\right) \Delta v
        - \nabla a\left(u_0, v_0\right) \cdot \nabla v, \quad v \in W_{\pi}^{1,p}(\Omega). 
    \end{array}\]
    An operator matrix $A$ can then be given by
    \[A: \mathcal{W}^{1,p}(0,T;(E_1,E_0)) 
    \to L^{\infty} \left(0,T;\mathcal{L}\left(E_1,E_0\right)\right), \quad
    w \mapsto  \left(\begin{matrix}
    A_{11}(w) & 0 \\
    0         & A_{22}(w)
    \end{matrix}\right).\]
    Defining $F$ as in \eqref{eqn:source-matrix}, then
    problem \eqref{eqn:auxiliary-strong} can be rewritten as
    a quasilinear abstract Cauchy problem with the same form as 
    \eqref{eqn:quasilinear-abstract}:
    \begin{equation} \label{eqn:quasilinear-abstract-strong}
        \left\{
            \begin{array}{l} 
                \dfrac{d w}{d t} + A(w)w = F(w), \quad t \in (0,T_0), \vspace{0.2em} \\ 
                w(0) = (f,f)^{\mathrm{T}}, 
        \end{array} \right.
    \end{equation}
    
    Next, we check the maximal $L^p$-regularity of the linearized operator matrix. 
    {Again, embeddings \eqref{eqn:slobodeckij-to-bessel} and \eqref{eqn:bessel-to-holder}} imply 
    that $\delta \in \left(0,1-\alpha-\frac{2+N}{p}\right)$,
    \[\partial^{\alpha}_{x_j} \in \mathcal{L}\left(W^{2-\frac{2}{p},p}_{\pi}(\Omega), 
        C_{\pi}^{1,1-\alpha-\frac{N+2}{p}-\delta}(\overline{\Omega})\right), 
    \quad j = 1,2, \cdots, N.\]
    It follows from Amann's embedding theorem \eqref{eqn:amann-embedding} that there exists
    $\rho = 1-\alpha-\frac{N+2}{p}-\delta \in (0,1)$ such that for any
    $u \in \mathcal{W}^{1,p}\left(0,T;\left(E_1^{1/2},E_0^{1/2}\right)\right)$,
    \[\left|\nabla^{\alpha}u\right|
    \in C\left([0,T],C^{1,\rho}_{\pi}(\overline{\Omega})\right).\]
    The property of Nemytskii operators implies that 
    \[b\left(\left|\nabla^{\alpha}u\right|\right), \
      b'\left(\left|\nabla^{\alpha}u\right|\right) \in 
      C\left([0,T],C^{0,\rho}_{\pi}(\overline{\Omega})\right).\]
    Set $\rho_0 = 1-\alpha-\frac{N+2}{p}-\delta$, we choose $\rho_1 = \min\{\rho, \rho_0\}$ 
    if $\gamma = 1$ and $\rho_1 = \min \left\{\rho, \min\{1,\gamma-1\}\rho_0\right\}$ 
    if $\gamma > 1$, then for any
    $u, v \in \mathcal{W}^{1,p}\left(0,T;\left(E_1^{1/2},E_0^{1/2}\right)\right)$,
    \begin{equation} \label{eqn:u-c-holder-continuity}
    c\left(\left|\nabla^{\alpha} u\right|, \overline{v}_{\varepsilon}\right), \ 
    \partial_{x_j} c\left(\left|\nabla^{\alpha} u\right|, v\right) \in 
    C\left([0,T],C^{0,\rho_1}_{\pi}(\overline{\Omega})\right)
    \end{equation}
    since H{\"o}lder spaces {on $\overline{\Omega}$} are algebras. It follows that 
    \begin{equation}  \label{eqn:continuous-strong-1}
        \left[t \mapsto A_{11}(w(t)) \right] \in 
        C\left([0,T], \mathcal{L}\left(E_1^{1/2},E_0^{1/2}\right)\right)
    \end{equation}
    for any $w=(u,v)^{\mathrm{T}} \in \mathcal{W}^{1,p}(0,T;(E_1,E_0))$. On 
    the other hand, {as in the proof of 
    Theorem \ref{thm:well-posedness-weak},} for {each fixed} $t \in (0,T)$ we 
    consider the principal symbol of $\mathbf{A}_{11}^t = -A_{11}(w(t))${, given by} 
    \[{\mathcal{A}_{11,\pi}^t(x,\xi) \coloneqq
    c\left(\left|\nabla^{\alpha} u(t)\right|, \overline{v}_{\varepsilon}(t)\right) |\xi|^2, 
    \quad x \in \Omega, \ \xi \in \mathbb{R}^N \setminus \{0\}.}\] 
    It follows from \eqref{eqn:u-c-holder-continuity} that there exists $C_1>0$ such 
    that for any $x \in \Omega$ and $|\xi|=1$, 
    \[0< \frac{1}{C_1} < {\mathcal{A}_{11,\pi}^t(x,\xi)} < C_1,\]
    which implies that $\left\{\lambda \in \mathbb{C}: \mathrm{Re} \lambda 
    < \frac{1}{C_1} \right\} \subset \rho(\mathcal{A}_{11,\pi}^{t})$ holds 
    for all $x \in \Omega$ and $|\xi|=1$. $\mathcal{A}_{11,\pi}$ 
    is real, so there exists $\theta \in \left(0,\frac{\pi}{2}\right)$ such that 
    for all $x \in \Omega$ and $|\xi|=1$,
    \[\sigma(\mathcal{A}_{11,\pi}^{t}) \subset \left\{\lambda \in \mathbb{C} : 
    \mathrm{Re} \lambda \geq \frac{1}{C_1} \right\} 
    \cap \mathbb{R} \subset \left\{z \in \mathbb{C}  \setminus 
    \{0\} : |\mathrm{arg}(z)|<\theta \right\}.\]
    Thus, $-A_{11}(w(t))$ is uniformly $\left(C_1, \theta\right)$-elliptic 
    {in the sense of \cite{AMANN1994,KUNSTMANN2004}} for every 
    $t \in (0,T)$. From \cite[Corollary 9.5]{AMANN1994}, it follows that
    $A_{11}(w(t))$ is a sectorial operator on $L^2(\Omega)$, i.e. $A_{11}(w(t))$ is the generator 
    of an analytic semigroup $T_2(t)$ on $L^2(\Omega)$. \cite[Theorem 9.4.2]{FRIEDMAN1983} 
    implies that $T_2(t)$ can be represented by a kernel satisfying a Gaussian bound. 
    It follows from \cite[Theorem 3.1]{HIEBER1997} that for any $t \in (0,T)$,
    \begin{equation} \label{eqn:maximal-regularity-strong-1}
        A_{11}(w(t)) \in \mathcal{MR}^p\left(E_1^{1/2},E_0^{1/2}\right).
    \end{equation} 
    Set $\rho_3 = 1-\frac{N+2}{p}$, we choose $\rho_2 = \min\{1, \mu - 1\}\rho_3$ 
    if $\gamma = 1$ and $\rho_2 = \min \{1,\gamma-1, \mu - 1 \} \rho_3$ 
    if $\gamma > 1$, it is easy to see that for any 
    $u, v \in \mathcal{W}^{1,p}\left(0,T;\left(E_1^{1/2},E_0^{1/2}\right)\right)$,
    \[a\left(u, \overline{v}_{\varepsilon}\right), \ 
    \partial_{x_j} a\left(u, v\right) \in 
    C\left([0,T],C^{0,\rho_2}_{\pi}(\overline{\Omega})\right),\]
    from which it follows that 
    \begin{equation}  \label{eqn:continuous-strong-2}
        \left[t \mapsto A_{22}(w(t)) \right] \in 
        C\left([0,T], \mathcal{L}\left(E_1^{1/2},E_0^{1/2}\right)\right)
    \end{equation}
    for any $w=(u,v)^{\mathrm{T}} \in \mathcal{W}^{1,p}(0,T;(E_1,E_0))$. The process of 
    obtaining \eqref{eqn:maximal-regularity-strong-1} reveals that the maximal $L^p$-
    regularity of the linearized operator ultimately boils down to sufficient regularity
    of the coefficients. A similar argument yields that for any $t \in (0,T)$,
    \begin{equation} \label{eqn:maximal-regularity-strong-2}
        A_{22}(w(t)) \in \mathcal{MR}^p\left(E_1^{1/2},E_0^{1/2}\right).
    \end{equation} 
    From \eqref{eqn:maximal-regularity-strong-1} and 
    \eqref{eqn:maximal-regularity-strong-2}, it follows that for all $t \in (0,T)$ and 
    each $G \in L^p(0,T)$, the linear abstract Cauchy problem 
    \[\left\{
    \begin{array}{l} 
        \dfrac{d z}{d \tau} + A(w(t))z = G(\tau), \quad \tau \in (0,T), \vspace{0.2em} \\ 
        z(0) = (0,0)^{\mathrm{T}}
    \end{array} \right.\] 
    has a unique strong $L^p$ solution on $(0,T)$, i.e. 
    \begin{equation} \label{eqn:maximal-regularity-strong-each-t}
        A(w(t)) \in \mathcal{MR}^p\left(E_1,E_0\right).
    \end{equation} 
    By \eqref{eqn:continuous-strong-1}, \eqref{eqn:continuous-strong-2}, 
    \eqref{eqn:maximal-regularity-strong-each-t} and \cite[Theorem 7.1]{AMANN2004}, 
    we can conclude that for any $w \in \mathcal{W}^{1,p}(0,T;(E_1,E_0))$,
    \[A(w) \in \mathcal{MR}^p(0,T;(E_1,E_0)).\]
     
    Now we establish Lipschitz estimates of the operator matrix. For $R>0$, 
    take $w_1=(u_1,v_1)^{\mathrm{T}}, w_2=(u_2,v_2)^{\mathrm{T}} 
    \in B_{\mathcal{W}^{1,p}(0,T;(E_1,E_0))}(0,R)$, then 
    \begin{align*}
        & \|A(w_1)-A(w_2)\|_{L^{\infty}(0,T;\mathcal{L}(E_1,E_0))} \\
        &=   \sup_{t \in [0,T]} 
        \sup_{\substack{w=(u,v)^{\mathrm{T}} \\ \|w\|_{E_1}=1}} 
        \|A(w_1(t))w-A(w_2(t))w\|_{E_0} \\
        &=   \sup_{t \in [0,T]} \sup_{\|u\|_{E_1^{1/2}}=1} 
         \|A_{11}(w_1(t))u-A_{11}(w_2(t))u\|_{E_0^{1/2}} \\
         &\quad + \sup_{t \in [0,T]} \sup_{\|v\|_{E_1^{1/2}}=1}
        \|A_{22}(w_1(t))v-A_{22}(w_2(t))v\|_{E_0^{1/2}} \\
        &\coloneqq  I_1 + I_2.
    \end{align*}
    We estimate $I_1$ and $I_2$ separately. From now on $C=C(R)$ denotes a positive constant 
    which can take different value in different places. Through a process similar to the 
    Lipschitz estimate of $A_{11}$ in the proof of Theorem \ref{thm:well-posedness-weak}, 
    we obtain the following estimate by applying the trick of adding and subtracting 
    the same term repeatedly. 
    \begin{align*}
        I_1 \leq \, & \sup_{t \in [0,T]} \sup_{\|u\|_{W^{2,p}_{\pi}(\Omega)}=1} 
        \left\|\left(c\left(\left|\nabla^{\alpha} u_1\right|, 
        \overline{v_1}_{\varepsilon}\right) - c\left(\left|\nabla^{\alpha} u_2\right|, 
        \overline{v_2}_{\varepsilon}\right)\right) \Delta u \right\|_{L^p_{\pi}(\Omega)} \\
        &+ \sup_{t \in [0,T]} \sup_{\|u\|_{W^{2,p}_{\pi}(\Omega)}=1} 
        \left\|\left(\nabla c\left(\left|\nabla^{\alpha} u_1\right|, v_1\right) 
        - \nabla c\left(\left|\nabla^{\alpha} u_2\right|, v_2\right)\right) \cdot \nabla u
        \right\|_{L^p_{\pi}(\Omega)}  \\
        \leq \, & C \sup_{t \in [0,T]}  
        \left\|\overline{v_1}_{\varepsilon}(t)^{\gamma} 
        b\left(\left|\nabla^{\alpha} u_1(t)\right|\right) 
        - \overline{v_2}_{\varepsilon}(t)^{\gamma} 
        b\left(\left|\nabla^{\alpha} u_2(t)\right|\right)\right\|_{\infty} \\
        &+ C \sum_{j=1}^N \sup_{t \in [0,T]} \left\|v_1(t)^{\gamma} 
        b\left(\left|\nabla^{\alpha} u_1(t)\right|\right) \right\|_{\infty} 
        \left\| \partial_{x_j} \left(v_1(t) - v_2(t)\right)\right\|_{\infty} \\
        &+ C \sum_{j=1}^N \sup_{t \in [0,T]} 
        \left\| \partial_{x_j} v_2(t) \right\|_{\infty} 
        \left\|v_1(t)^{\gamma-1} b\left(\left|\nabla^{\alpha} u_1(t)\right|\right) 
        - v_2(t)^{\gamma-1} b\left(\left|\nabla^{\alpha} u_2(t)\right|\right) \right\|_{\infty} \\
        &+ C \sum_{j=1}^N \sup_{t \in [0,T]} \left\|v_1(t)^{\gamma}\right\|_{\infty}
        \left\| \partial_{x_j} \left(b\left(\left|\nabla^{\alpha} u_1(t)\right|\right)
         - b\left(\left|\nabla^{\alpha} u_2(t)\right|\right)\right)\right\|_{\infty} \\
        &+ C \sum_{j=1}^N \sup_{t \in [0,T]} 
        \left\|\partial_{x_j} b \left(\left|\nabla^{\alpha} u_2(t)\right|\right) \right\|_{\infty}
        \left\|  v_1(t)^{\gamma} - v_2(t)^{\gamma}\right\|_{\infty} \\
        \leq \, & C \sup_{t \in [0,T]} 
        \left\| u_1(t) - u_2(t)\right\|_{W^{2-\frac{2}{p}, p}_{\pi}(\Omega)} 
        + C \sup_{t \in [0,T]} 
        \left\| v_1(t) - v_2(t)\right\|_{W^{2-\frac{2}{p}, p}_{\pi}(\Omega)} \\
        & + C \sum_{j=1}^N \left\|b'\left(\left|\nabla^{\alpha}u_1(t)\right|\right)\right\|_{\infty} 
        \left\|\partial_{x_j}\left(\left|\nabla^{\alpha}u_1(t)\right| 
        - \left|\nabla^{\alpha}u_2(t)\right|\right)\right\|_{\infty} \\
        & + C \sum_{j=1}^N \left\| \partial_{x_j}\left(\left|\nabla^{\alpha}u_2(t)\right|\right) \right\|_{\infty}
         \left\|b'\left(\left|\nabla^{\alpha}u_1(t)\right|\right)
        - b'\left(\left|\nabla^{\alpha}u_2(t)\right|\right)\right\|_{\infty} \\
        \leq \, & C \left\| u_1 - u_2\right\|_{C\left([0,T],W^{2-\frac{2}{p}, p}_{\pi}(\Omega)\right)} 
        + C \left\| v_1 - v_2\right\|_{C\left([0,T],W^{2-\frac{2}{p}, p}_{\pi}(\Omega)\right)} \\
        & + C \sup_{t \in [0,T]} \left\|\left|\nabla^{\alpha}u_1(t)\right| 
        - \left|\nabla^{\alpha}u_2(t)\right|\right\|_{C^{1,\rho_1}(\overline{\Omega})} \\
        \leq \, & C \left\| w_1 - w_2\right\|_{\mathcal{W}^{1,p}(0,T;(E_1,E_0))}.
    \end{align*}
    In the same way, we have  
    \[I_2 \leq C \left\| w_1 - w_2\right\|_{\mathcal{W}^{1,p}(0,T;(E_1,E_0))}\]
    and furthermore
    \[\|A(w_1)-A(w_2)\|_{L^{\infty}(0,T;\mathcal{L}(E_1,E_0))}  
    \leq \, C \left\| w_1 - w_2\right\|_{\mathcal{W}^{1,p}(0,T;(E_1,E_0))},\]
    which implies that 
    \begin{equation} \label{eqn:operator-matrix-bounded-lip-strong}
    A \in \mathcal{C}_{\mathrm{Volt}}^{0,1}\left(\mathcal{W}^{1,p}(0,T;(E_1,E_0)),
    \mathcal{MR}^p(0,T;(E_1,E_0))\right).
    \end{equation}
    The Lipschitz estimate of the source term is similar to that done in the 
    proof of Theorem \ref{thm:well-posedness-weak}. We have 
    \begin{equation} \label{eqn:source-term-bounded-lip-strong}
        F \in \mathcal{C}_{\mathrm{Volt}}^{0,1}\left(\mathcal{W}^{1,p}(0,T;(E_1,E_0)); 
        L^{\infty} \left(0,T;E_0\right), L^{p} \left(0,T;E_0\right)\right).
    \end{equation}
        
    It follows from \eqref{eqn:operator-matrix-bounded-lip-strong}, 
    \eqref{eqn:source-term-bounded-lip-strong} and 
    Theorem \ref{thm:amann-main} that there exist a maximal $T_{\max} \in \left(0,T_{0}\right]$ and 
    a unique solution $w=(u,v)^{\mathrm{T}} \in \mathcal{W}_{\mathrm{loc}}^{1,p}(0,T_{\max};(E_1,E_0))$ 
    of the quasilinear abstract Cauchy problem \eqref{eqn:quasilinear-abstract-strong}.  
    {Now set $\varepsilon = \inf_{x \in \Omega} f(x)$.} The minimum principle 
    $v \geq \varepsilon$ implies that $v = \overline{v}_{\varepsilon}$. 
    {At this point, the concept of solution to quasilinear abstract Cauchy 
    problem \eqref{eqn:quasilinear-abstract-strong} coincides with the concept 
    of strong solution to \eqref{eqn:main}.}
    Therefore, $(u,v)$ is the unique strong solution to problem \eqref{eqn:main} on 
    $[0, T_{\max})$, which concludes the proof.
\end{proof}

Now we present a regularity result for the strong solutions to problem \eqref{eqn:main}. 

\begin{theorem} \label{thm:regularity-strong}
    Let the assumptions of Theorem \ref{thm:well-posedness-strong} hold. 
    Let $(u,v)$ be the strong solution to problem \eqref{eqn:main} on 
    its maximal interval of existence $[0, T_{\max})$. Then there 
    exists $\rho \in (0,1)$ such that 
    $u, v \in C^{2+\rho, 1+\frac{\rho}{2}}(\overline{Q_T})$ for 
    every $T \in (0, T_{\max})$.
\end{theorem}

\begin{proof}
    We will accomplish the proof using the techniques provided in \cite[Section 14]{AMANN1993}. 
    Henceforth, $T$ can take any value in 
    $(0, T_{\max})$, $C>0$ denotes a constant which can take different
    value in different places. Theorem \ref{thm:well-posedness-strong} and  
    Amann's embedding theorem imply that 
    \[u, v \in L^p\left(0,T;W_{\pi}^{2,p}(\Omega)\right)   
    \cap W^{1,p}\left(0,T;L_{\pi}^{p}(\Omega)\right) 
    \cap C\left([0,T],W_{\pi}^{2-\frac{2}{p},p}(\Omega)\right).\]
    Utilizing H{\"o}lder's inequality and the interpolation inequality of Slobodeckij spaces, it can be inferred that
    \begin{align*}
    \|u(t)-u(s)\|_{W_{\pi}^{\left(2-\frac{2}{p}\right)\theta,p}(\Omega)}
    \leq \, & C \|u(t)-u(s)\|_{L_{\pi}^{p}(\Omega)}^{1-\theta}
    \|u(t)-u(s)\|_{W_{\pi}^{2-\frac{2}{p},p}(\Omega)}^{\theta} \\
    \leq \, & C \left(\int_s^t 1 d \tau\right)^{\frac{1-\theta}{p'}}
    \left(\int_s^t \left\|\dot{u}(\tau)\right\|_{L_{\pi}^{p}(\Omega)}^p 
    d \tau\right)^{\frac{1-\theta}{p}} \\
    \leq \, & C |t-s|^{\frac{(1-\theta)(p-1)}{p}}
    \end{align*}
    holds for all $0 < s \leq t < T$ and $0<\theta<1$, from which it follows that
    \[u \in C^{\frac{(1-\theta)(p-1)}{p}}\left([0,T],
    W_{\pi}^{2 \theta -\frac{2 \theta}{p},p}(\Omega)\right).\]
    Using arguments similar to those employed in the proofs of 
    Theorems \ref{thm:well-posedness-weak} and \ref{thm:well-posedness-strong}, 
    it can be deduced that
    \[u \in C^{\frac{(1-\theta)(p-1)}{p}}\left([0,T],
    C_{\pi}^{1,2 \theta - 1 - \frac{2\theta+N}{p}}(\overline{\Omega})\right)\]
    and 
    \[\left|\nabla^{\alpha}u\right| \in C^{\frac{(1-\theta)(p-1)}{p}}\left([0,T],
    C_{\pi}^{1,2 \theta - 1 - \frac{2\theta+N}{p} - \alpha-\delta}(\overline{\Omega})\right)\]
    for $\theta \in \left(\frac{N+(\alpha+1)p}{2(p-1)},1\right)$ and 
    $\delta \in \left(0,2 \theta - 1 - \frac{2\theta+N}{p} - \alpha\right)$. Similarly,
    \[v \in C^{\frac{(1-\theta)(p-1)}{p}}\left([0,T],
    C_{\pi}^{1,2 \theta - 1 - \frac{2\theta+N}{p}}(\overline{\Omega})\right).\]
    Set 
    \[\rho_1 = \frac{(1-\theta)(p-1)}{p}, \ 
    \rho_2 =2 \theta - 1 - \frac{2\theta+N}{p}, \
    \rho_3 = 2 \theta - 1 - \frac{2\theta+N}{p} - \alpha-\delta, \]
    we choose $\rho = \min\{1,\mu-1\} \min\{\rho_1, \rho_3\}$ if 
    $\gamma=1$, and $\rho = \min\{1,\gamma-1,\mu-1\} \underset{i=1,2,3}{\min}~ \rho_i$ 
    if $\gamma>1$, then
    \[\partial_{x_j} c\left(\left|\nabla^{\alpha} u\right|,v\right), \
    \partial_{x_j} a(u,v), \
    -\lambda K'(Ku-f)
    \in C^{\rho,\frac{\rho}{2}}\left(\overline{Q_T}\right)\]
    since H{\"o}lder spaces {on $\overline{\Omega}$} are algebras. Rewrite
    \[\begin{array}{c}
        c\left(\left|\nabla^{\alpha} u\right|,v\right) \coloneqq a_{11}(t,x), \ 
        a(u,v) \coloneqq a_{22}(t,x), \ 
        -\lambda K'(Ku-f) \coloneqq f_1(t,x), \\[1ex]
        \nabla c\left(\left|\nabla^{\alpha} u\right|,v\right) \coloneqq \mathbf{b_1}(t,x), \ 
        \nabla a(u,v) \coloneqq \mathbf{b_2}(t,x), 
    \end{array}\]
    then $u,v$ {turn out to satisfy a.e. in $Q_T$ the following}
    linear parabolic system {formulated for $\widetilde{u}$ and $\widetilde{v}$:}
    \begin{equation}
        \left\{
        \begin{array}{ll}
            \displaystyle \vspace{0.2em} 
            \widetilde{u}_t = 
            a_{11}(t,x) \Delta \widetilde{u} + \mathbf{b_1}(t,x) \cdot \nabla \widetilde{u} 
            - f_1(t,x), & \textrm{in~~} (0,T_0) \times \Omega,  \\ 
            \widetilde{v}_t = 
            a_{22}(t,x) \Delta \widetilde{v} + \mathbf{b_2}(t,x) \cdot \nabla \widetilde{v} 
            , & \textrm{in~~} (0,T_0) \times \Omega,   \\
            \widetilde{u}, \widetilde{v} \textrm{~~periodic on} \Omega & \textrm{for~~} (0,T_0), \\
            \widetilde{u}(0,\cdot) = \widetilde{v}(0,\cdot) = f, & \textrm{on~~} \Omega, 
        \end{array} \label{eqn:main-linearized}
        \right. 
    \end{equation}
    {since $(u,v)$ is a strong solution to problem \eqref{eqn:main}. 
    System \eqref{eqn:main-linearized}} has H{\"o}lder continuous coefficients 
    and a H{\"o}lder continuous right-hand side. 
    It follows from classical results \cite[Section 14-18]{SOLONNIKOV1965} 
    that \eqref{eqn:main-linearized} possesses a unique classical solution 
    {$u^{\bullet}, v^{\bullet}$} $\in C^{2+\rho, 1+\frac{\rho}{2}}(\overline{Q_T})$.
    {Since both $(u,v)$ and $(u^{\bullet}, v^{\bullet})$ satisfy \eqref{eqn:main-linearized} in $Q_T$, 
    we conclude that $u = u^{\bullet}$, $v = v^{\bullet}$}. 
    Therefore, $u,v \in C^{2+\rho, 1+\frac{\rho}{2}}(\overline{Q_T})$.
\end{proof}

\begin{corollary} \label{cor:regularity-infty}
    Assume that $\frac{\beta}{2}, \frac{\mu}{2}, \gamma$ are integers,
    $0<f \in C_{\pi}^{\infty}(\overline{\Omega})$. 
    Let $(u,v)$ be the strong solution to problem \eqref{eqn:main} on 
    its maximal interval of existence $[0, T_{\max})$. Then 
    $u, v \in C^{\infty}(\overline{Q_T})$ for 
    every $T \in (0, T_{\max})$. 
\end{corollary}

\begin{proof}
    The bootstrapping argument used in the proof of Theorem \ref{thm:regularity-strong}
    can be iterated in order to see that for $k \geq 2$,
    if $f \in  W^{k-\frac{2}{p},p}(\Omega)$, 
    then there exists $\rho \in (0,1)$ such that 
    $u,v \in C^{k+\rho, \frac{k+\rho}{2}}(\overline{Q_T})$ by means of classical results 
    \cite[Section 14-18]{SOLONNIKOV1965}, as repeated 
    differentiations can be achieved under the assumption. 
    The arbitrariness of $k$ concludes the proof.
\end{proof}

\section{Numerical experiments}\label{sec5}

\begin{figure}[!htbp]
    \centering
    \begin{minipage}[t]{.2\linewidth}
        \centering
        \includegraphics[width=\textwidth]{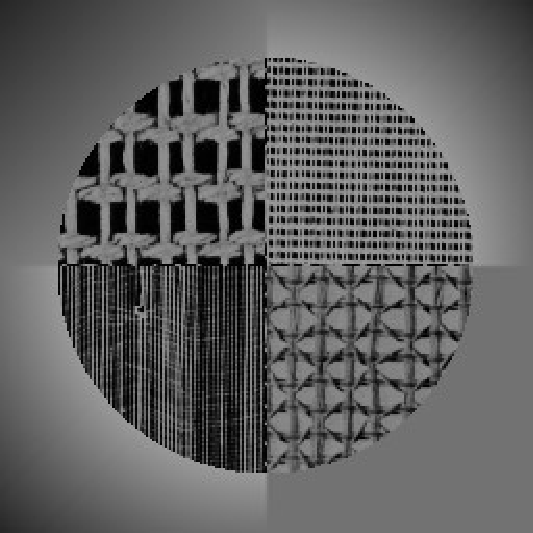}
        \subcaption{{Original}} \label{hybrid-original0}
    \end{minipage} 
    \begin{minipage}[t]{.2\linewidth}
        \centering
        \includegraphics[width=\textwidth]{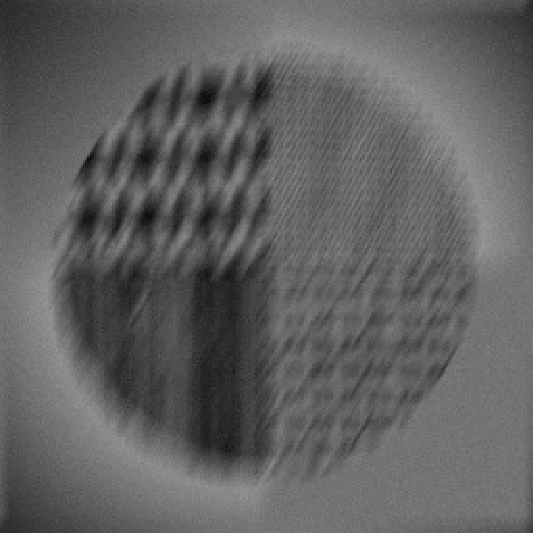}
        \subcaption{{Blurred}} \label{hybrid-blurred0}
    \end{minipage}
    \begin{minipage}[t]{.2\linewidth}
        \centering
        \includegraphics[width=\textwidth]{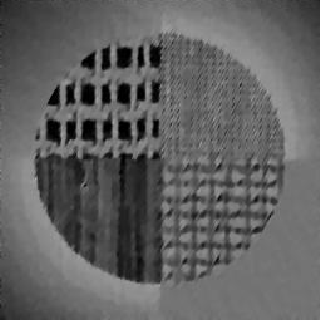}
        \subcaption{{$\lambda=5$}} \label{hybrid5l}
    \end{minipage}

    \begin{minipage}[b]{.2\linewidth}
        \centering
        \includegraphics[width=\textwidth]{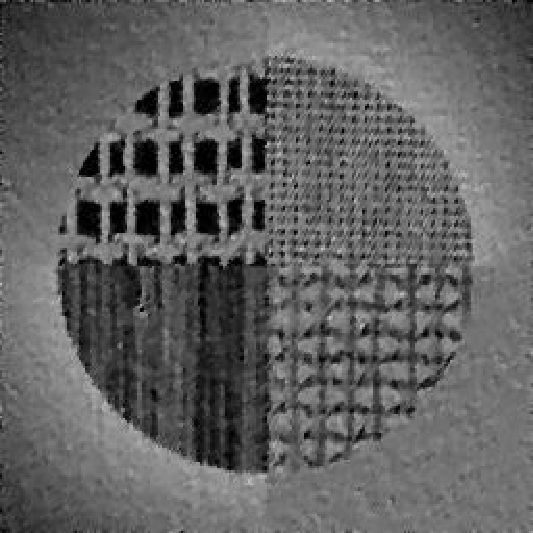}
        \subcaption{{$\lambda=15$}} \label{hybrid-Ours0}
    \end{minipage}
    \begin{minipage}[b]{.2\linewidth}
        \centering
        \includegraphics[width=\textwidth]{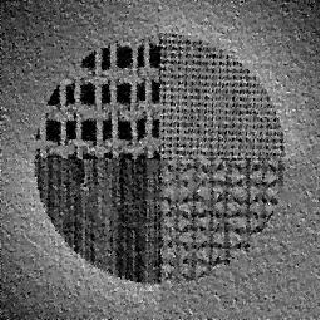}
        \subcaption{{$\lambda=50$}} \label{hybrid50l}
    \end{minipage}
    \begin{minipage}[b]{.2\linewidth}
        \centering
        \includegraphics[width=\textwidth]{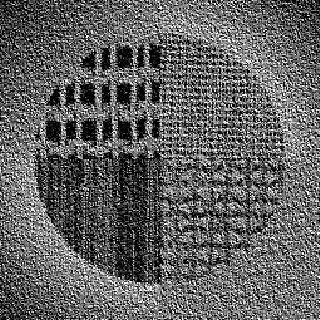}
        \subcaption{{$\lambda=500$}} \label{hybrid500l}
    \end{minipage}
    \caption{{Recovery results of Algorithm \ref{alg:alg1} on a hybrid 
    image with motion blur and corrupted by 
    the noise of standard deviation $\sigma = 3$ under different $\lambda$ values.
    (a) original image. (b) noisy blurred image. 
    (c)-(f) recovered images with $\lambda=5, 15, 50$ and $500$, respectively.}} \label{hybrid0}
\end{figure}

In this section, we present several numerical experiment examples to 
illustrate the effectiveness of the proposed model in image restoration. 
Firstly, a numerical discretization of \eqref{eqn:main} is derived. 
Assume that the discrete image to be $I \times  I$ pixels, $\tau$ to be 
the time step size and $h$ the space grid size. Then the equidistant 
spatio-temporal grid is given by
\[\mathcal{T}_{\tau,h}=\left\{(t_n,x_i,y_j): t_n=n\tau, x_i=ih, y_j=jh,
n = 0,1,2, \cdots, \ i, j = 0,1, \cdots, I-1\right\}.\]
Denote by $u^n = \{u_{i,j}^n\}_{\forall i, j}$ the grid function
at time $t_n$, which approximates the values of $u(t_n,\cdot)$ at grid points. 
Some other notations and assumptions are 
given for the following numerical scheme. 
\[\begin{array}{l}
    u^0_{i,j} = f_{i,j} \coloneqq f(x_i,y_j),\\[1ex]
    u^n_{i,0} = u^n_{i,I},\  u^n_{0,j} = u^n_{I,j},
    u^n_{i,-1} = u^n_{i,I-1},\  u^n_{-1,j} = u^n_{I-1,j},\\[1ex]
    {D_x^{\pm}} u^n_{i,j} \coloneqq \pm \frac{u^n_{i \pm 1,j} - u^n_{i,j}}{h},\  
    {D_y^{\pm}} u^n_{i,j} \coloneqq \pm \frac{u^n_{i, j\pm 1} - u^n_{i,j}}{h}, \\[1ex]
    \delta_x^2 u^n_{i,j} \coloneqq u^n_{i+1,j} - 2 u^n_{i,j} + u^n_{i-1,j}, \ 
    \delta_y^2 u^n_{i,j} \coloneqq u^n_{i,j+1} - 2 u^n_{i,j} + u^n_{i,j-1}. 
\end{array}\]
Similar notations and assumptions are used for $v$ and other functions. 
Denote by $\hat{g}$ the 2-D discrete Fourier transform (DFT) of a 
grid function ${g}$ and $\mathrm{F}^{-1}$ the 2-D inverse 
discrete Fourier transform (IDFT) operator.
To approximate $|\nabla^{\alpha} u^n|$, we use a central difference scheme 
provided by \cite{BAI2007}. The fractional-order difference can be defined as
\begin{equation} \label{eqn:fractional-order-difference}
    \begin{array}{l}
D_{x}^{\alpha}u^n = \mathrm{F}^{-1} \left[\left(1-e^{-\mathbbm{i} \omega_1 h}\right)^{\alpha}
e^{\mathbbm{i} \alpha \omega_1 \frac{h}{2}}  \hat{u}^{n}(\omega_1, \omega_2)\right], \\[1.5ex]
D_{y}^{\alpha}u^n = \mathrm{F}^{-1} \left[\left(1-e^{-\mathbbm{i} \omega_2 h}\right)^{\alpha}
e^{\mathbbm{i} \alpha \omega_2 \frac{h}{2}}  \hat{u}^{n}(\omega_1, \omega_2)\right],
\end{array}
\end{equation}
and the discrete fractional-order 
gradient $\nabla^{\alpha}_{h} {u^n} = (D_{x}^{\alpha}u^n, D_{y}^{\alpha}u^n)$.
{In this section, we focus on the case where $K$ is a 
convolution operator, namely $Ku = k * u$, where $k$ denotes the convolution kernel.}
The numerical approximation of 
\eqref{eqn:main} could be derived by local linearization method.
Denote by $a^n = a(u^n,v^n)$ the values at
the grid points of the diffusion coefficient $a(\cdot,\cdot)$.
If we have a numerical solution $(u^n, v^n)$, 
an explicit finite difference scheme 
\begin{equation} \label{eqn:explicit-scheme-v}
    \frac{v_{i,j}^{n+1}-v_{i,j}^n}{\tau} = \, {D_x^-} 
    \left(a_{i,j}^n {D_x^+} v_{i,j}^{n}\right) + 
    {D_y^-} \left(a_{i,j}^n {D_y^+} v_{i,j}^n\right), 
\end{equation}
can be used to obtain $v^{n+1}$,
then we can use a semi-implicit scheme
\begin{align} \label{eqn:semi-implicit-scheme-u}
    \frac{u_{i,j}^{n+1}-u_{i,j}^n}{\tau} = \, & {D_x^-} 
        \left(c_{i,j}^n {D_x^+} u_{i,j}^{n}\right) + {D_y^-} 
        \left(c_{i,j}^n {D_y^+} u_{i,j}^n\right) \notag \\
        &- \lambda \left(\mathbf{K}' * \mathbf{K} * u^{n+1}\right)_{i,j} 
        + \lambda \left(\mathbf{K}' * f\right)_{i,j},
\end{align}
to obtain $u^{n+1}$. Here $c^n = c\left(|\nabla^{\alpha}_{h} u^n|, v^{n+1}\right)$, 
$\nabla^{\alpha}_{h}$ is the discrete fractional-order 
gradient, $\mathbf{K} =(k_{i,j})$ is a discrete convolution kernel with 
adjoint $\mathbf{K}'$, which is obtained from the convolution
kernel $k$, i.e. $k_{i,j}=k(ih,jh)$. The function $f$ here is only defined on the 
grid, but we do not denote it. 

In the numerical implementation of the proposed model,
the explicit scheme \eqref{eqn:explicit-scheme-v} can 
be directly computed, and the semi-implicit scheme 
\eqref{eqn:semi-implicit-scheme-u} can 
be solved using the DFT and IDFT. The 
image restoration process based on the proposed model 
is summarized as Algorithm \ref{alg:alg1}. {The symbol $\odot$ appearing 
in step 10 denotes element-wise (pointwise) multiplication, i.e. the Hadamard product.}

\begin{algorithm} 
    \caption{The restoration algorithm based on the proposed model}
    \label{alg:alg1} 
    \begin{algorithmic}[1]
        \renewcommand{\algorithmicrequire}{\textbf{Input:}}
		\REQUIRE {Degraded} image $f$, $\tau$, $\alpha$, $\beta$, $k_1$, $\lambda$, 
        $\mathbf{K}$, $\gamma$, $\mu$, $\lambda_1$, $tol$, $maxIter$
        \renewcommand{\algorithmicrequire}{\textbf{Initialize:}}
        \REQUIRE $u^0 = v^0 = f$, $n = 0$
        \renewcommand{\algorithmicensure}{\textbf{Output:}}
		\ENSURE $u^n$
        \STATE Compute $M = \displaystyle \max_{i,j} f_{i,j}$
        \FOR {each image $u^n$, $v^n$}
        \STATE Compute $a^n = \lambda_1 \frac{|u^n|^{\mu}}{M^{\mu}} 
        + \left(1-\lambda_1\right) \frac{|v^n|^{\gamma}}{M^{\gamma}}$ \label{ste:begin}
        \STATE Compute $v^{n+1}$ by \eqref{eqn:explicit-scheme-v} 
        \STATE Compute $D_{x}^{\alpha}u^n$ and $D_{y}^{\alpha}u^n$ 
        by \eqref{eqn:fractional-order-difference}
        \STATE Compute $c^n = \frac{|v^{n+1}|^{\gamma}}{M^{\gamma}} 
        \frac{1}{1 + k_1 \left(D_{x}^{\alpha}u^n + D_{y}^{\alpha}u^n\right)^{\beta / 2}}$
        \FOR {$i,j = 0,1, \cdots, I-1$}
        \STATE Compute {\small $ d_{i,j}^n \coloneqq u_{i,j}^n +
        \tau  \left({D_x^-}\left(c_{i,j}^n {D_x^+} u_{i,j}^{n}\right) + {D_y^-} 
        \left(c_{i,j}^n {D_y^+} u_{i,j}^n\right)\right) +
        \tau \lambda \left(\mathbf{K}' * f\right)_{i,j}$} 
        \ENDFOR
        \STATE Compute $u^{n+1} = \mathrm{F}^{-1} \left[\frac{\hat{d^n}}
        {\hat{I} + \tau \lambda \hat{\mathbf{K}}' {\odot} \hat{\mathbf{K}}}\right]$ \label{ste:end}
        \IF{$\frac{\|u^{n+1} - {u^n}\|_2}{\|u^{n+1}\|_2} > tol$ \AND $i < maxIter$}
        \STATE Set $n = n + 1$
        \STATE Repeat \ref{ste:begin}-\ref{ste:end} steps for $u^n$ and $v^n$
        \ENDIF
        \ENDFOR
	\end{algorithmic}
\end{algorithm}

In order to quantify the restoration effect, for the original image {$u^{\dagger}$}
and the compared image $u$, the restoration performance is measured in terms of the peak 
signal noise ratio (PSNR) 
\[\mathrm{PSNR}(u,{u^{\dagger}}) = 10 \log_{10} \left(\dfrac{\Sigma_{i,j} 255^2}
    {\Sigma_{i,j} (u_{i,j} - {u^{\dagger}_{i,j}})^2}\right), \] 
and the structural similarity index measure (SSIM)
\[\mathrm{SSIM}(u,{u^{\dagger}}) = \dfrac{(2\mu_u \mu_{u^{\dagger}} +c_1)(2 \sigma_{u{u^{\dagger}}} + c_2)}
    {\mu_u^2 + \mu_{{u^{\dagger}}}^2 + c_1 + \sigma_u^2 + \sigma_{{u^{\dagger}}}^2 + c_2},\]
where  $c_1, c_2$ are two variables to stabilize the division with weak denominator, 
$\mu_u, \mu_{{u^{\dagger}}}, \sigma_u, \sigma_{{u^{\dagger}}}$ and $\sigma_{u{u^{\dagger}}}$ are the local means, 
standard deviations and cross-covariance for image $u,{u^{\dagger}}$, respectively. The 
better quality image will have higher values of PSNR and SSIM.

\begin{figure}[!htbp] 
    \centering
    \begin{minipage}[b]{.44\linewidth}
        \centering
        \includegraphics[width=\textwidth]{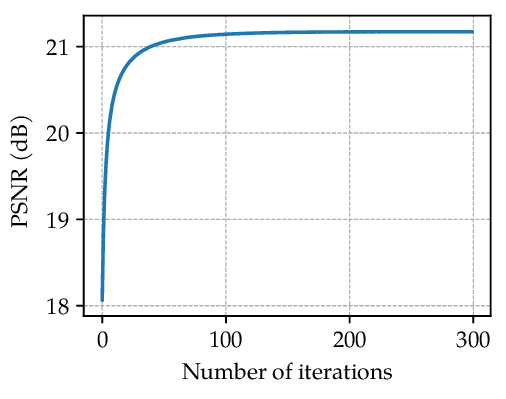} 
        \subcaption{} \label{fig:rd_psnr}
    \end{minipage} 
    \begin{minipage}[b]{.548\linewidth}
        \centering
        \includegraphics[width=\textwidth]{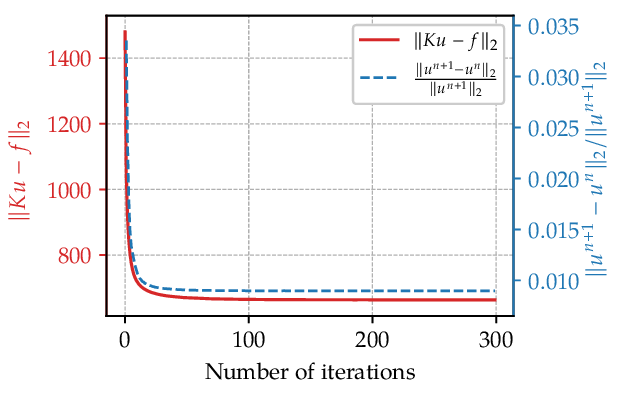} 
        \subcaption{} \label{fig:rd_err_rel}
    \end{minipage}
    \caption{{Evolution of PSNR, the data fidelity and the relative error with 
    respect to the iteration number in Algorithm \ref{alg:alg1}, corresponding 
    to the case $\lambda = 15$ in Fig. \ref{hybrid0}.}} \label{fig:psnr-err-rel}
\end{figure}

\begin{table}[htbp]
\centering
{\begin{tabular}{cccc}
\hline
 & $\lambda$ & PSNR/SSIM (noisy) & PSNR/SSIM (recovered) \\
\hline
$\sigma = 3$ & 15  & 18.07/0.4382 & 21.17/0.6281 \\
$\sigma = 5$ & 4.5 & 18.01/0.3738 & 20.38/0.6155 \\
$\sigma = 7$ & 2.5 & 17.90/0.3091 & 19.99/0.5905 \\
\hline
\end{tabular}}
\caption{{Comparison of the optimal $\lambda$ values and the corresponding PSNRs/SSIMs at different noise 
levels $\sigma$ for the experiments shown in Figs. \ref{hybrid0} and \ref{hybrid-5-7}.}}
\label{tab:psnr-ssim-5-7}
\end{table}

\begin{figure}[!htbp]
    \centering
    \begin{minipage}[t]{.2\linewidth}
        \centering
        \includegraphics[width=\textwidth]{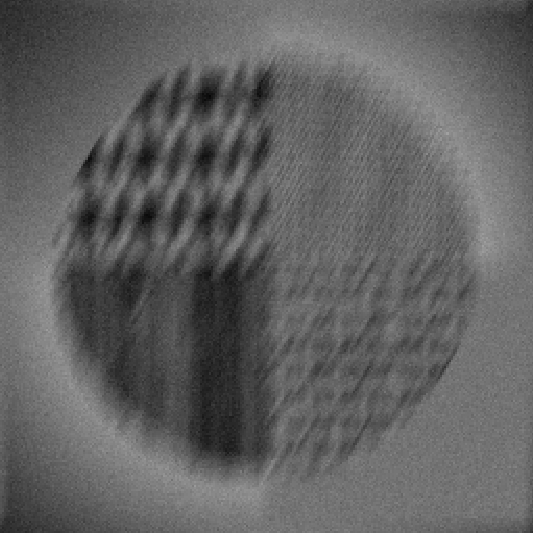}
    \end{minipage} 
    \begin{minipage}[t]{.2\linewidth}
        \centering
        \includegraphics[width=\textwidth]{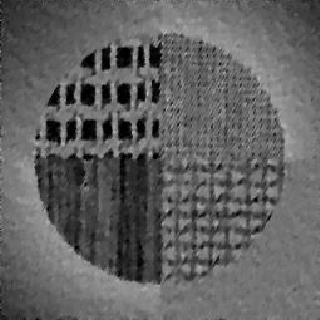}
    \end{minipage}
    \begin{minipage}[t]{.2\linewidth}
        \centering
        \includegraphics[width=\textwidth]{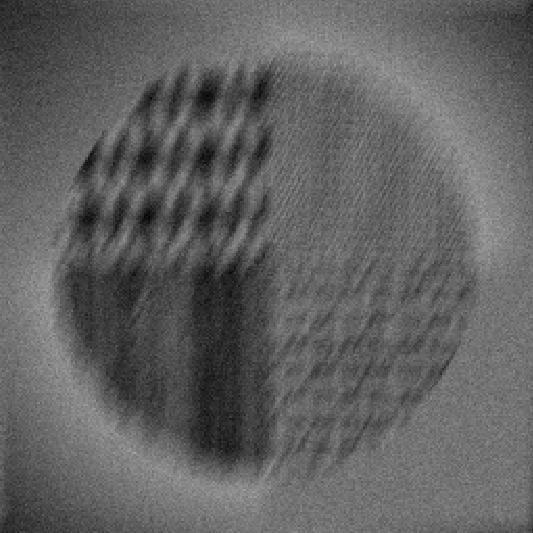}
    \end{minipage}
    \begin{minipage}[t]{.2\linewidth}
        \centering
        \includegraphics[width=\textwidth]{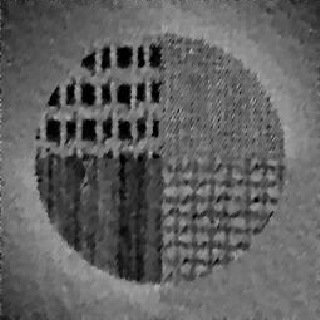}
    \end{minipage}
    \caption{{Noisy blurred images and the corresponding restorations by 
    Algorithm \ref{alg:alg1} with noise levels $\sigma = 5$ (left two columns) 
    and $\sigma = 7$ (right two columns).}} \label{hybrid-5-7}
\end{figure}

{We apply Algorithm \ref{alg:alg1} to a hybrid test image
to demonstrate the smoothing effect induced by 
the diffusion term in the proposed model, as well as the influence of the regularization 
parameter $\lambda$. As shown in Fig. \ref{hybrid-blurred0}, the test image 
(originally shown in Fig. \ref{hybrid-original0}) is blurred by a motion blur kernel with an 
angle of $\frac{\pi}{3}$ and a length of $20$, and further distorted by additive isotropic 
Gaussian noise with a standard deviation of $\sigma = 3$. In this experiment, we set $h = 1$, 
and use the parameter configuration $\tau = 0.5$, $\alpha = \lambda_1 = 0.9$, 
$\beta = \gamma = k_1 = 1$ and $\mu = 0.4$ in Algorithm \ref{alg:alg1}, with $300$ iterations performed. 
Figures \ref{hybrid5l}-\ref{hybrid500l} show the restoration results obtained with different 
values of $\lambda$. It can be observed that when $\lambda = 5$, the image becomes 
over-smoothed, with the central textured region completely flattened under the smoothing 
effect of diffusion. When $\lambda = 15$, the algorithm yields the best visual 
restoration. For $\lambda = 50$, textures appear sharper but the image becomes noticeably 
noisier, reducing perceived quality. As $\lambda$ increases to $500$, the image is largely 
buried in noise. Similar behavior is often seen when minimizing the energy 
\eqref{eqn:energy}; this is consistent with viewing the diffusion and reaction 
terms in system \eqref{eqn:main} as playing roles similar to the regularization 
and fidelity terms in \eqref{eqn:energy}. In practice, choosing $\lambda$ appropriately 
helps balance the contributions of these two terms and can prevent 
over-smoothing or over-fitting \cite{LANGER2017}. Figure \ref{fig:psnr-err-rel} presents the convergence 
behavior of Algorithm~\ref{alg:alg1} for the experiment with $\lambda = 15$. As 
shown in Fig. \ref{fig:rd_psnr}, PSNR increases rapidly at early iterations and 
then levels off. Fig. \ref{fig:rd_err_rel} displays the evolution of the data-fidelity 
residual $\|Ku-f\|_2$ and the relative iteration 
error $\frac{\|u^{n+1}-u^n\|_2}{\|u^{n+1}\|_2}$, both of which decay quickly at first and 
subsequently stabilize.}

{Next, keeping the blur kernel fixed and increasing the added noise to 
standard deviation $\sigma=5$ or $\sigma=7$, we examine how the value 
of $\lambda$ required for best visual restoration 
changes (all other parameters in Algorithm \ref{alg:alg1} remain the same). 
Table \ref{tab:psnr-ssim-5-7} reports the optimal $\lambda$ values and the 
corresponding PSNR and SSIM for $\sigma=3,5,7$. One observes that as the noise 
level increases, $\lambda$ needs to be reduced in order to achieve better 
visual quality without overfitting; this behavior is similar to the empirical observations \cite{GALATSANOS1992}
made when minimizing the energy \eqref{eqn:energy}. 
Figure \ref{hybrid-5-7} shows the noisy blurred images 
for $\sigma=5$ and $\sigma=7$ and the corresponding restorations produced 
by Algorithm \ref{alg:alg1} using the $\lambda$ values listed in 
Table \ref{tab:psnr-ssim-5-7}.}

\begin{figure}[htbp] 
    \centering
    \begin{minipage}[b]{.3278\linewidth}
        \centering
        \includegraphics[width=\textwidth]{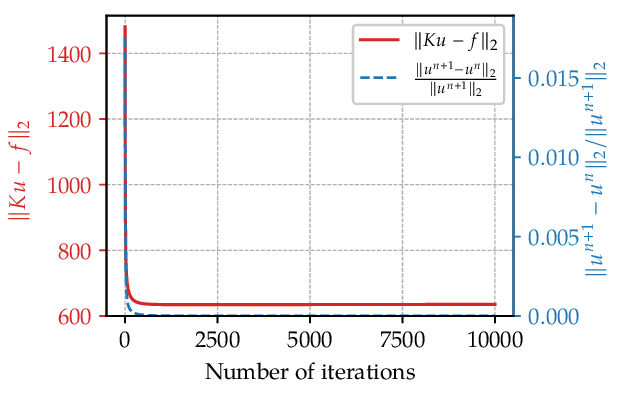} 
        \subcaption{} \label{fig:rd_err_rel_0.1_10000}
    \end{minipage}
    \begin{minipage}[b]{.3278\linewidth}
        \centering
        \includegraphics[width=\textwidth]{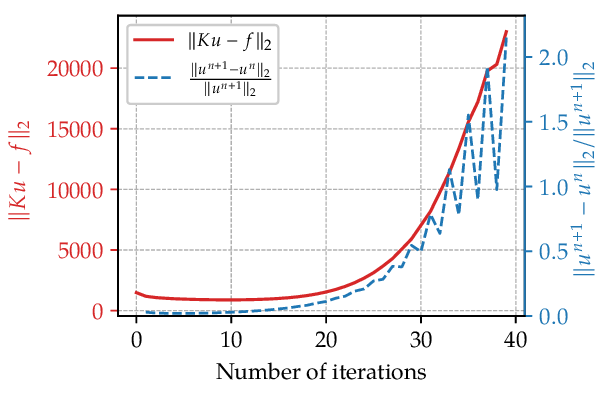} 
        \subcaption{} \label{fig:rd_err_rel_0.15_40}
    \end{minipage}
    \begin{minipage}[b]{.3278\linewidth}
        \centering
        \includegraphics[width=\textwidth]{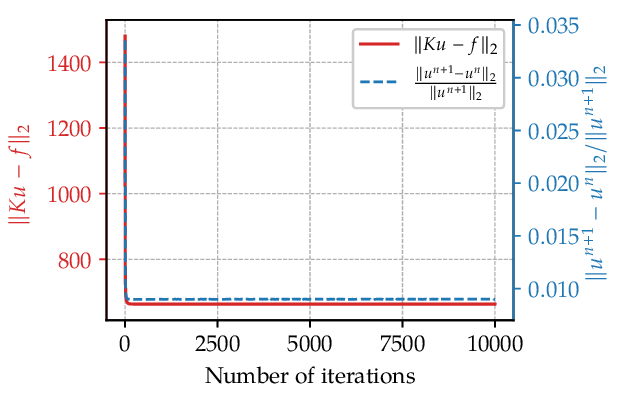} 
        \subcaption{} \label{fig:rd_err_rel_0.5_10000}
    \end{minipage}
    \caption{{Evolution of the data fidelity and the relative error 
    with respect to the iteration number for the restoration of 
    Fig. \ref{hybrid-blurred0} using the proposed model with $\lambda=15$. 
    (a) Explicit scheme \eqref{eqn:explicit-scheme-u}, $\tau=0.1$, $10 \ 000$ iterations. 
    (b) Explicit scheme \eqref{eqn:explicit-scheme-u}, $\tau=0.15$, $40$ iterations. 
    (c) Semi-implicit scheme \eqref{eqn:semi-implicit-scheme-u}, $\tau=0.5$, 
    $10 \ 000$ iterations.}} \label{fig:err-rel}
\end{figure}

{Next we discuss the stability of the numerical scheme for computing $u^n$.
We first consider the stability condition for the semi-implicit scheme
\eqref{eqn:semi-implicit-scheme-u}.}
Freezing $c_{i,j}^n$ as a constant $c$, consider the following constant 
    coefficient difference equation
    \begin{equation} \label{eqn:frozen-scheme-u}
        u_{i,j}^{n+1} = u_{i,j}^{n} 
        + \frac{\tau {c}}{h^2}\left(\delta_x^2 u_{i,j}^{n} + \delta_y^2 u_{i,j}^{n}\right)
        - \tau \lambda \left(\mathbf{K}' * \mathbf{K} * u^{n+1}\right)_{i,j},
    \end{equation}
    whose stability conditions can be obtained by von Neumann analysis. 
    {The term $\lambda(\mathbf{K}'*f)_{i,j}$ in 
    \eqref{eqn:semi-implicit-scheme-u} is omitted because, by Duhamel's principle, 
    the nonhomogeneous term does not alter the amplification 
    factor (see \cite[Section 9.3]{STRIKWERDA2004}).} The values of 
    the grid function $u^n$ at each grid point can be represented by its DFT:
    \begin{equation} \label{eqn:fourier-u}
        u_{i,j}^{n} = {\frac{1}{I^2}} \sum_{p=0}^{I-1} \sum_{q=0}^{I-1} \hat{u}^{n}(\omega_1, \omega_2) 
    e^{\mathbbm{i}(\omega_1 x_i + \omega_2 y_j)}, \quad \omega_1 = \frac{2 \pi p}{I {h}}, \ 
    \omega_2 = \frac{2 \pi q}{I{h}}.
    \end{equation}
    Noting that the DFT of $\mathbf{K}'$ always equals the complex conjugate of 
    the DFT of $\mathbf{K}$, utilizing the convolution property of the DFT, we have
    \begin{equation} \label{eqn:fourier-convolution}
        \left(\mathbf{K}' * \mathbf{K} * u^{n+1}\right)_{i,j} 
    = {\frac{1}{I^2}} \sum_{p=0}^{I-1} \sum_{q=0}^{I-1} \hat{u}^{n+1}(\omega_1, \omega_2) 
    |\hat{\mathbf{K}}(\omega_1, \omega_2)|^2 
    e^{\mathbbm{i}(\omega_1 x_i + \omega_2 y_j)}.
    \end{equation}
    Substituting \eqref{eqn:fourier-u} and \eqref{eqn:fourier-convolution} 
    into \eqref{eqn:frozen-scheme-u} and comparing coefficients, we obtain that
    \begin{align} \label{eqn:fourier-frozen-scheme-u}
        \hat{u}^{n+1}(\omega_1, \omega_2) =& \left[1-\frac{4 \tau c}{h^2} 
        \left(\sin^2\frac{\omega_1 h}{2}+\sin^2\frac{\omega_2 h}{2}\right)\right]
        \hat{u}^{n}(\omega_1, \omega_2) \notag \\
        &-\tau \lambda |\hat{\mathbf{K}}(\omega_1, \omega_2)|^2 
        \hat{u}^{n+1}(\omega_1, \omega_2).
    \end{align}
    Then the amplification factor can be derived by simple calculation:
    \[G(\omega_1,\omega_2) = \frac{\hat{u}^{n+1}(\omega_1, \omega_2)}
    {\hat{u}^{n}(\omega_1, \omega_2)} = \frac{1-\frac{4c \tau}{h^2} 
    \left(\sin^2\frac{\omega_1 h}{2}+\sin^2\frac{\omega_2 h}{2}\right)}
    {1+\tau \lambda |\hat{\mathbf{K}}(\omega_1, \omega_2)|^2}.\]
    {The strict stability of the frozen-coefficient scheme \eqref{eqn:frozen-scheme-u} in 
    the $\ell^2$ norm requires} the strict von Neumann 
    condition $|G(\omega_1,\omega_2)| \leq 1$ {for all $\omega_1,\omega_2$. 
    $G(\omega_1,\omega_2) \leq 1$ holds automatically, whereas 
    ensuring $G(\omega_1,\omega_2)\geq -1$ requires 
    \[\frac{4 \tau c}{h^2} \left(\sin^2\frac{\omega_1 h}{2}+\sin^2\frac{\omega_2 h}{2}\right)
    \leq 2 + \tau \lambda |\hat{\mathbf{K}}(\omega_1, \omega_2)|^2 \]
    for all $\omega_1,\omega_2$. 
    In particular this must hold at the Nyquist frequency 
    point $(\omega_1^*,\omega_2^*)=(\frac{\pi}{h},\frac{\pi}{h})$, so we require
    $\frac{8 \tau c}{h^2} \leq 2 + \tau \lambda K_{\mathrm{m}}$, 
    where $K_{\mathrm{m}} \coloneqq |\hat{\mathbf{K}}(\omega_1^*,\omega_2^*)|^2$.}
    Therefore, 
    \[{\frac{8 \tau}{h^2} \max_{n,i,j} c\left(|\nabla^{\alpha}_{h} u^n|, v^{n+1}\right) \leq 2 + \tau \lambda K_{\mathrm{m}}}\]
    becomes the necessary condition to ensure that the 
    scheme \eqref{eqn:semi-implicit-scheme-u} is strictly stable 
    {in the $\ell^2$ norm.}

{On the other hand, by a similar procedure one obtains} for the explicit scheme
\begin{equation} \label{eqn:explicit-scheme-u}
    \frac{u_{i,j}^{n+1}-u_{i,j}^n}{\tau} = {D_x^-}
    \left(c_{i,j}^n {D_x^+} u_{i,j}^{n}\right) + {D_y^-}
    \left(c_{i,j}^n {D_y^+} u_{i,j}^n\right) 
    - \lambda \left(\mathbf{K}' * (\mathbf{K} * u^{n} - f)\right)_{i,j}
\end{equation} 
{that the frozen-coefficient scheme has the amplification factor}
\[{\widetilde{G}(\omega_1,\omega_2) = 1-\frac{4c \tau}{h^2} 
\left(\sin^2\frac{\omega_1 h}{2}+\sin^2\frac{\omega_2 h}{2}\right) 
-\tau \lambda |\hat{\mathbf{K}}(\omega_1, \omega_2)|^2.}\]
{Requiring $|\widetilde{G}(\omega_1,\omega_2)|\leq 1$ yields}
\[{\frac{4 \tau c}{h^2} \left(\sin^2\frac{\omega_1 h}{2}+\sin^2\frac{\omega_2 h}{2}\right) 
+ \tau \lambda |\hat{\mathbf{K}}(\omega_1, \omega_2)|^2 \leq 2}\] 
{for all $\omega_1,\omega_2$. Besides the restriction at the Nyquist 
frequency point}
\[{\frac{8 \tau}{h^2} \max_{n,i,j} c\left(|\nabla^{\alpha}_{h} u^n|, v^{n+1}\right) 
+ \tau \lambda K_{\mathrm{m}} \leq 2,}\]
{strict stability of the explicit scheme in the $\ell^2$ norm also 
requires the constraint at the zero-frequency point $(\omega_1,\omega_2)=(0,0)$,}
\begin{equation} \label{eqn:explicit-scheme-u-condition-2}
{\tau\lambda\bigl|\widehat{\mathbf{K}}(0,0)\bigr|^2 \le 2,}
\end{equation}
which imposes stricter restriction on {the choice of} the time step size. 
{It is well known that $\widehat{\mathbf{K}}(0,0)=\sum_{i,j} k_{i,j}$, 
and the discrete convolution kernel is usually assumed 
to satisfy $\sum_{i,j} k_{i,j}=1$ (see \cite{CHAN2005,AUBERT1997}), 
so \eqref{eqn:explicit-scheme-u-condition-2} reduces to $\tau \le \frac{2}{\lambda}$. 
For example, when restoring Fig.~\ref{hybrid-blurred0} with $\lambda=15$, 
choosing $\tau=0.1$ and using the explicit scheme \eqref{eqn:explicit-scheme-u} yields 
stable behavior after $10 \ 000$ iterations, as shown in Fig. \ref{fig:rd_err_rel_0.1_10000}, 
and produces a visually convincing restoration (see Fig. \ref{fig:hybrid_0.1_10000}). However, 
when $\tau$ is increased to $0.15 > \frac{2}{\lambda} = \frac{2}{15} \approx 0.133$, 
the iteration becomes clearly unstable shortly after it starts, 
as in Fig. \ref{fig:hybrid_0.15_40}, and the restoration is 
rendered meaningless (see Fig. \ref{fig:rd_err_rel_0.15_40}). By contrast, using 
the semi-implicit scheme \eqref{eqn:semi-implicit-scheme-u} with $\tau=0.5$, the iteration 
remains stable after $10 \ 000$ iterations (Fig. \ref{fig:rd_err_rel_0.1_10000}) and yields 
a good visual restoration result (Fig. \ref{fig:hybrid_0.5_10000}).}

\begin{figure}[htbp] 
    \centering
    \begin{minipage}[b]{.2\linewidth}
        \centering
        \includegraphics[width=\textwidth]{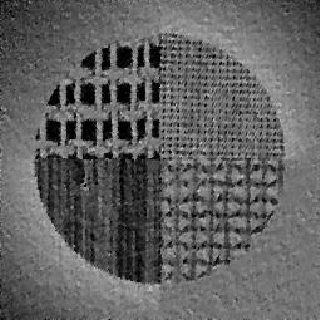} 
        \subcaption{} \label{fig:hybrid_0.1_10000}
    \end{minipage}
    \begin{minipage}[b]{.2\linewidth}
        \centering
        \includegraphics[width=\textwidth]{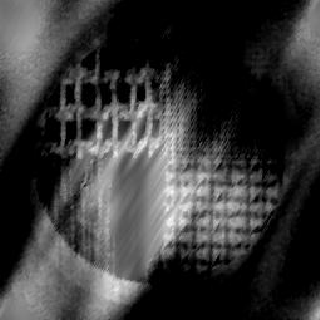} 
        \subcaption{} \label{fig:hybrid_0.15_40}
    \end{minipage}
    \begin{minipage}[b]{.2\linewidth}
        \centering
        \includegraphics[width=\textwidth]{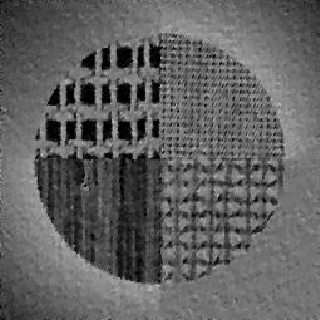} 
        \subcaption{} \label{fig:hybrid_0.5_10000}
    \end{minipage}
    \caption{{Recovery results of 
    Fig. \ref{hybrid-blurred0} using the proposed model with $\lambda=15$. 
    (a) Explicit scheme \eqref{eqn:explicit-scheme-u}, $\tau=0.1$, $10 \ 000$ iterations. 
    (b) Explicit scheme \eqref{eqn:explicit-scheme-u}, $\tau=0.15$, $40$ iterations. 
    (c) Semi-implicit scheme \eqref{eqn:semi-implicit-scheme-u}, $\tau=0.5$, 
    $10 \ 000$ iterations.}} \label{fig:hybrid_ie}
\end{figure}

{In deblurring tasks, for relatively small noise levels $\sigma$ one typically chooses 
a larger $\lambda$ to improve restoration fidelity; for example, when $\sigma=3$ we empirically 
found that $\lambda$ on the order of $10^1$ gives the best visual 
results (see the experiments later). 
When the explicit scheme is employed, the time step $\tau$ is constrained by 
the condition $\tau \le \frac{2}{\lambda}$; for practical choices of $\lambda$, 
this yields an upper bound on $\tau$ that lies roughly between $0.02$ and $0.2$},
resulting in a considerable number of time steps for the deblurring process. 
{The semi-implicit scheme, by contrast, does not incur this additional restriction.} 
To mitigate computational cost, we {therefore} adopt the semi-implicit 
scheme \eqref{eqn:semi-implicit-scheme-u} in the following experiments.

\begin{remark}  
    {$ $}
    \begin{enumerate}[(i)] 
        {\item The strict stability in the $\ell^2$ norm demands 
        that $\|u^{n+1}\|_2 \leq \|u^0\|_2 = \|f\|_2$. Since $u$ generally does not 
        satisfy a maximum principle, this condition is somewhat stringent. One may be 
        interested in employing the generalized von Neumann condition 
        $|G(\omega_1,\omega_2)| \leq 1 + O(\tau)$ to analyze the 
        explicit and semi-implicit schemes in order to obtain sharper stability conditions, 
        but this becomes more involved and requires further investigation. Moreover, as 
        shown in Fig. \ref{fig:rd_err_rel_0.15_40} and Fig. \ref{fig:hybrid_0.15_40}, 
        at the very least, the growth induced by explicitly treating the reaction 
        term is unacceptable.
        \item The conditionally stable semi-implicit 
        scheme \eqref{eqn:semi-implicit-scheme-u} is sufficient for the numerical 
        experiments presented in this paper. However, for images of higher resolution, 
        using larger time steps to reduce computational cost becomes appealing. 
        One possible improvement is to treat the diffusion term implicitly, 
        which leads to the following scheme:
        \begin{align} \label{eqn:implicit-scheme-u} 
            \frac{u_{i,j}^{n+1}-u_{i,j}^n}{\tau} 
            = \, & D_x^- \left(c_{i,j}^n D_x^+ u_{i,j}^{n+1}\right) 
            + D_y^- \left(c_{i,j}^n D_y^+ u_{i,j}^{n+1}\right) \notag \\
            &- \lambda \left(\mathbf{K}' * \mathbf{K} * u^{n+1}\right)_{i,j} 
        + \lambda \left(\mathbf{K}' * f\right)_{i,j},
        \end{align}
        Because of the variable coefficients, this scheme can no longer be efficiently 
        solved using the DFT. Some operator-splitting methods may be required, and the 
        presence of an implicitly treated reaction term makes the corresponding algorithmic 
        design calls for more in-depth study. 
        In addition, acceleration techniques such as Fast Explicit Diffusion \cite{WEICKERT2016} could 
        also be explored to speed up computations based 
        on \eqref{eqn:semi-implicit-scheme-u}, which would be an interesting topic 
        for future work.}
\end{enumerate} 
\end{remark}

{We proceed to conduct a series of comparative experiments in which}
three common types of blur kernels are considered: the motion blur 
kernel with angle $\frac{\pi}{3}$ and length $20$, disk kernel with a radius of $3$
and $5 \times 5$ average kernel.
We test several noisy blurred images, distorted by {additive isotropic} Gaussian noise with a standard 
deviation $\sigma = 3$ and by blurring with one of the previously mentioned kernels, to 
verify the effectiveness of the proposed model. For illustration, the result for 
the texture1 ($256 \times 256$ pixels), texture2 ($256 \times 256$ pixels), 
hybrid ($256 \times 256$ pixels), satellite1($512 \times 512$ pixels), 
satellite2($512 \times 512$ pixels) and satellite3($512 \times 512$ pixels) are presented, 
see the original test images in Fig. \ref{original}. 

\begin{figure}[htbp] 
    \centering
    \begin{minipage}[b]{.2\linewidth}
        \centering
        \includegraphics[width=\textwidth]{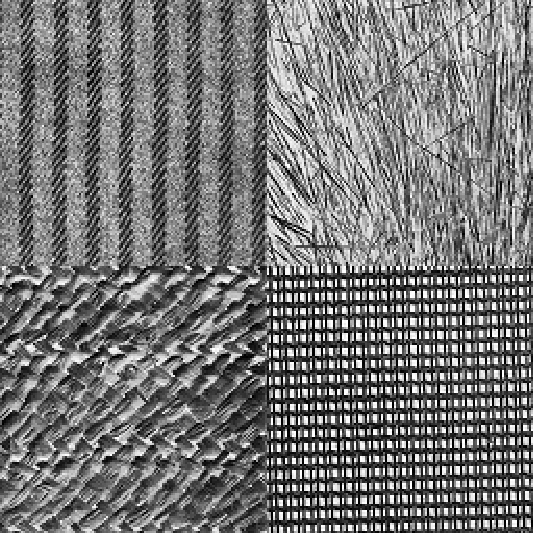}
        \subcaption{Texture1}
    \end{minipage} 
    \begin{minipage}[b]{.2\linewidth}
        \centering
        \includegraphics[width=\textwidth]{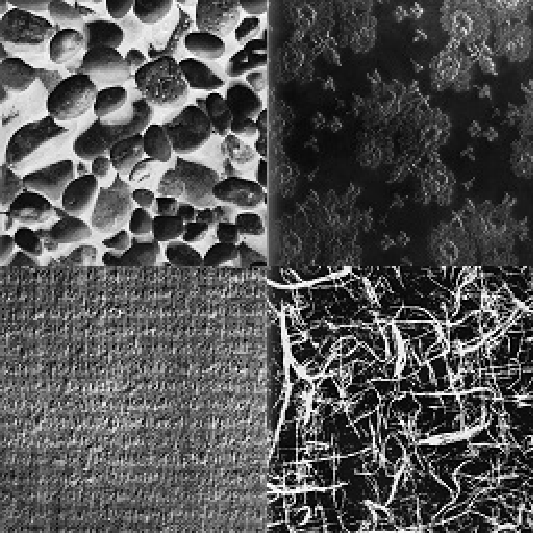}
        \subcaption{Texture2}
    \end{minipage}
    \begin{minipage}[b]{.2\linewidth}
        \centering
        \includegraphics[width=\textwidth]{hybrid.eps}
        \subcaption{Hybrid}
    \end{minipage}

    \begin{minipage}[b]{.2\linewidth}
        \centering
        \includegraphics[width=\textwidth]{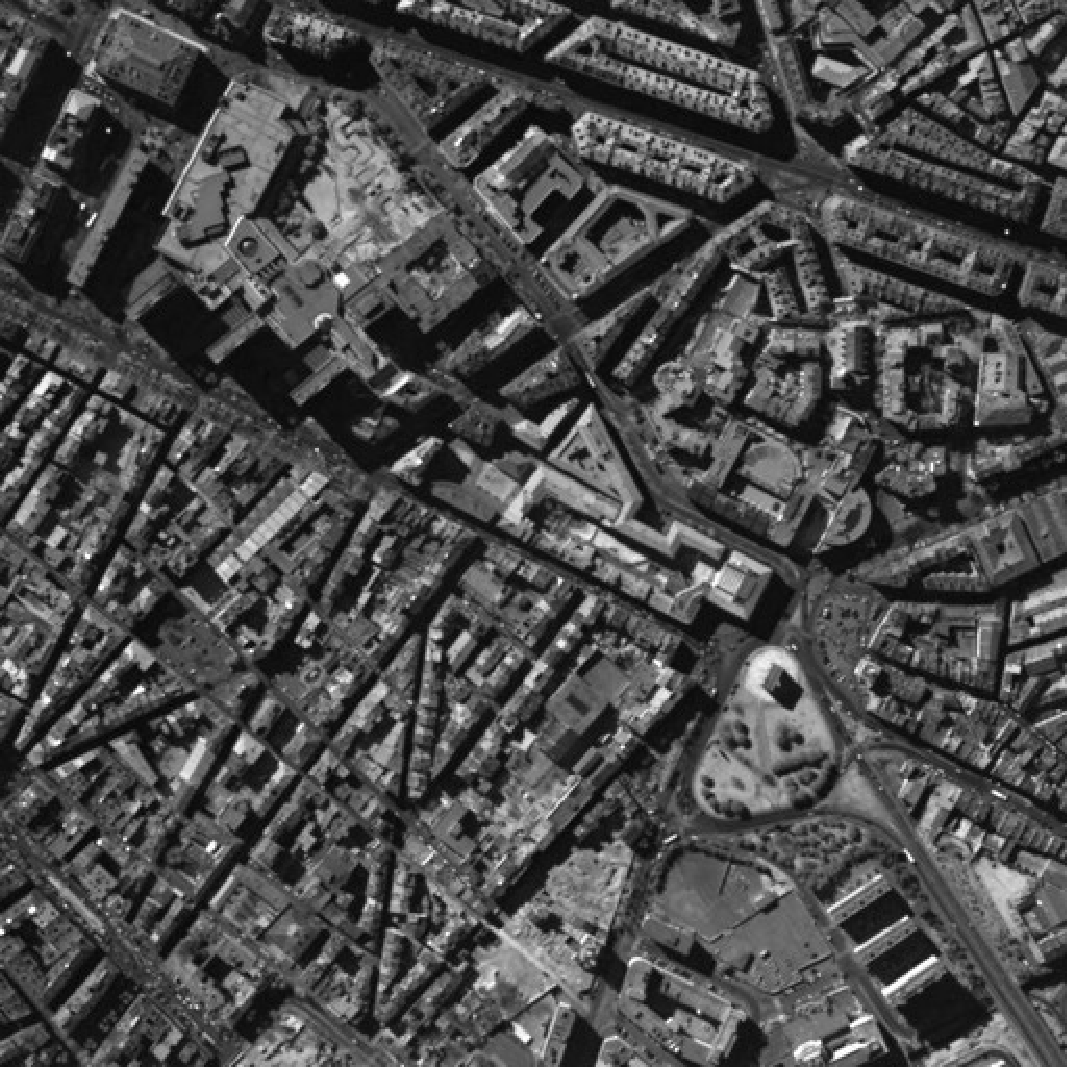}
        \subcaption{Satellite1}
    \end{minipage}
    \begin{minipage}[b]{.2\linewidth}
        \centering
        \includegraphics[width=\textwidth]{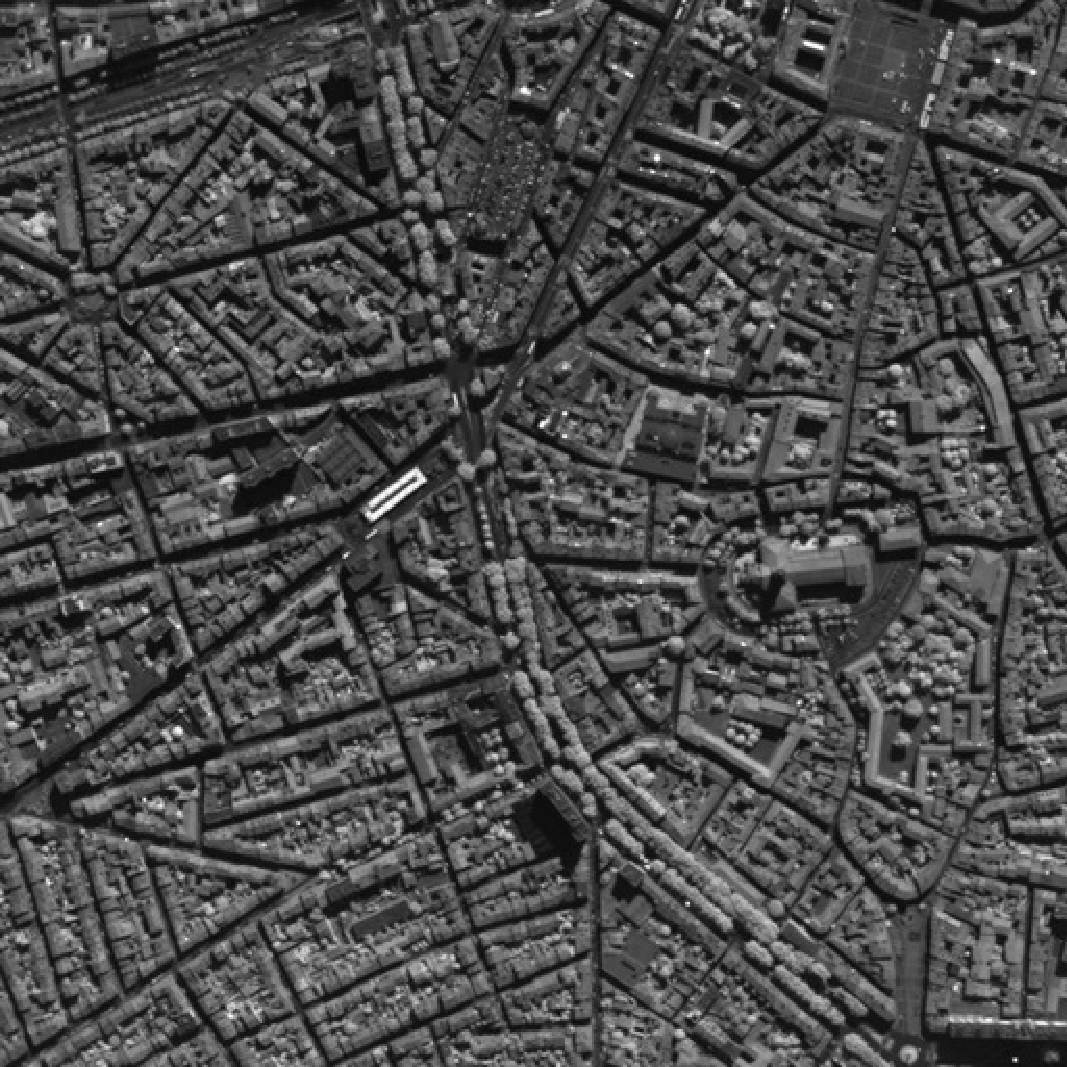}
        \subcaption{Satellite2}
    \end{minipage}
    \begin{minipage}[b]{.2\linewidth}
        \centering
        \includegraphics[width=\textwidth]{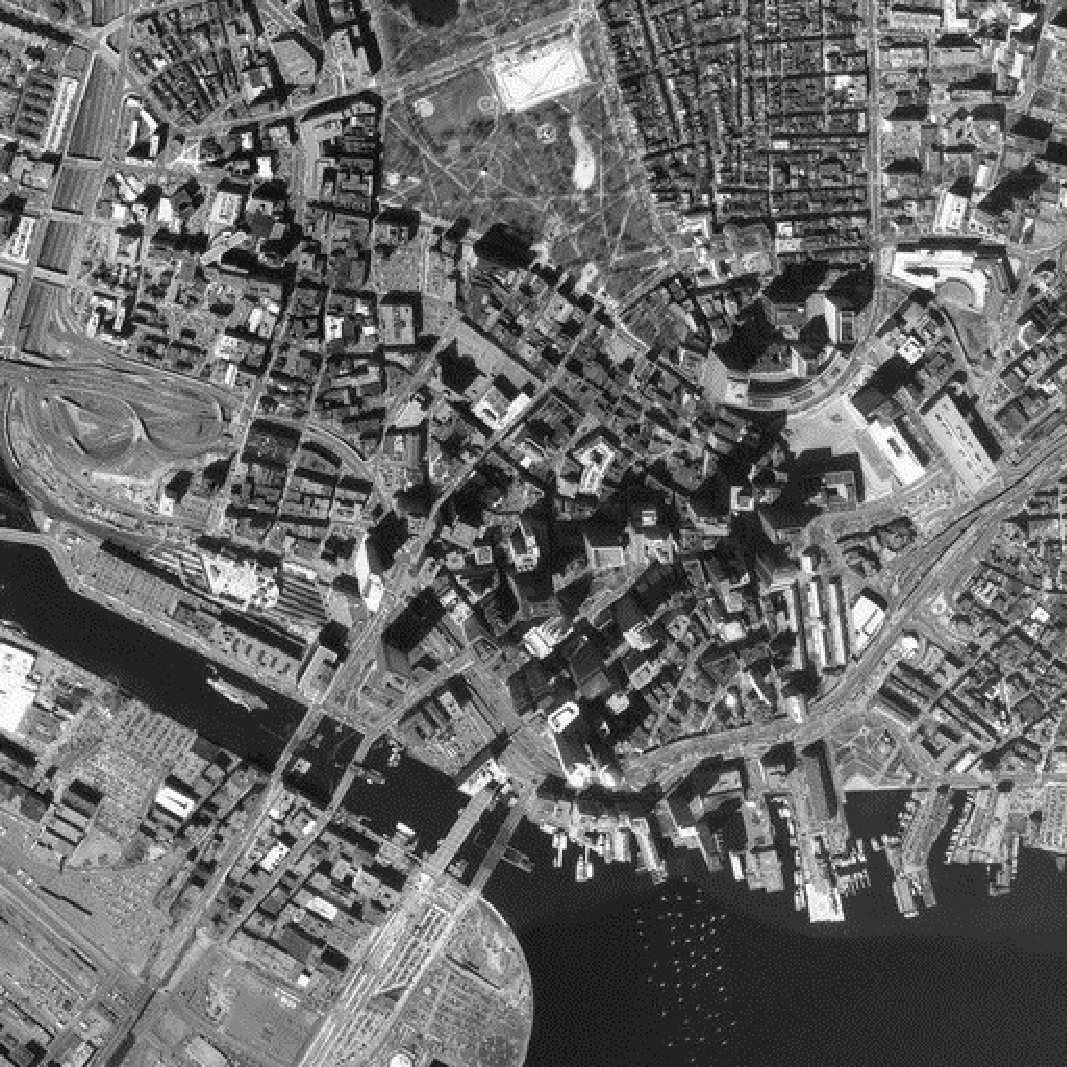}
        \subcaption{Satellite3}
    \end{minipage}
    \caption{Test images.} \label{original}
\end{figure}

The proposed model is compared with 
fast total variation (FastTV) minimization \cite{HUANG2008}, nonlocal total 
variation (NLTV) method \cite{GILBOA2009, ZHANG2010},
nonlocal adaptive biharmonic (NLABH) regularization \cite{WEN2023}, 
a nonlinear fractional diffusion (NFD) model \cite{GUO2019} and 
a Perona-Malik type equation regularized by the $p$-Laplacian (PLRPM) \cite{GUIDOTTI2013}. 
The free parameters for FastTV, NLTV, NLABH, NFD and PLRPM are set as suggested in the reference papers.
For the proposed model, while there may be more suitable parameter choices for each image, 
in {the following experiments}, we set $\tau=0.5$, $\alpha = 0.9$, $\gamma = 1$, $\mu = 0.4$, 
$k_1 = 1$, $\lambda_1 = 0.9$, $maxIter = 500$, and $tol = 0.005$ 
throughout all experiments. For the texture1, texture2, hybrid, satellite 1, 
satellite 2 and satellite 3 images, we select $\beta=0.96, 0.63, 1 , 0.7, 0.9$ 
and $0.8$ respectively. The parameter $\lambda$ is chosen based on the 
specific image and the type of blur kernel. For a comparison of the performance quantitatively, 
we list the PSNR and SSIM values of the restored results in 
Table \ref{results-1} and Table \ref{results-2}.

\begin{table}[htbp] 
    \centering
    \caption{Comparison of PSNR and SSIM for different models in the experiments on 
    texture1, texture2 and hybrid image. Bold values indicate the best result.}
    \label{results-1}
    \begin{tabular}{c*{7}c}
        \toprule
        \multirow{1}*{} 
        & \multicolumn{2}{c}{Figure \ref{texture1}} 
        & \multicolumn{2}{c}{Figure \ref{texture2}} 
        & \multicolumn{2}{c}{Figure \ref{hybrid}} \\
        \cmidrule{2-7} & PSNR  & SSIM  & PSNR  & SSIM & PSNR  & SSIM \\
        \midrule
         
         {Degraded} & {12.77} & {0.1901}          
                  & {15.60} & {0.3640}
                  & {18.07} & {0.4382} \\
         FastTV   & 15.17          & 0.5320          & 19.00          & 0.6887
                  & 21.00          & \textbf{0.7136} \\
         NLTV     & \textbf{16.35} & 0.6599          & 19.26          & \textbf{0.7427}
                  & 21.14          & 0.7087          \\
         NLABH    & 16.11          & 0.6698          & 18.89          & 0.6710
                  & 20.21          & 0.4598          \\
         NFD      & 16.18          & 0.6738          & 19.06          & 0.6763
                  & 20.36          & 0.4723          \\
         PLRPM    & 16.30          & 0.6599          & 19.29          & 0.7068
                  & \textbf{21.17} & 0.6415          \\
         Ours     & 16.25          & \textbf{0.6745} & \textbf{19.39} & 0.6992
                  & \textbf{21.17} & 0.6281          \\
        \bottomrule
    \end{tabular}
\end{table}

\begin{table}[htbp] 
    \centering
    \caption{Comparison of PSNR and SSIM for different models in the experiments on 
    three satellite images. Bold values indicate the best result.}
    \label{results-2}
    \begin{tabular}{c*{7}c}
        \toprule
        \multirow{1}*{} 
        & \multicolumn{2}{c}{Figure \ref{satellite1}} 
        & \multicolumn{2}{c}{Figure \ref{satellite2}} 
        & \multicolumn{2}{c}{Figure \ref{satellite3}} \\
        \cmidrule{2-7} & PSNR  & SSIM  & PSNR  & SSIM & PSNR  & SSIM \\
        \midrule
        {Degraded} & {18.64} & {0.5171}          
                  & {21.68} & {0.7761}
                  & {19.20} & {0.7680} \\
         FastTV  & 23.97          & 0.8888          & 24.91          & 0.9322          
                 & 22.41          & 0.9229          \\
         NLTV    & 23.87          & 0.8838          & 25.46          & 0.9387
                 & \textbf{22.59} & \textbf{0.9291} \\
         NLABH   & 22.56          & 0.8475          & 24.82          & 0.9325
                 & 22.00          & 0.9049          \\
         NFD     & 22.67          & 0.8548          & 24.77          & 0.9301
                 & 22.06          & 0.9059          \\
         PLRPM   & 24.04          & 0.8949          & 25.28          & 0.9380
                 & 22.55          & 0.9205          \\
         Ours    & \textbf{24.44} & \textbf{0.9029} & \textbf{25.47} & \textbf{0.9395}
                 & 22.55          & 0.9212          \\
        \bottomrule
    \end{tabular}
\end{table}

\begin{figure}[!htbp]
    \centering
    \begin{minipage}[t]{.2\linewidth}
        \centering
        \includegraphics[width=\textwidth]{texture1.eps}
        \subcaption{Original} \label{texture1-original}
    \end{minipage} 
    \begin{minipage}[t]{.2\linewidth}
        \centering
        \includegraphics[width=\textwidth]{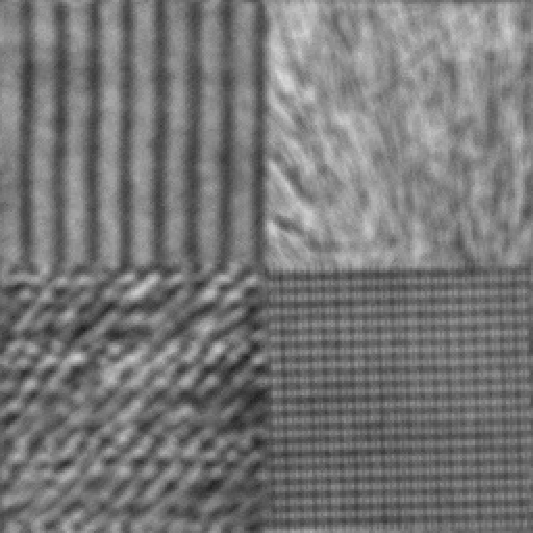}
        \subcaption{Blurred} \label{texture1-blurred}
    \end{minipage}
    \begin{minipage}[t]{.2\linewidth}
        \centering
        \includegraphics[width=\textwidth]{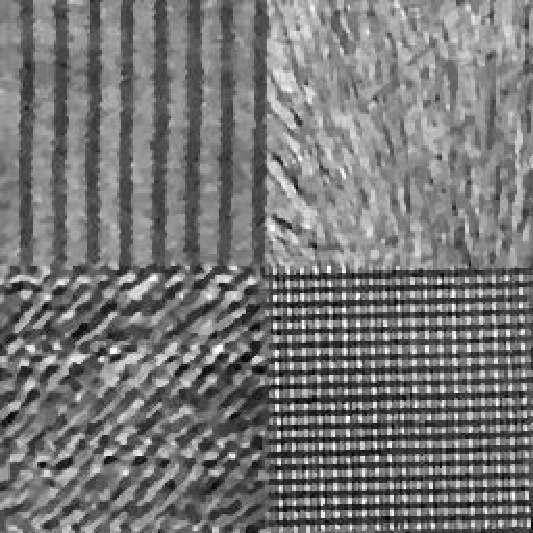}
        \subcaption{FastTV} \label{texture1-FastTV}
    \end{minipage}
    \begin{minipage}[t]{.2\linewidth}
        \centering
        \includegraphics[width=\textwidth]{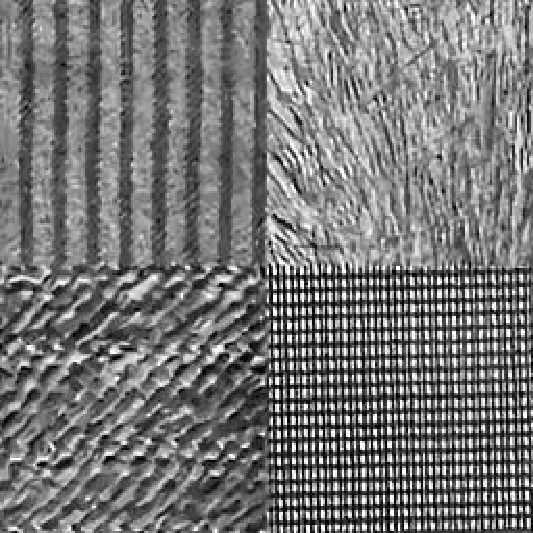}
        \subcaption{NLTV} \label{texture1-NLTV}
    \end{minipage}
    
    \begin{minipage}[b]{.2\linewidth}
        \centering
        \includegraphics[width=\textwidth]{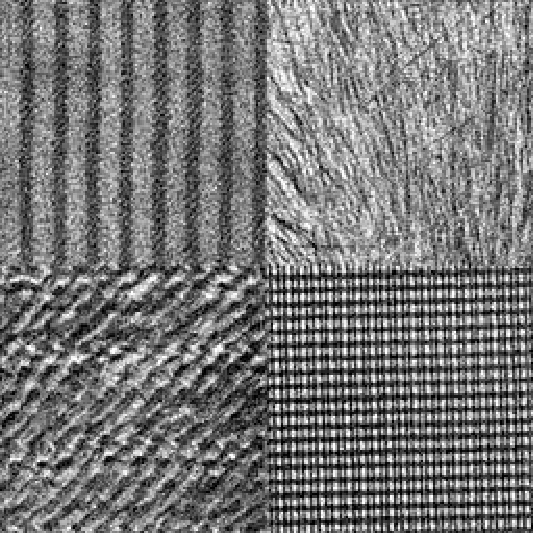}
        \subcaption{NLABH} \label{texture1-NLABH}
    \end{minipage}
    \begin{minipage}[b]{.2\linewidth}
        \centering
        \includegraphics[width=\textwidth]{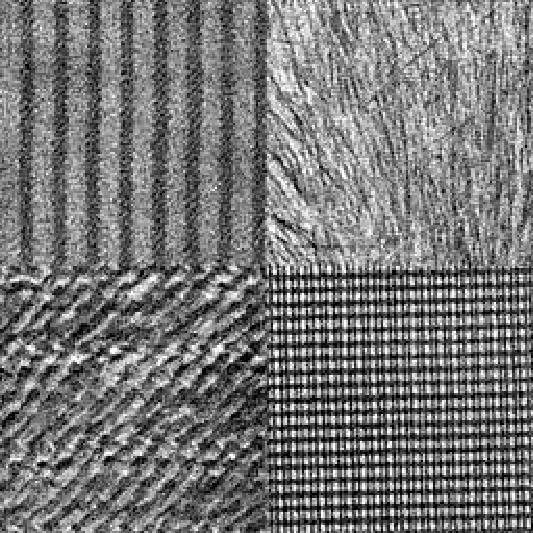}
        \subcaption{NFD} \label{texture1-NFD}
    \end{minipage}
    \begin{minipage}[b]{.2\linewidth}
        \centering
        \includegraphics[width=\textwidth]{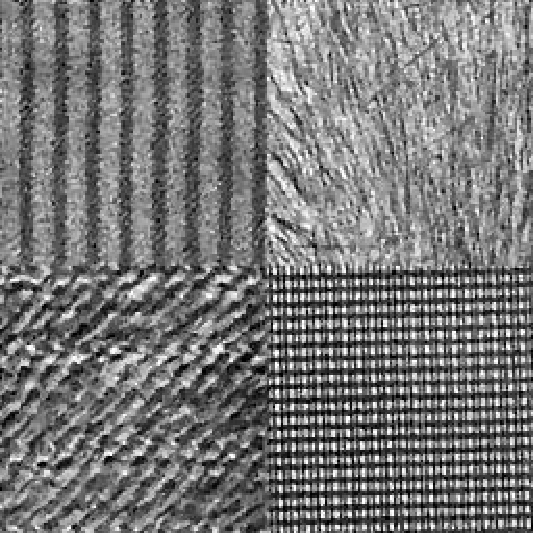}
        \subcaption{PLRPM} \label{texture1-PLRPM}
    \end{minipage}
    \begin{minipage}[b]{.2\linewidth}
        \centering
        \includegraphics[width=\textwidth]{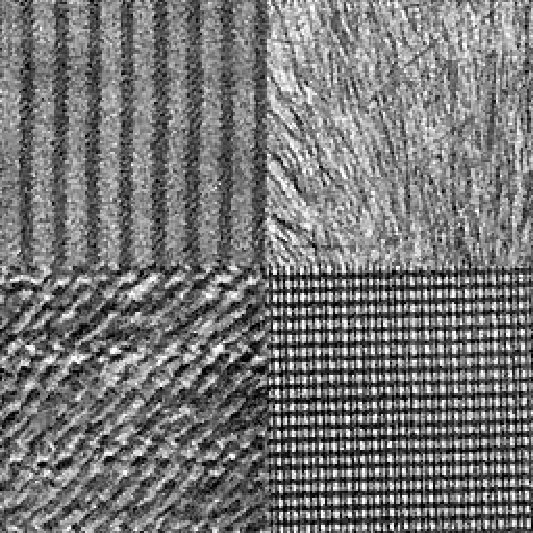}
        \subcaption{Ours} \label{texture1-Ours}
    \end{minipage}
    \caption{Recovery results for the texture1 image with disk blur and corrupted by 
    the noise of standard deviation $\sigma = 3$. (a) original image. 
    (b) noisy blurred image, PSNR=$12.77$. (c)-(h) recovered images.} \label{texture1}
\end{figure}

\begin{figure}[!htbp]
    \centering
    \begin{minipage}[t]{.2\linewidth}
        \centering
        \includegraphics[width=\textwidth]{texture2.eps}
        \subcaption{Original} \label{texture2-original}
    \end{minipage} 
    \begin{minipage}[t]{.2\linewidth}
        \centering
        \includegraphics[width=\textwidth]{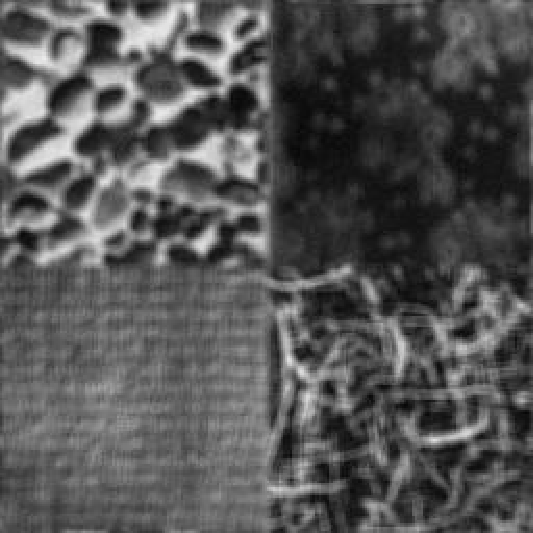}
        \subcaption{Blurred} \label{texture2-blurred}
    \end{minipage}
    \begin{minipage}[t]{.2\linewidth}
        \centering
        \includegraphics[width=\textwidth]{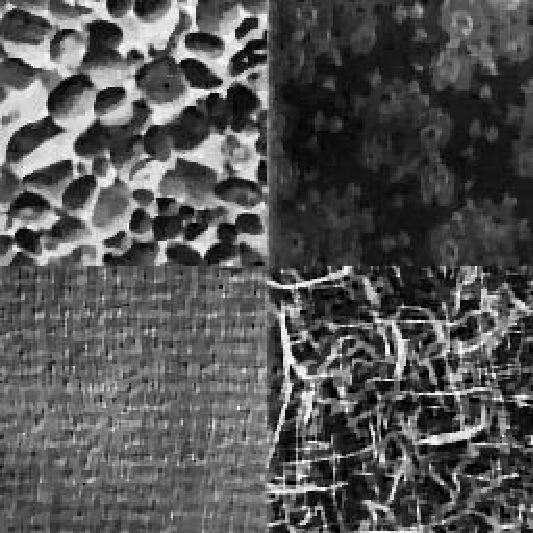}
        \subcaption{FastTV} \label{texture2-FastTV}
    \end{minipage}
    \begin{minipage}[t]{.2\linewidth}
        \centering
        \includegraphics[width=\textwidth]{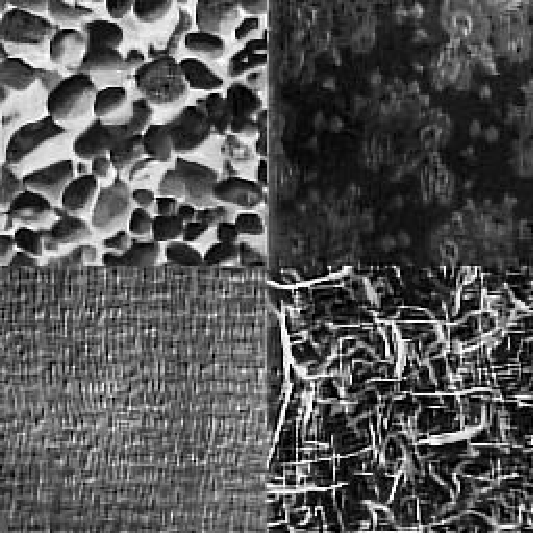}
        \subcaption{NLTV} \label{texture2-NLTV}
    \end{minipage}
    
    \begin{minipage}[b]{.2\linewidth}
        \centering
        \includegraphics[width=\textwidth]{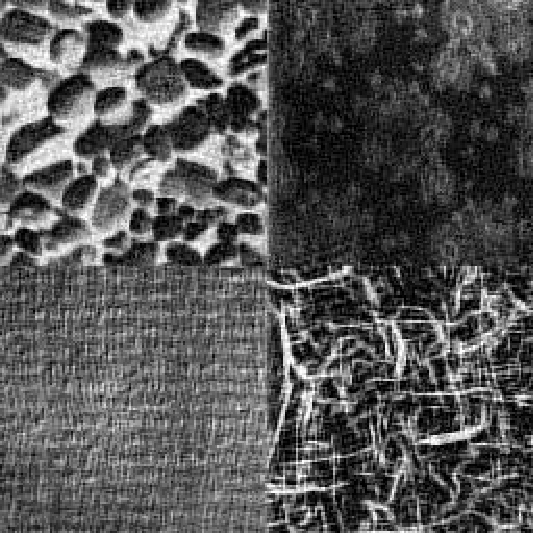}
        \subcaption{NLABH} \label{texture2-NLABH}
    \end{minipage}
    \begin{minipage}[b]{.2\linewidth}
        \centering
        \includegraphics[width=\textwidth]{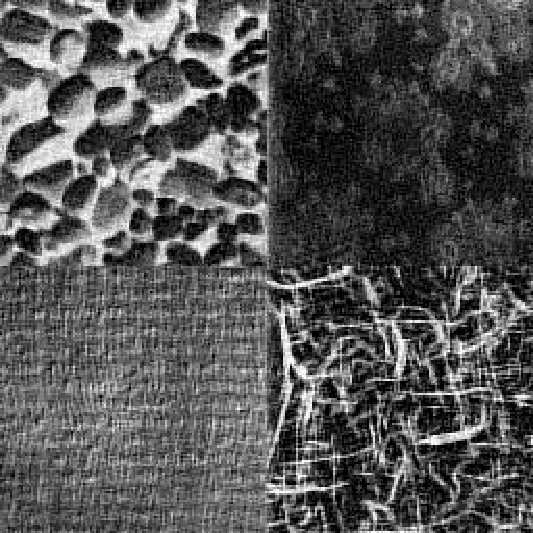}
        \subcaption{NFD} \label{texture2-NFD}
    \end{minipage}
    \begin{minipage}[b]{.2\linewidth}
        \centering
        \includegraphics[width=\textwidth]{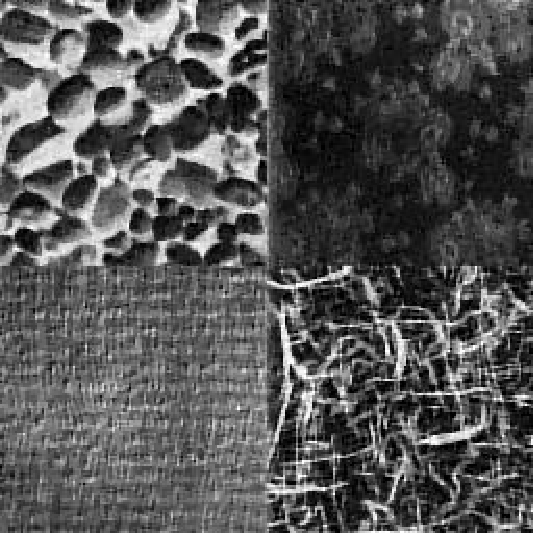}
        \subcaption{PLRPM} \label{texture2-PLRPM}
    \end{minipage}
    \begin{minipage}[b]{.2\linewidth}
        \centering
        \includegraphics[width=\textwidth]{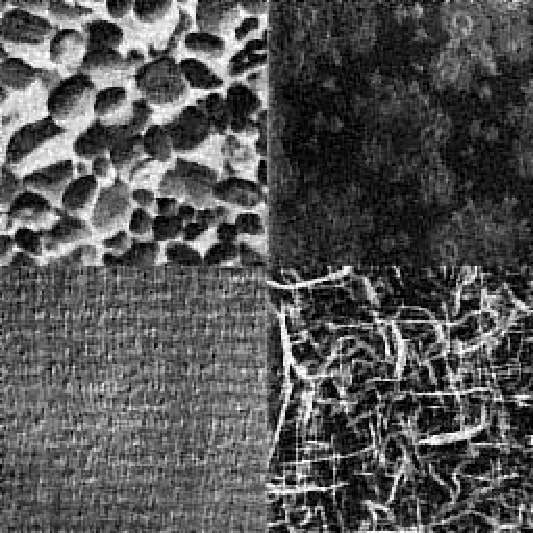}
        \subcaption{Ours} \label{texture2-Ours}
    \end{minipage}
    \caption{Recovery results for the texture2 image with average blur and corrupted by 
    the noise of standard deviation $\sigma = 3$. (a) original image. 
    (b) noisy blurred image, PSNR=$15.60$. (c)-(h) recovered images.} \label{texture2}
\end{figure}

\begin{figure}[!htbp]
    \centering
    \begin{minipage}[t]{.2\linewidth}
        \centering
        \includegraphics[width=\textwidth]{hybrid.eps}
        \subcaption{Original} \label{hybrid-original}
    \end{minipage} 
    \begin{minipage}[t]{.2\linewidth}
        \centering
        \includegraphics[width=\textwidth]{hybridb.eps}
        \subcaption{Blurred} \label{hybrid-blurred}
    \end{minipage}
    \begin{minipage}[t]{.2\linewidth}
        \centering
        \includegraphics[width=\textwidth]{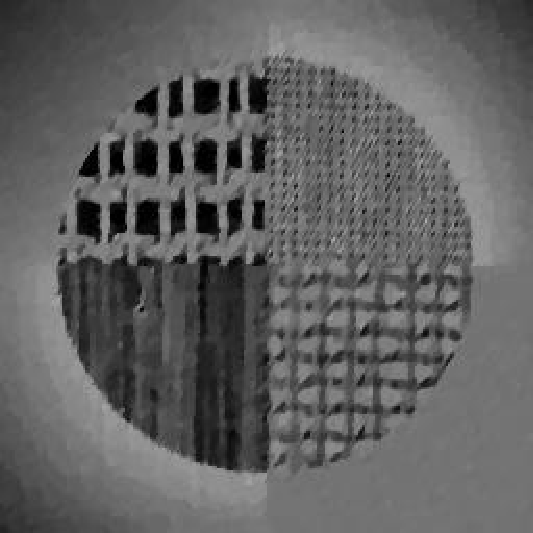}
        \subcaption{FastTV} \label{hybrid-FastTV}
    \end{minipage}
    \begin{minipage}[t]{.2\linewidth}
        \centering
        \includegraphics[width=\textwidth]{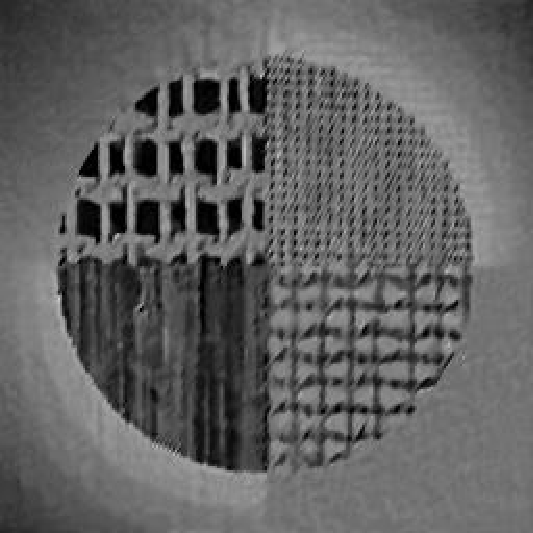}
        \subcaption{NLTV} \label{hybrid-NLTV}
    \end{minipage}
    
    \begin{minipage}[b]{.2\linewidth}
        \centering
        \includegraphics[width=\textwidth]{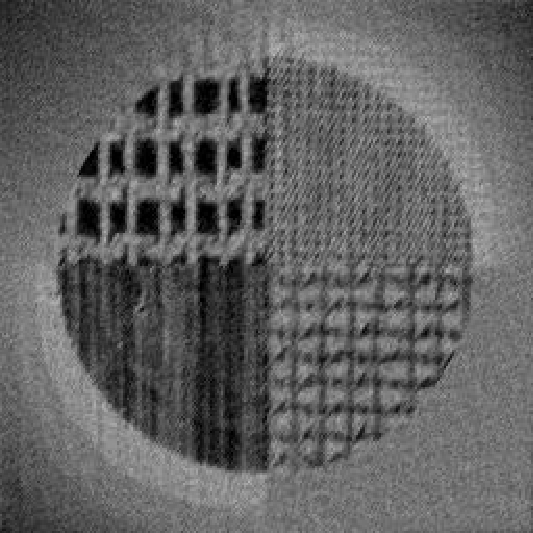}
        \subcaption{NLABH} \label{hybrid-NLABH}
    \end{minipage}
    \begin{minipage}[b]{.2\linewidth}
        \centering
        \includegraphics[width=\textwidth]{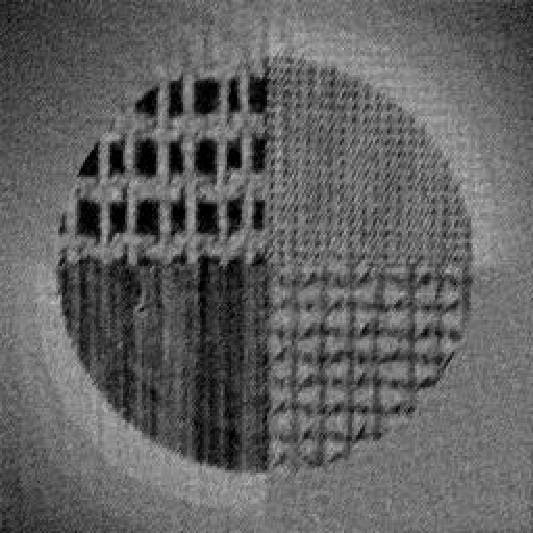}
        \subcaption{NFD} \label{hybrid-NFD}
    \end{minipage}
    \begin{minipage}[b]{.2\linewidth}
        \centering
        \includegraphics[width=\textwidth]{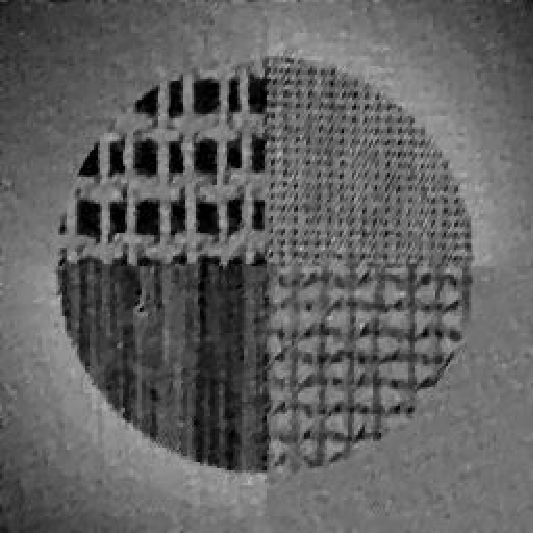}
        \subcaption{PLRPM} \label{hybrid-PLRPM}
    \end{minipage}
    \begin{minipage}[b]{.2\linewidth}
        \centering
        \includegraphics[width=\textwidth]{hybridh.eps}
        \subcaption{Ours} \label{hybrid-Ours}
    \end{minipage}
    \caption{Recovery results for the hybrid image with motion blur and corrupted by 
    the noise of standard deviation $\sigma = 3$. (a) original image. 
    (b) noisy blurred image, PSNR=$18.07$. (c)-(h) recovered images.} \label{hybrid}
\end{figure}

\begin{figure}[!htbp]
    \centering
    \begin{minipage}[t]{.2\linewidth}
        \centering
        \includegraphics[width=\textwidth]{satellite1.eps}
        \subcaption{Original} \label{satellite1-original}
    \end{minipage} 
    \begin{minipage}[t]{.2\linewidth}
        \centering
        \includegraphics[width=\textwidth]{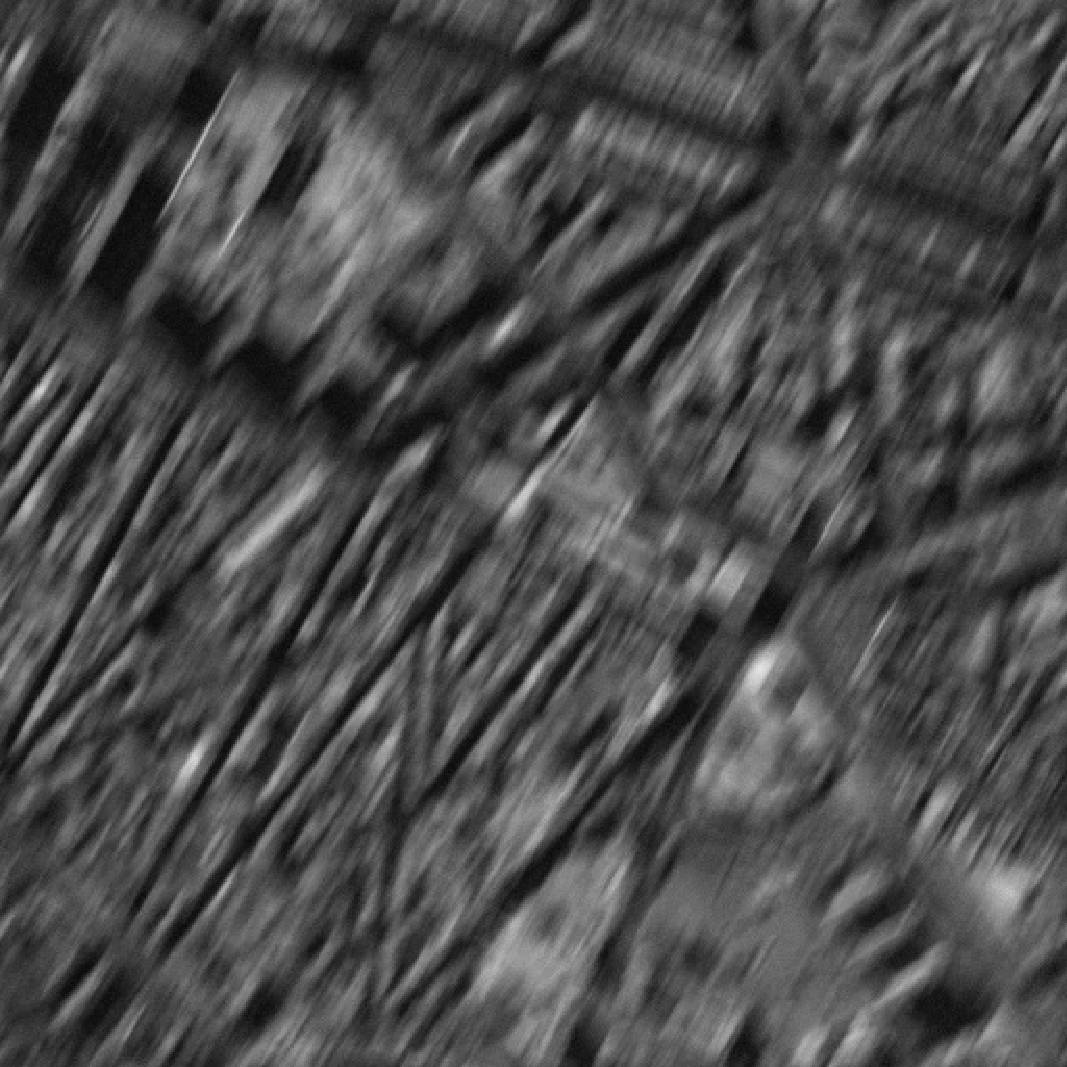}
        \subcaption{Blurred} \label{satellite1-blurred}
    \end{minipage}
    \begin{minipage}[t]{.2\linewidth}
        \centering
        \includegraphics[width=\textwidth]{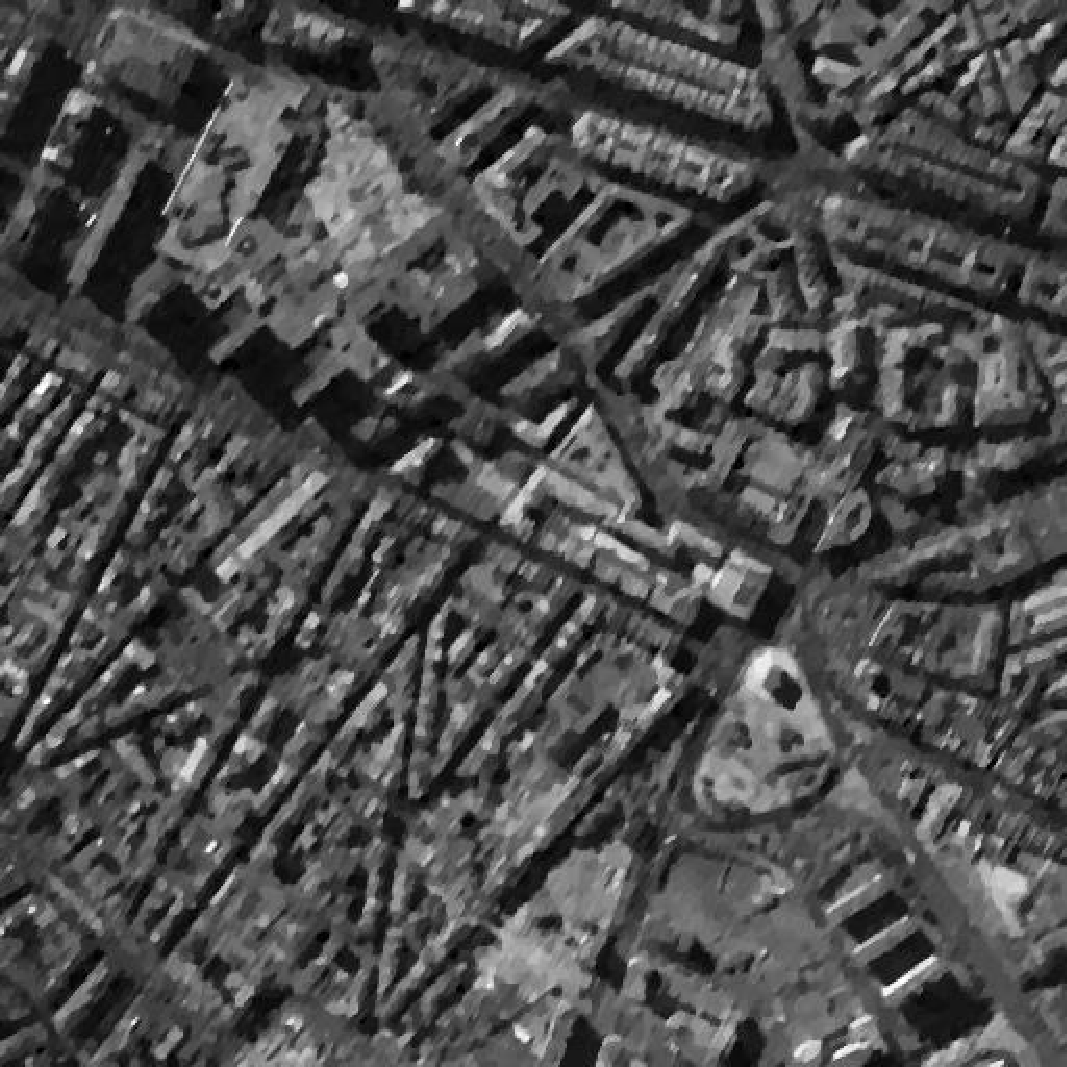}
        \subcaption{FastTV} \label{satellite1-FastTV}
    \end{minipage}
    \begin{minipage}[t]{.2\linewidth}
        \centering
        \includegraphics[width=\textwidth]{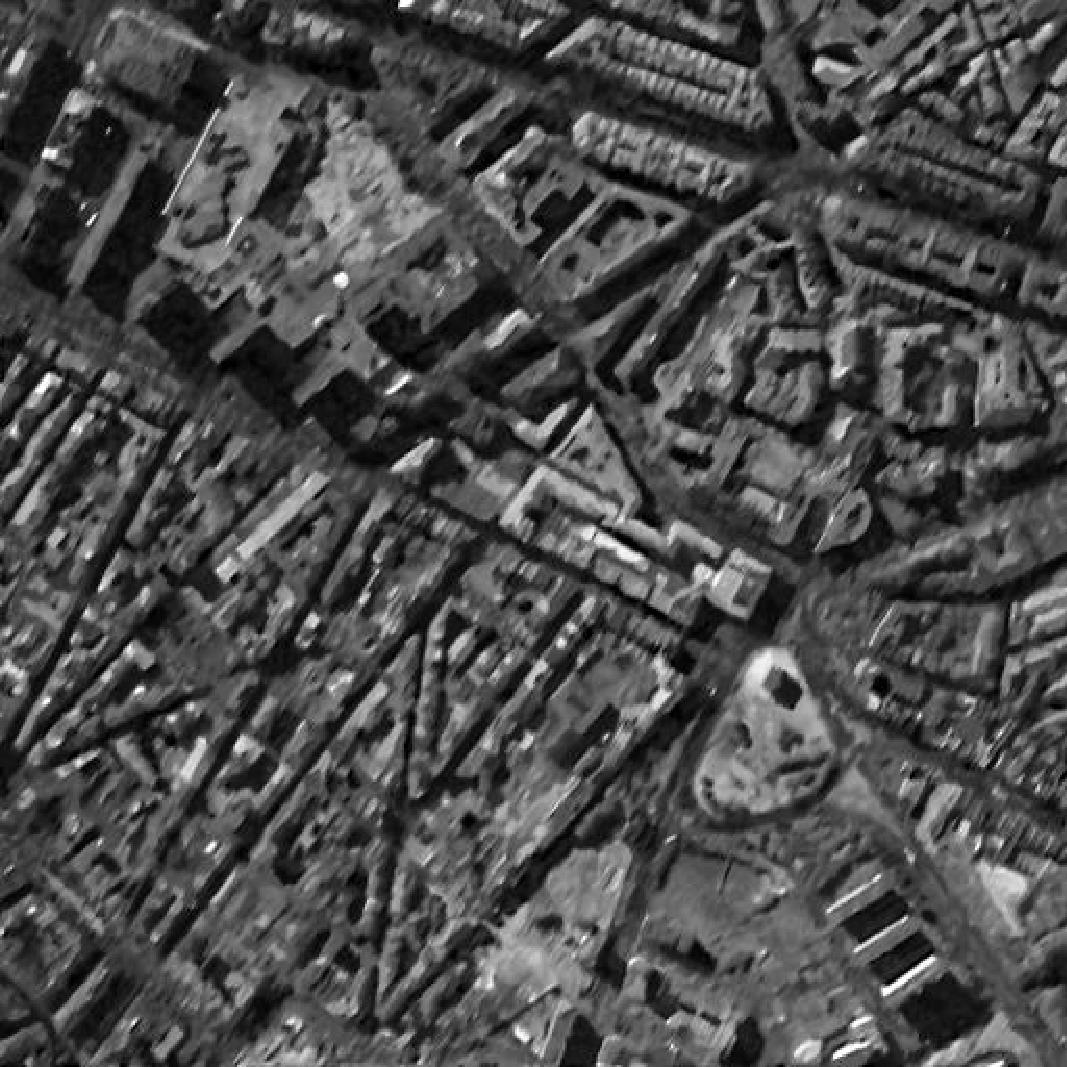}
        \subcaption{NLTV} \label{satellite1-NLTV}
    \end{minipage}
    
    \begin{minipage}[b]{.2\linewidth}
        \centering
        \includegraphics[width=\textwidth]{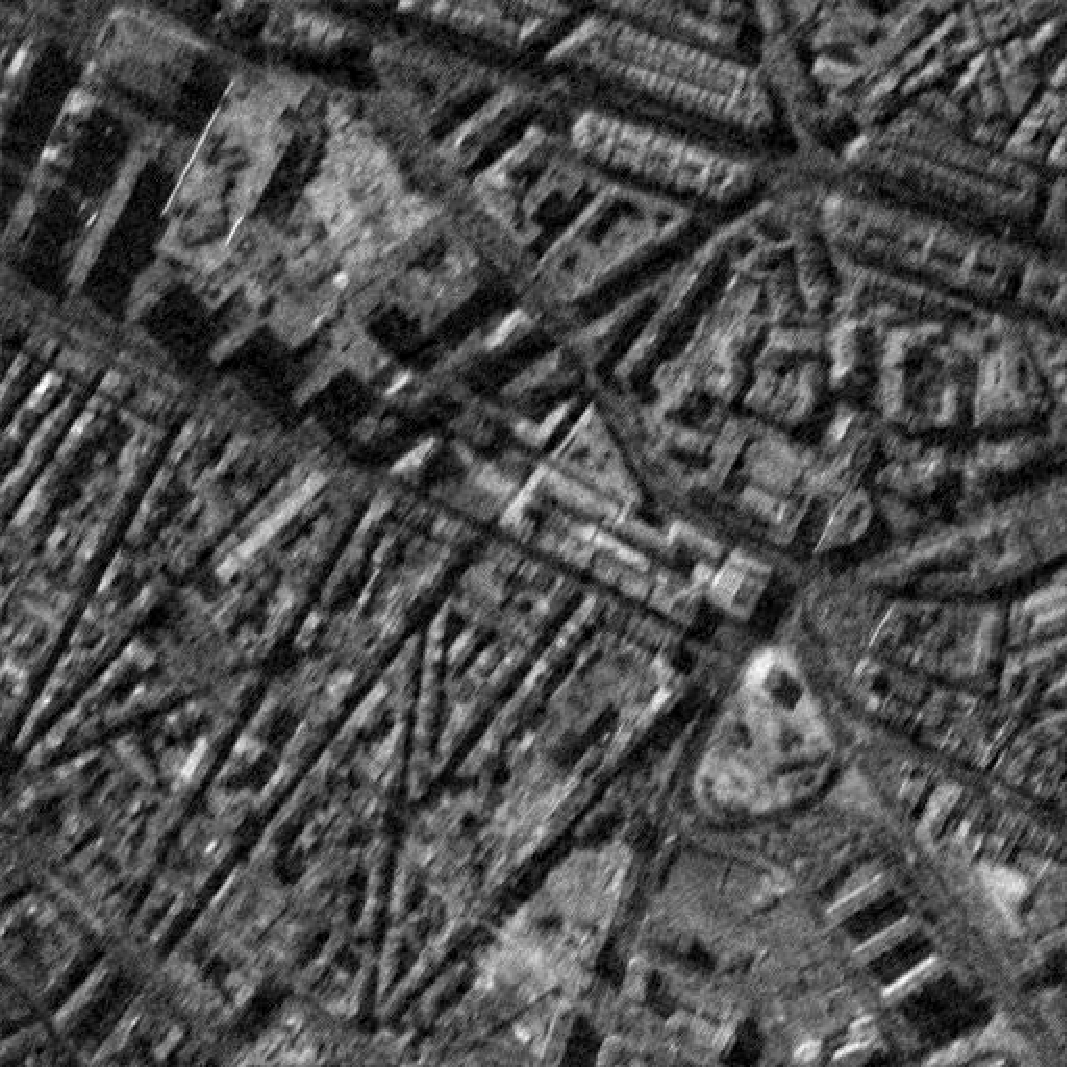}
        \subcaption{NLABH} \label{satellite1-NLABH}
    \end{minipage}
    \begin{minipage}[b]{.2\linewidth}
        \centering
        \includegraphics[width=\textwidth]{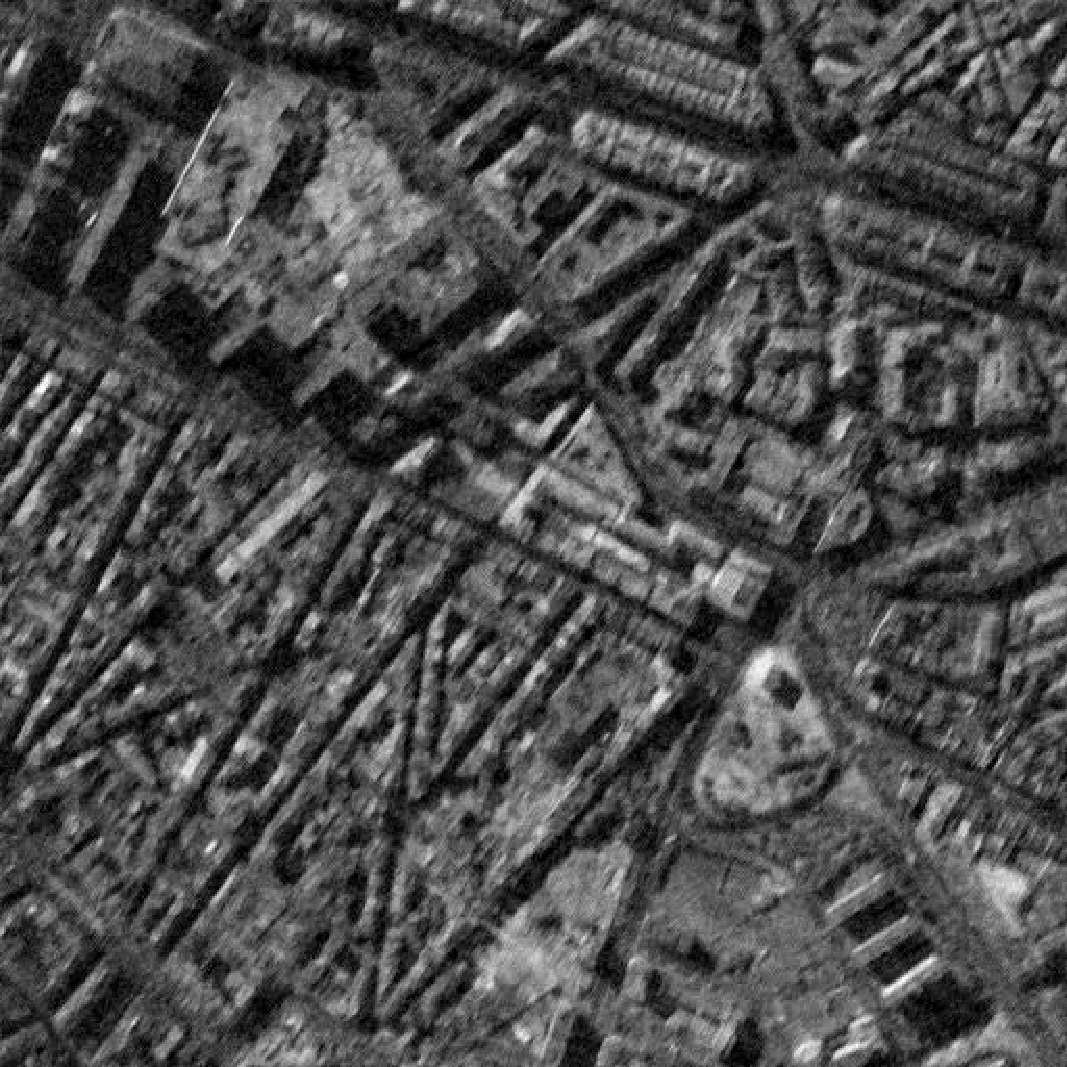}
        \subcaption{NFD} \label{satellite1-NFD}
    \end{minipage}
    \begin{minipage}[b]{.2\linewidth}
        \centering
        \includegraphics[width=\textwidth]{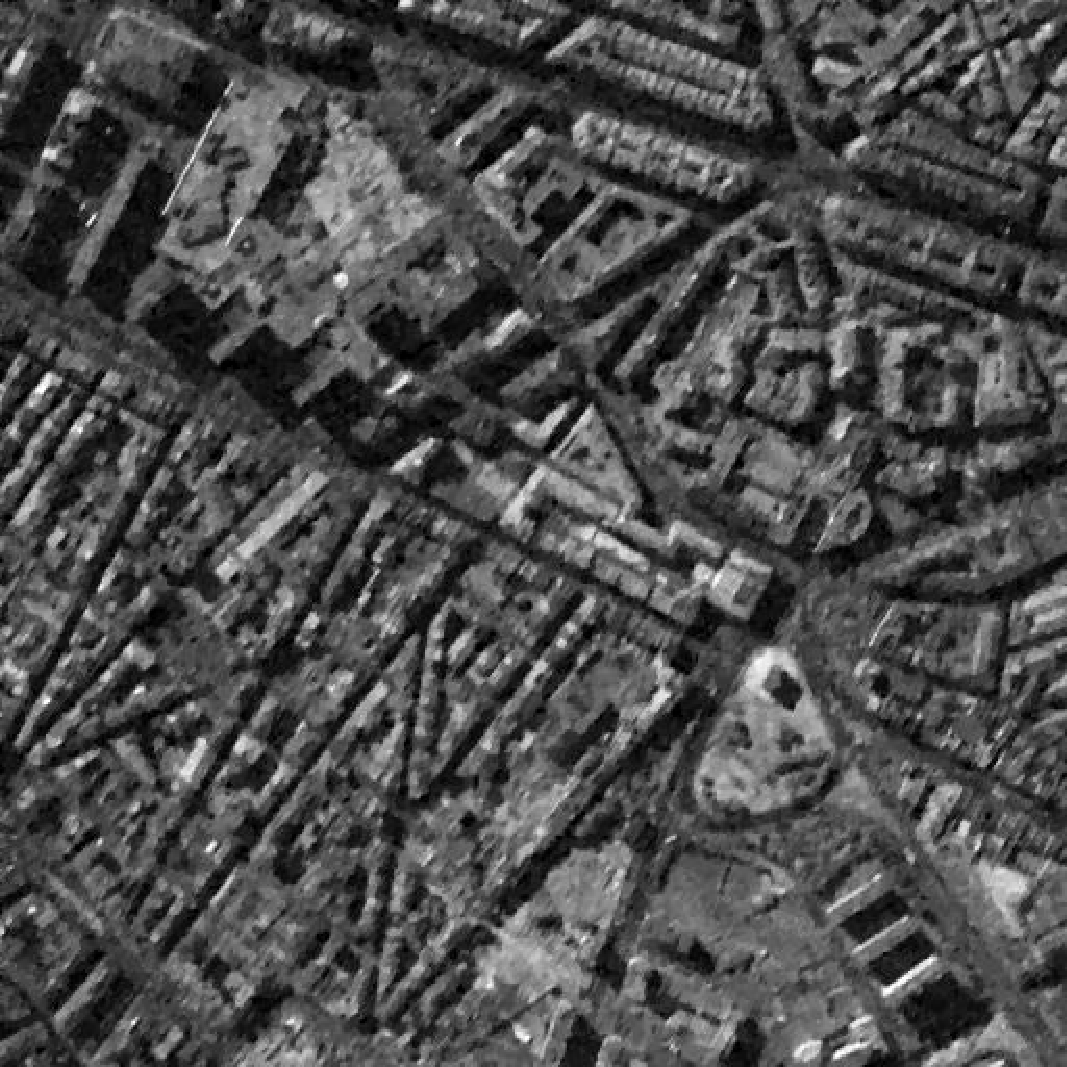}
        \subcaption{PLRPM} \label{satellite1-PLRPM}
    \end{minipage}
    \begin{minipage}[b]{.2\linewidth}
        \centering
        \includegraphics[width=\textwidth]{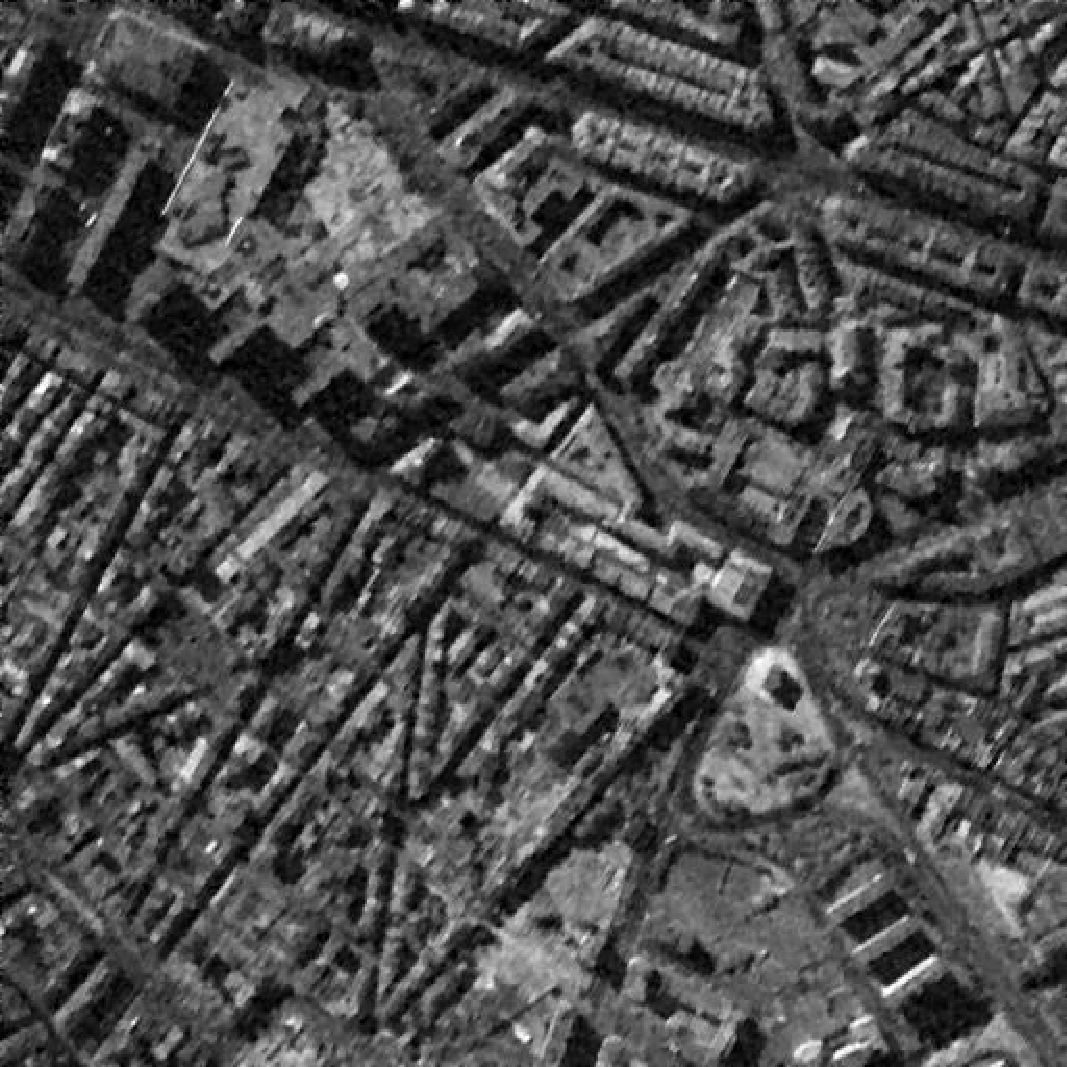}
        \subcaption{Ours} \label{satellite1-Ours}
    \end{minipage}
    \caption{Recovery results for the satellite1 image with motion blur and corrupted by 
    the noise of standard deviation $\sigma = 3$. (a) original image. 
    (b) noisy blurred image, PSNR=$18.64$. (c)-(h) recovered images.} \label{satellite1}
\end{figure}

\begin{figure}[!htbp]
    \centering
    \begin{minipage}[t]{.2\linewidth}
        \centering
        \includegraphics[width=\textwidth]{satellite2.eps}
        \subcaption{Original} \label{satellite2-original}
    \end{minipage} 
    \begin{minipage}[t]{.2\linewidth}
        \centering
        \includegraphics[width=\textwidth]{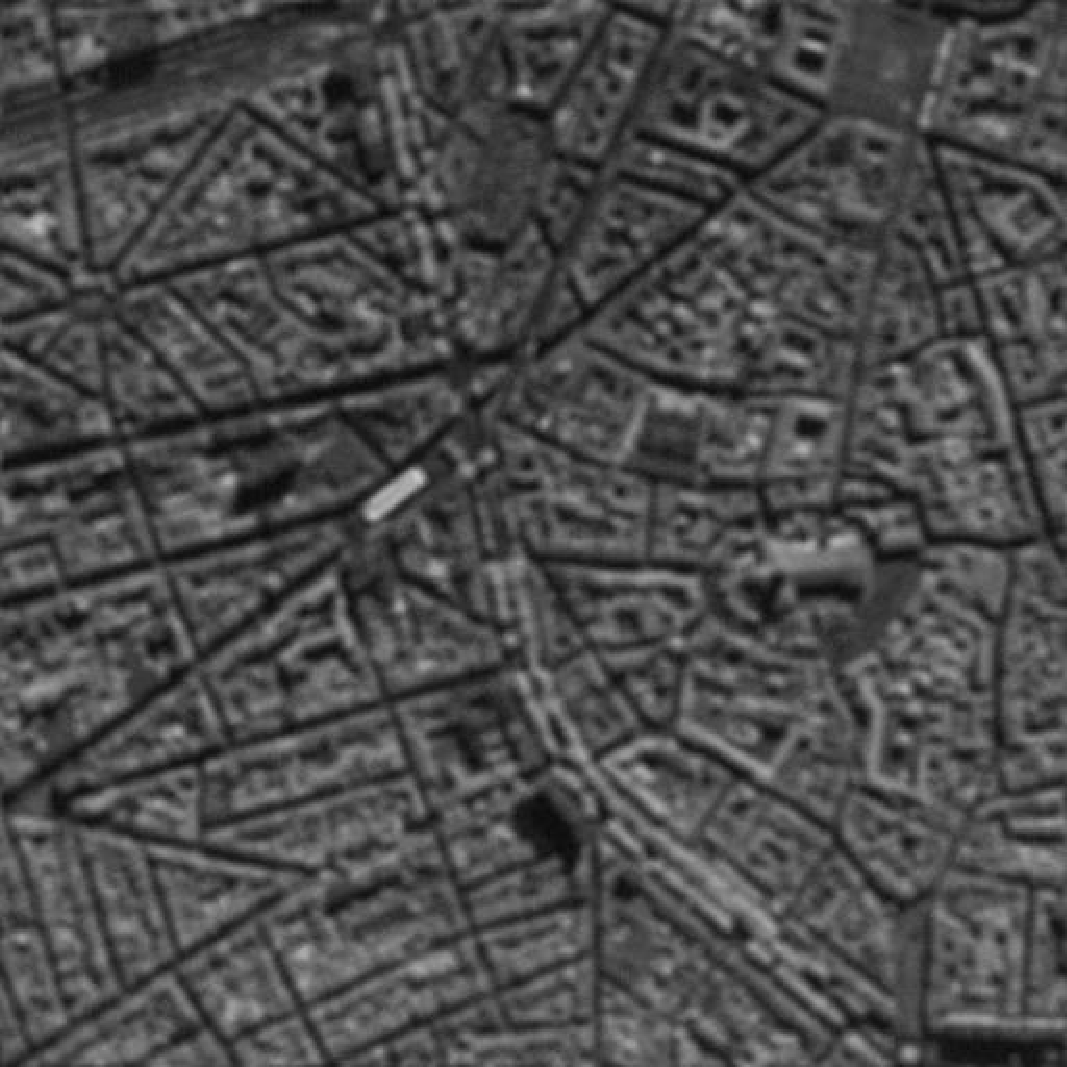}
        \subcaption{Blurred} \label{satellite2-blurred}
    \end{minipage}
    \begin{minipage}[t]{.2\linewidth}
        \centering
        \includegraphics[width=\textwidth]{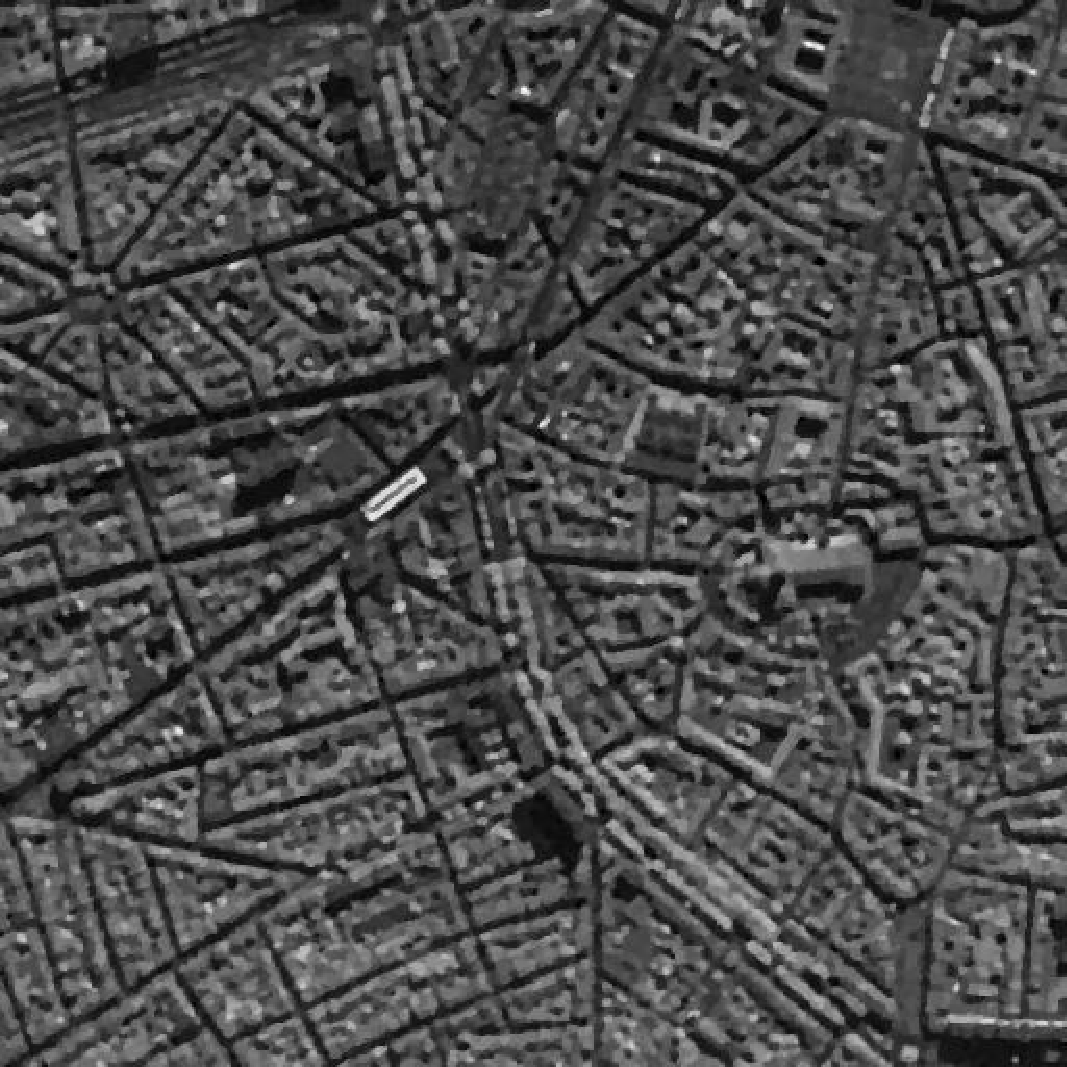}
        \subcaption{FastTV} \label{satellite2-FastTV}
    \end{minipage}
    \begin{minipage}[t]{.2\linewidth}
        \centering
        \includegraphics[width=\textwidth]{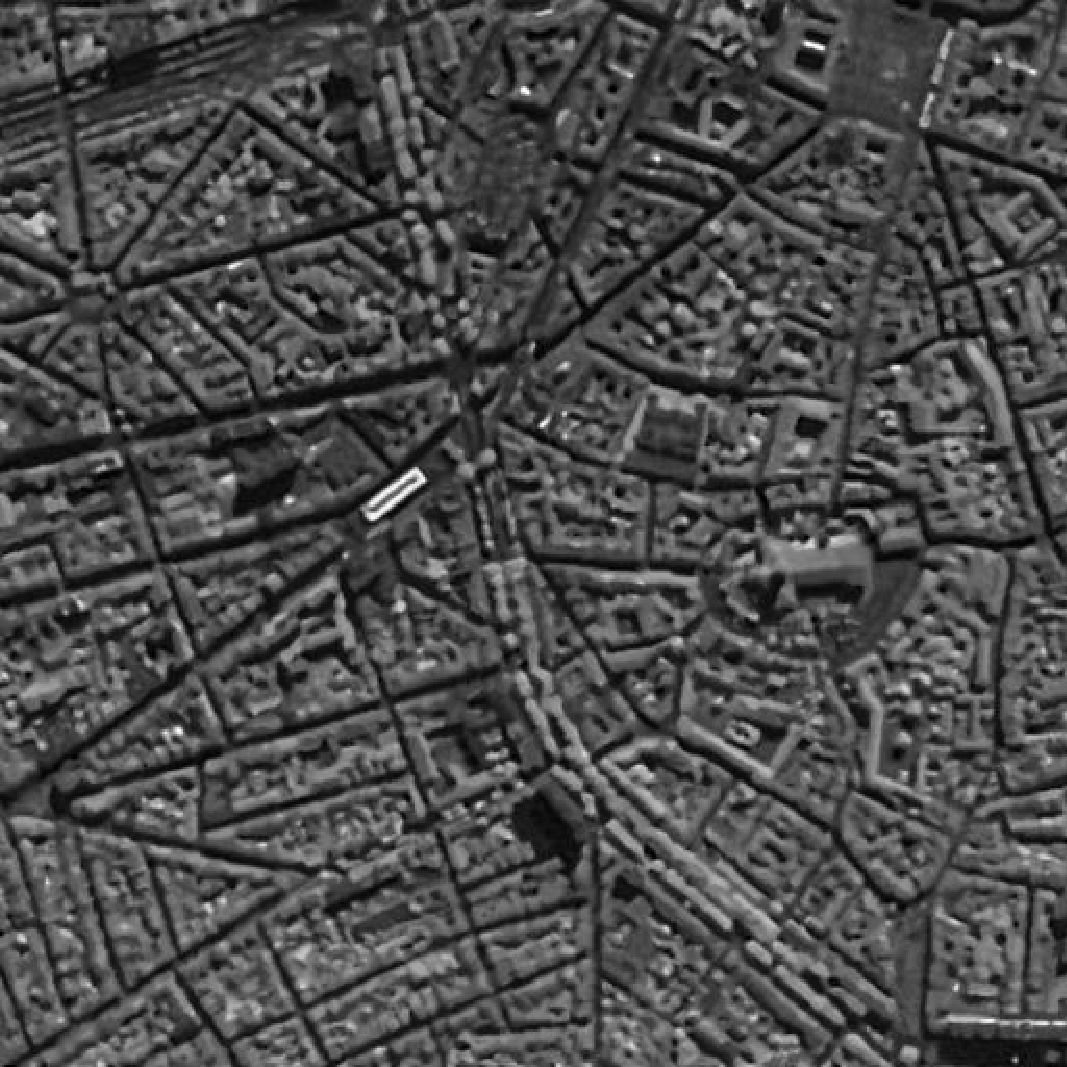}
        \subcaption{NLTV} \label{satellite2-NLTV}
    \end{minipage}
    
    \begin{minipage}[b]{.2\linewidth}
        \centering
        \includegraphics[width=\textwidth]{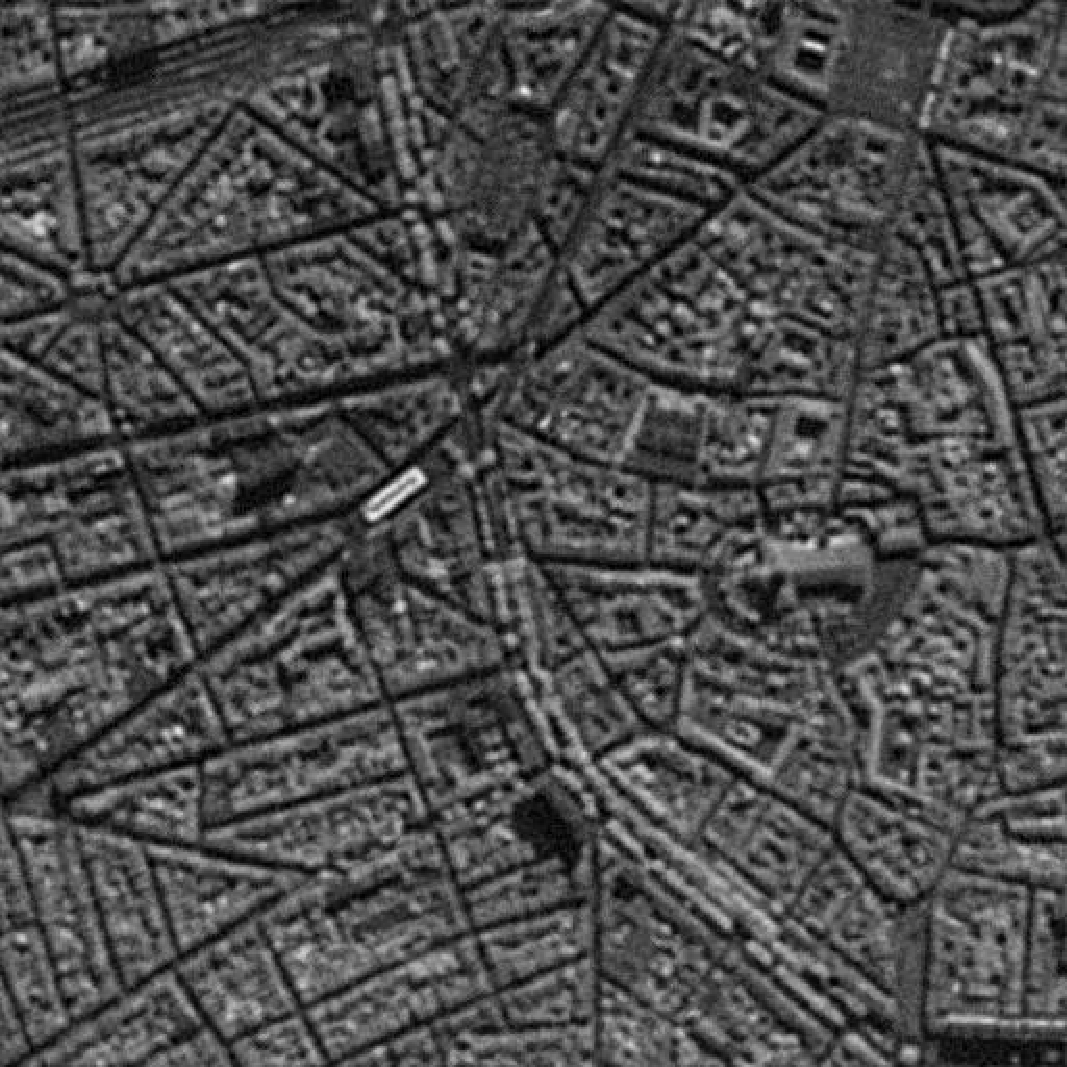}
        \subcaption{NLABH} \label{satellite2-NLABH}
    \end{minipage}
    \begin{minipage}[b]{.2\linewidth}
        \centering
        \includegraphics[width=\textwidth]{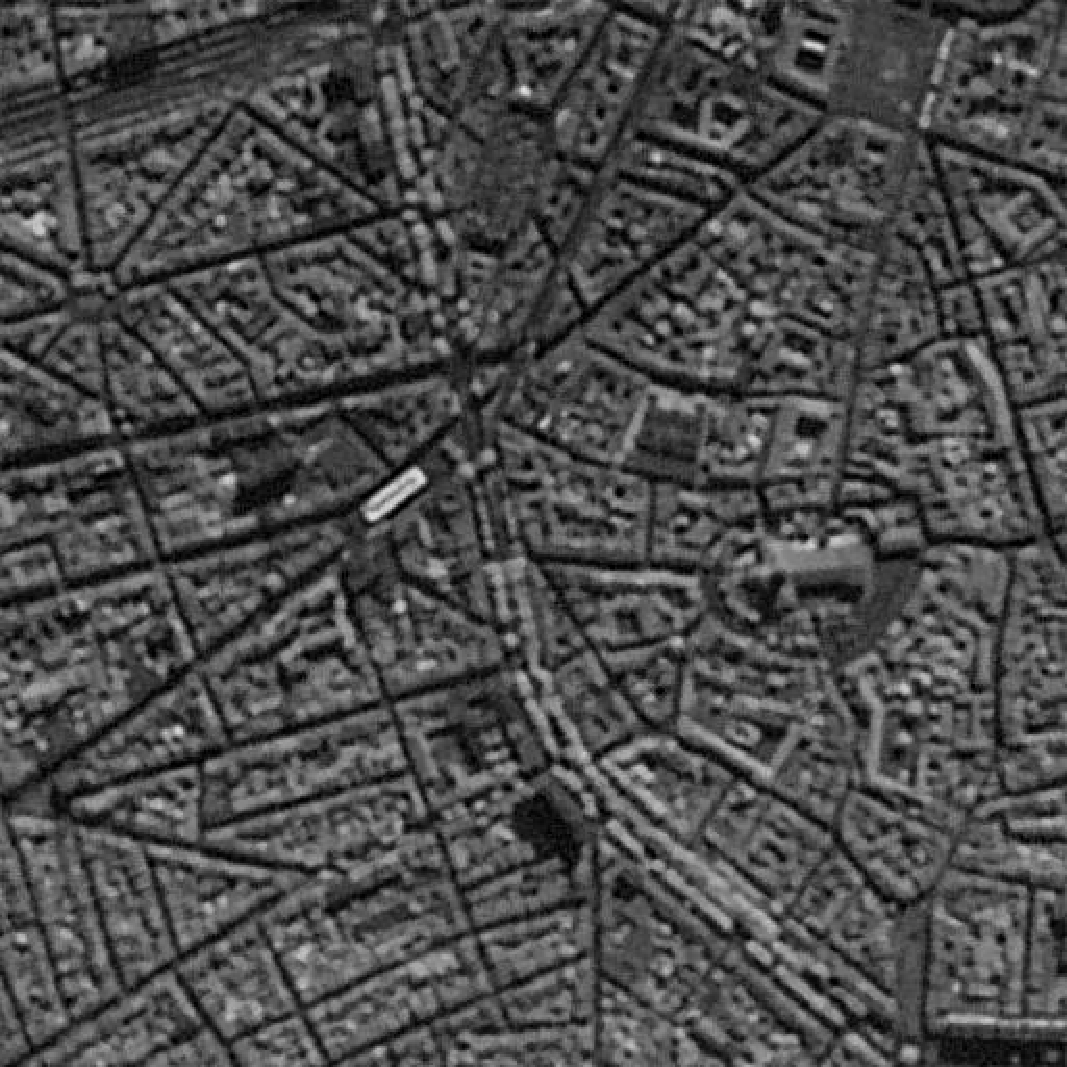}
        \subcaption{NFD} \label{satellite2-NFD}
    \end{minipage}
    \begin{minipage}[b]{.2\linewidth}
        \centering
        \includegraphics[width=\textwidth]{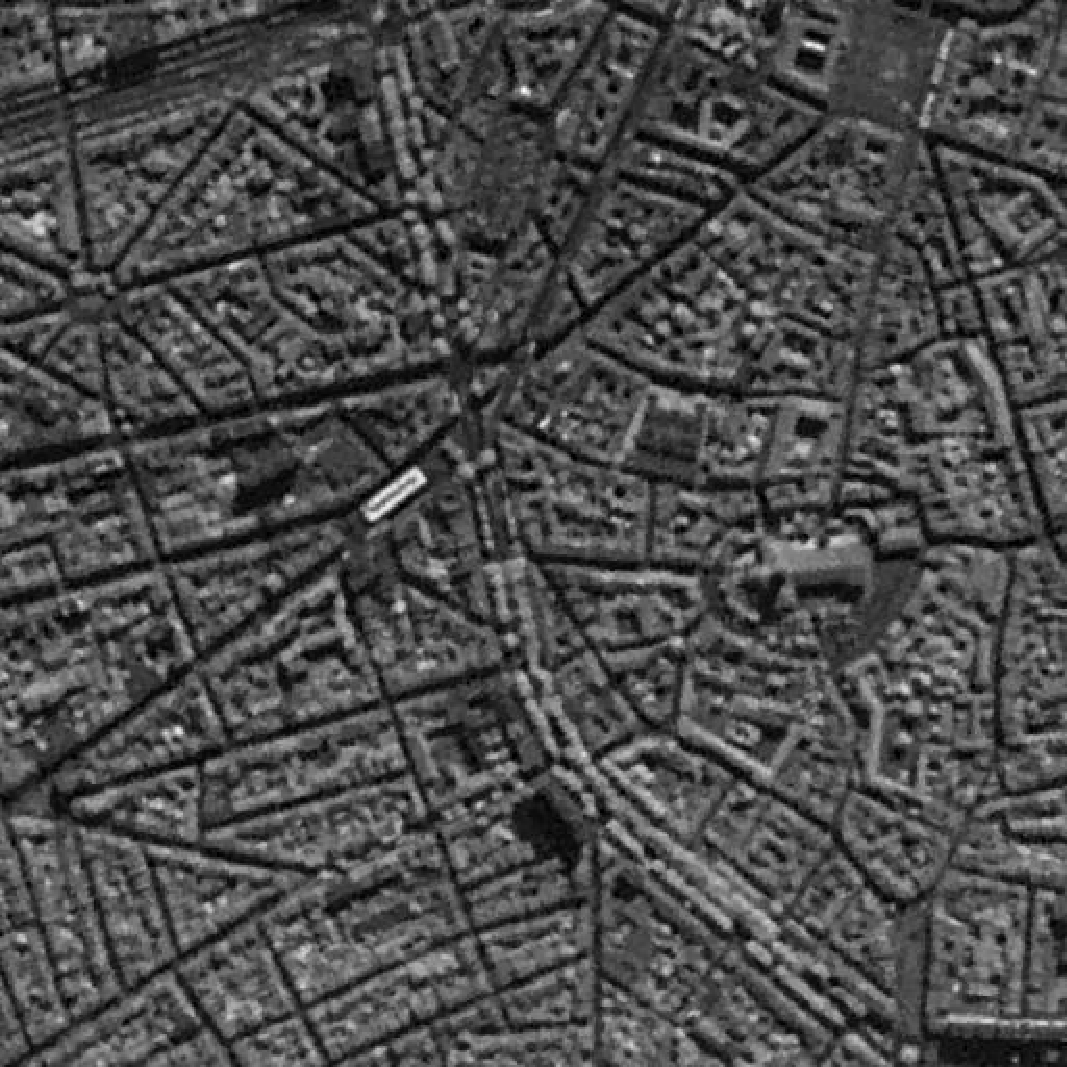}
        \subcaption{PLRPM} \label{satellite2-PLRPM}
    \end{minipage}
    \begin{minipage}[b]{.2\linewidth}
        \centering
        \includegraphics[width=\textwidth]{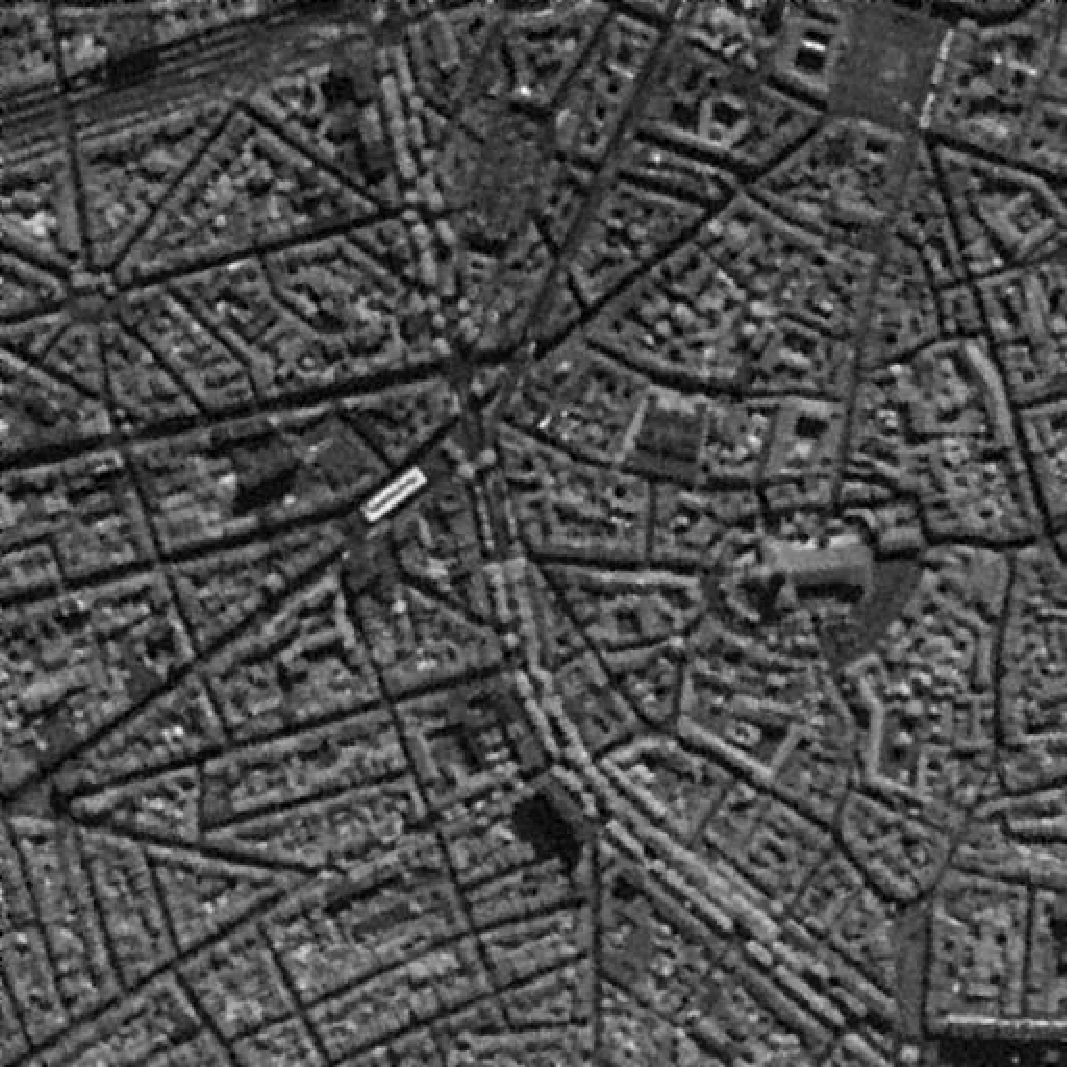}
        \subcaption{Ours} \label{satellite2-Ours}
    \end{minipage}
    \caption{Recovery results for the satellite2 image with disk blur and corrupted by 
    the noise of standard deviation $\sigma = 3$. (a) original image. 
    (b) noisy blurred image, PSNR=$21.68$. (c)-(h) recovered images.} \label{satellite2}
\end{figure}

\begin{figure}[htbp]
    \centering
    \begin{minipage}[t]{.2\linewidth}
        \centering
        \includegraphics[width=\textwidth]{satellite3.eps}
        \subcaption{Original} \label{satellite3-original}
    \end{minipage} 
    \begin{minipage}[t]{.2\linewidth}
        \centering
        \includegraphics[width=\textwidth]{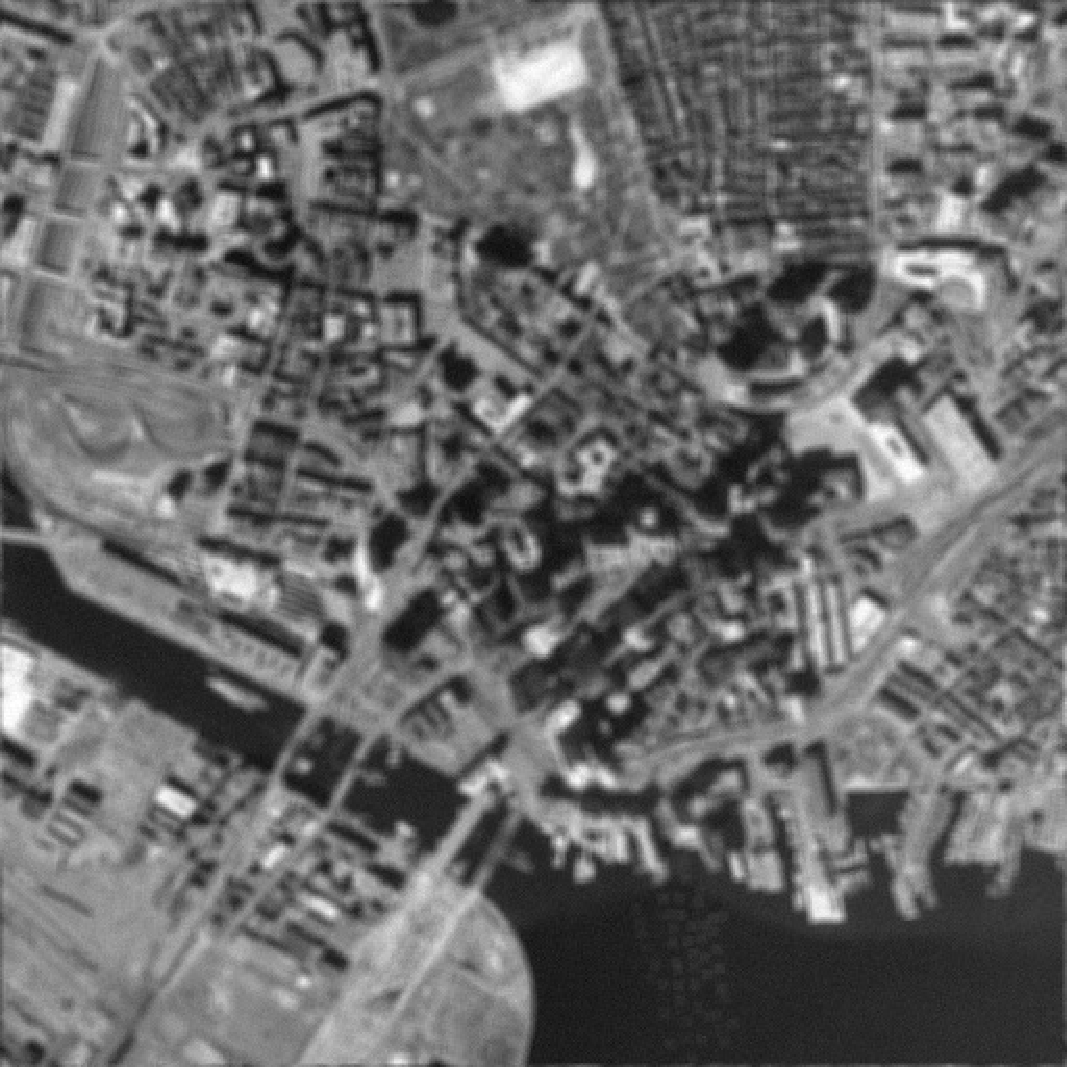}
        \subcaption{Blurred} \label{satellite3-blurred}
    \end{minipage}
    \begin{minipage}[t]{.2\linewidth}
        \centering
        \includegraphics[width=\textwidth]{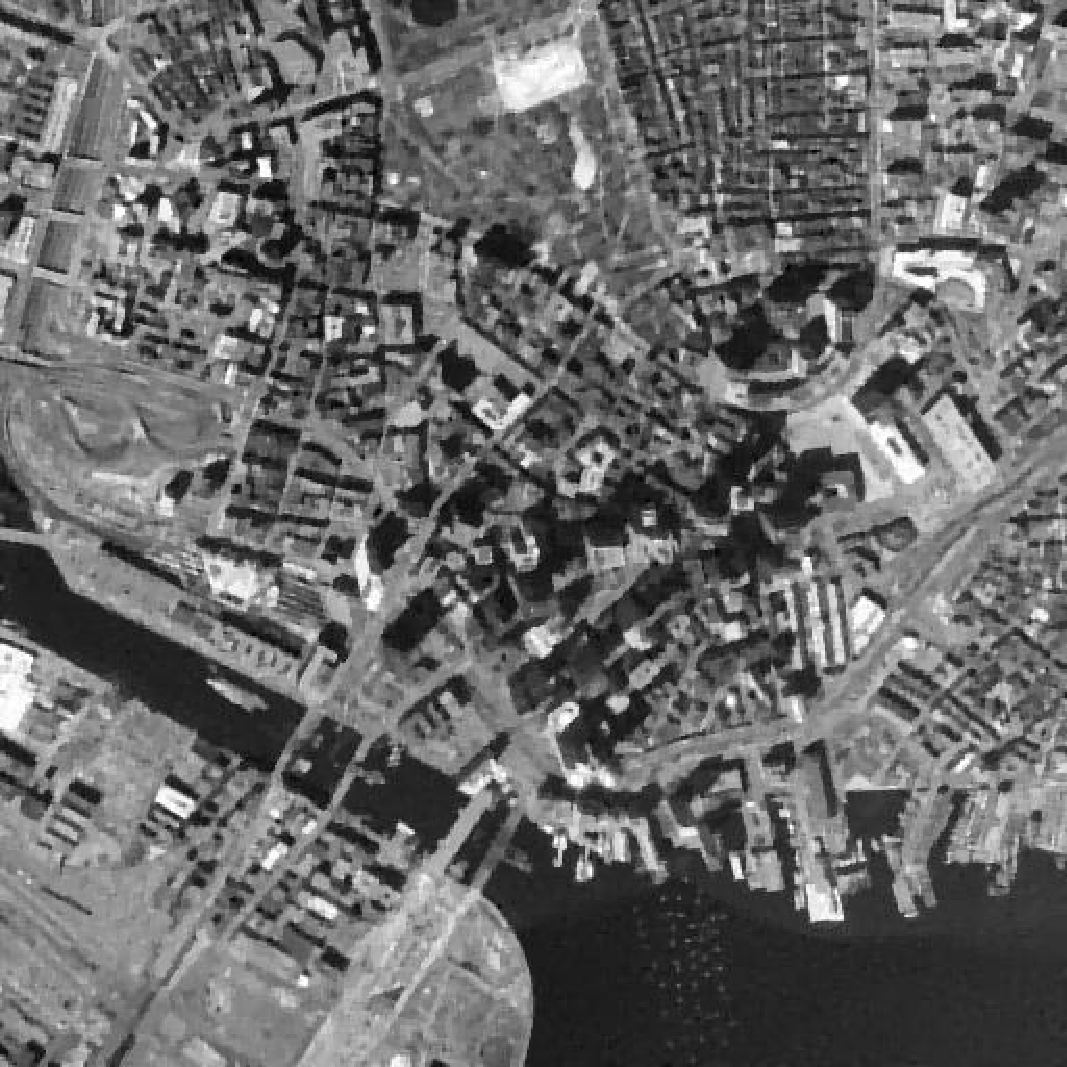}
        \subcaption{FastTV} \label{satellite3-FastTV}
    \end{minipage}
    \begin{minipage}[t]{.2\linewidth}
        \centering
        \includegraphics[width=\textwidth]{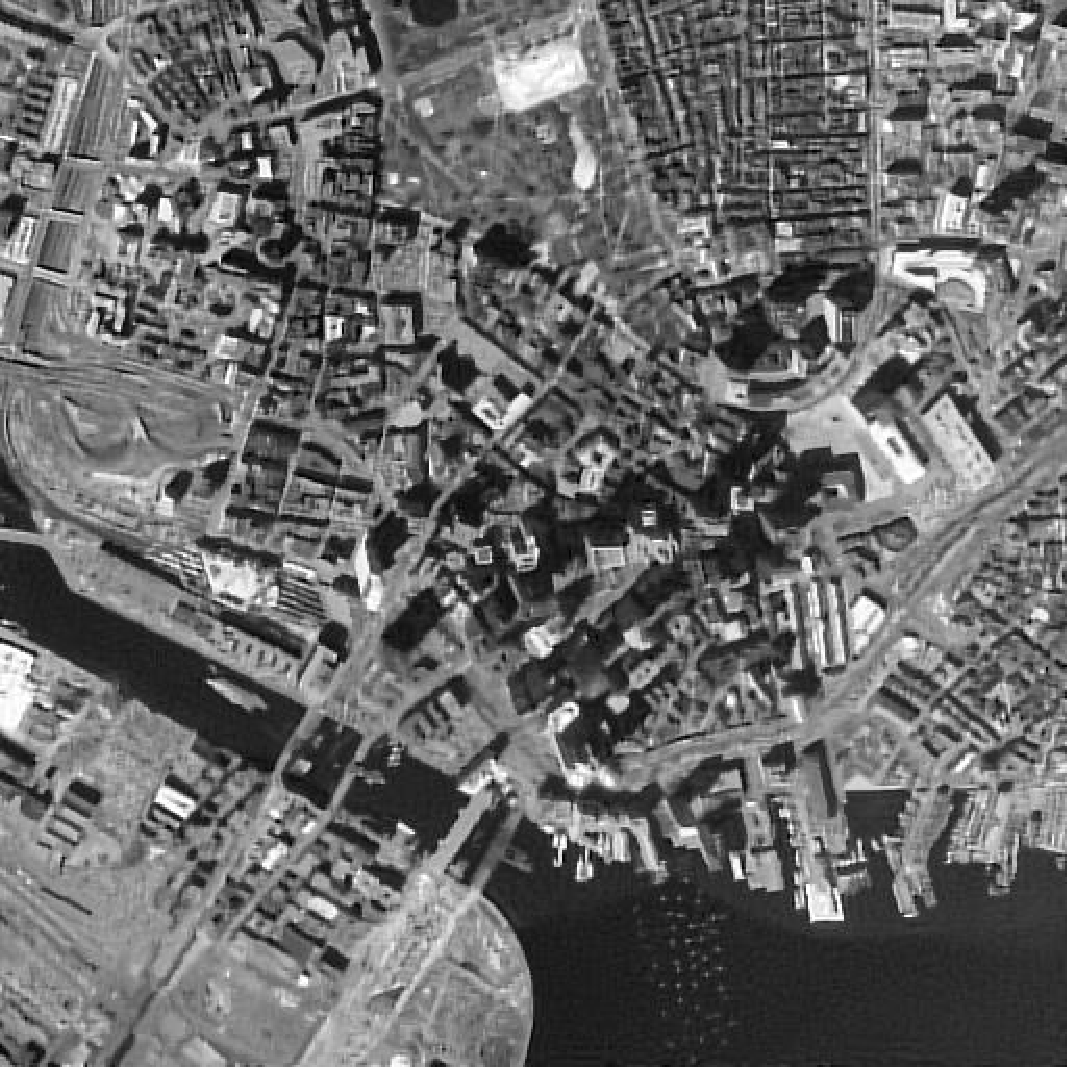}
        \subcaption{NLTV} \label{satellite3-NLTV}
    \end{minipage}
    
    \begin{minipage}[b]{.2\linewidth}
        \centering
        \includegraphics[width=\textwidth]{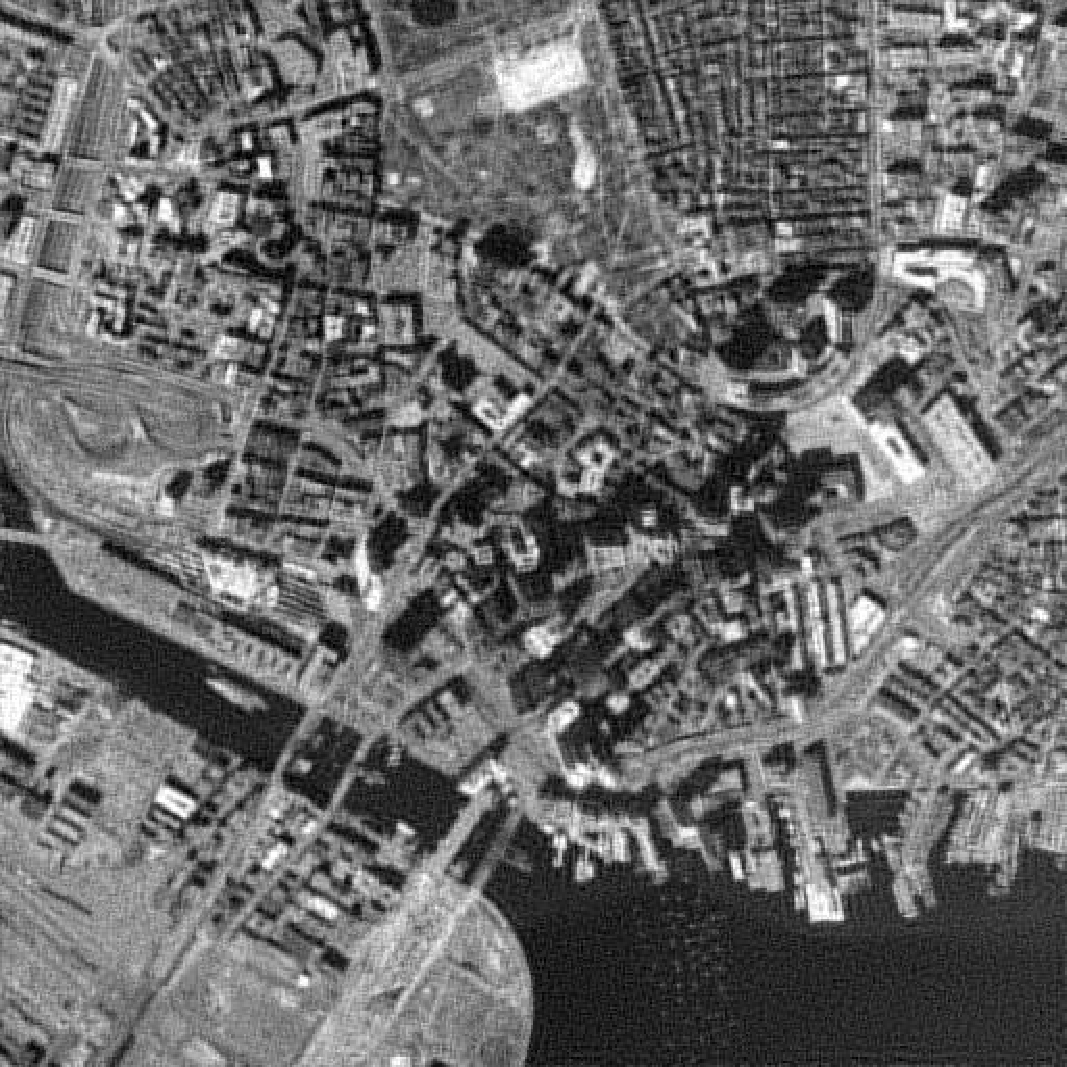}
        \subcaption{NLABH} \label{satellite3-NLABH}
    \end{minipage}
    \begin{minipage}[b]{.2\linewidth}
        \centering
        \includegraphics[width=\textwidth]{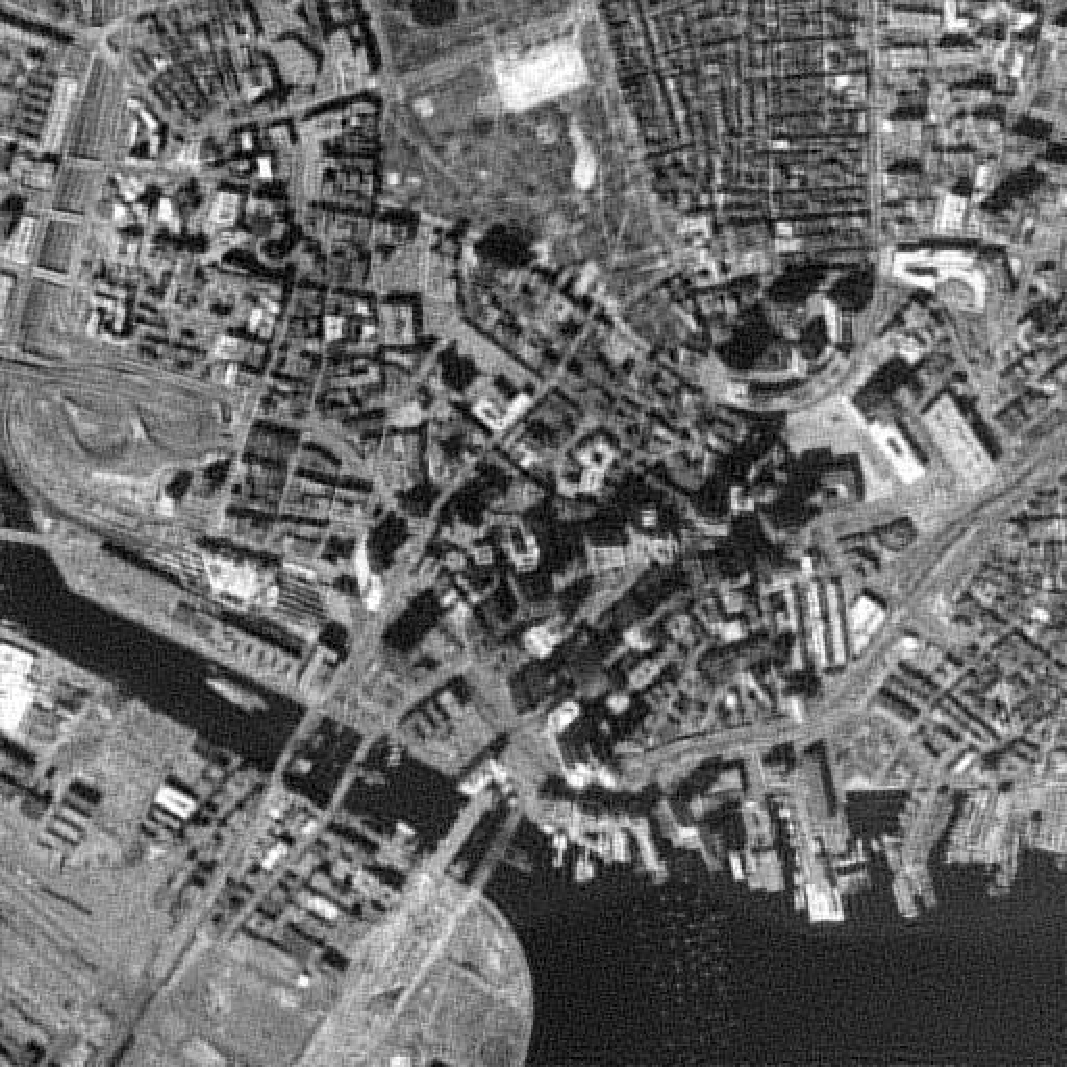}
        \subcaption{NFD} \label{satellite3-NFD}
    \end{minipage}
    \begin{minipage}[b]{.2\linewidth}
        \centering
        \includegraphics[width=\textwidth]{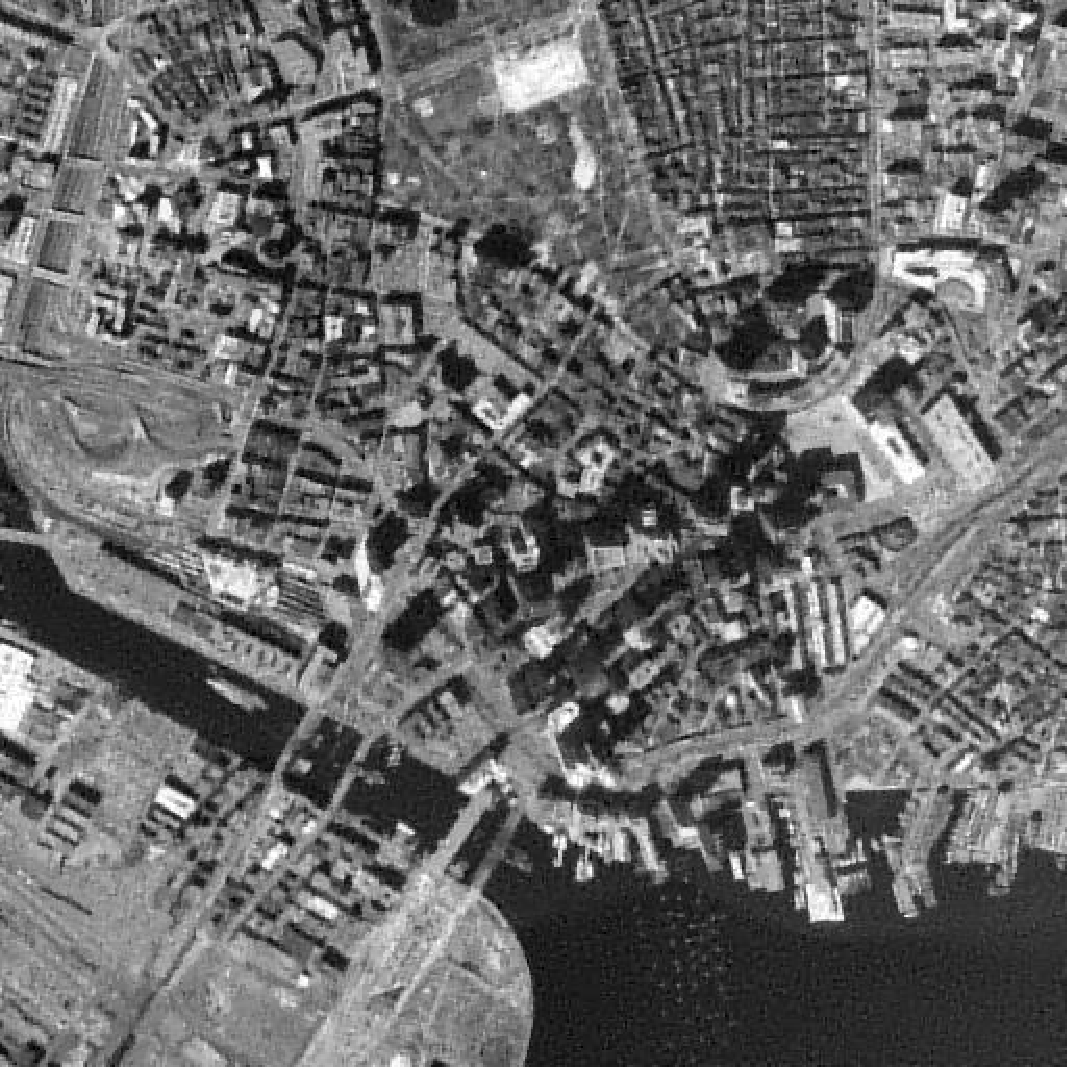}
        \subcaption{PLRPM} \label{satellite3-PLRPM}
    \end{minipage}
    \begin{minipage}[b]{.2\linewidth}
        \centering
        \includegraphics[width=\textwidth]{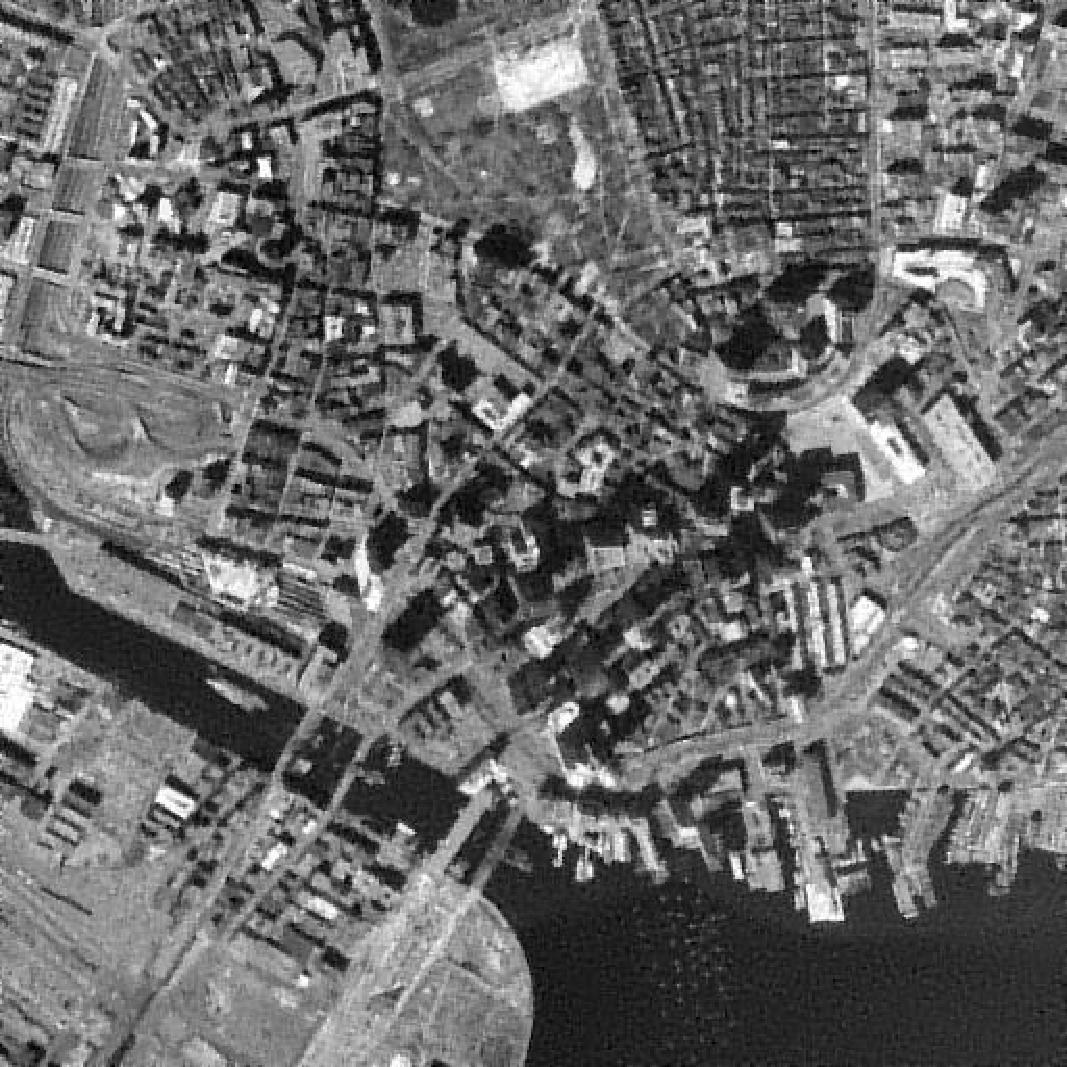}
        \subcaption{Ours} \label{satellite3-Ours}
    \end{minipage}
    \caption{Recovery results for the satellite3 image with average blur and corrupted by 
    the noise of standard deviation $\sigma = 3$. (a) original image. 
    (b) noisy blurred image, PSNR=$19.20$. (c)-(h) recovered images.} \label{satellite3}
\end{figure}

Now we report the numerical experiments of deblurring and denoising for the original
test images in Fig. \ref{original}. The corresponding results are shown in 
Figs. \ref{texture1}-\ref{satellite3}. Firstly, the restoration results 
for texture1 and texture2 are shown in Figs. \ref{texture1} and \ref{texture2}. 
Texture1 and texture2 images are blurred by disk kernel and average kernel, respectively. 
For these two experiments, we select $\lambda=45$ in the proposed model. We find that 
in Figs. \ref{texture1-FastTV} and \ref{texture2-FastTV}, there is minimal noise present, 
but the visual quality is over-smooth, with an evident loss of 
texture information. On the other hand, 
in Figs. \ref{texture1-NLABH}, \ref{texture1-NFD}, \ref{texture2-NLABH} 
and \ref{texture2-NFD}, texture details are better preserved, however, 
there is an increase in the noise level of the images. 
Compared to Figs. \ref{texture1-NLTV}, \ref{texture2-NLTV}, \ref{texture1-PLRPM}
and \ref{texture2-PLRPM}, the restoration results of the proposed model exhibit 
slightly higher noise levels but retain more texture information, resulting 
in a better visual quality, see Figs. \ref{texture1-Ours} and \ref{texture2-Ours}.
{For a detailed comparison, we also present zoomed-in views of 
a selected region from the texture1 image in Fig. \ref{texture1-box}. 
The corresponding restoration results for this region are shown in Fig. \ref{texture1-zoom}.}

\begin{figure}[!htbp]
    \centering
    \begin{minipage}[t]{.3\linewidth}
        \centering
        \includegraphics[width=\textwidth]{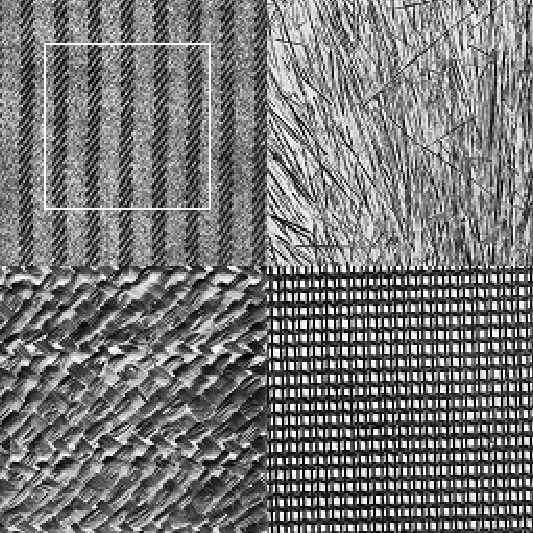} \subcaption{} \label{texture1-box}
    \end{minipage} 
    \begin{minipage}[t]{.3\linewidth}
        \centering
        \includegraphics[width=\textwidth]{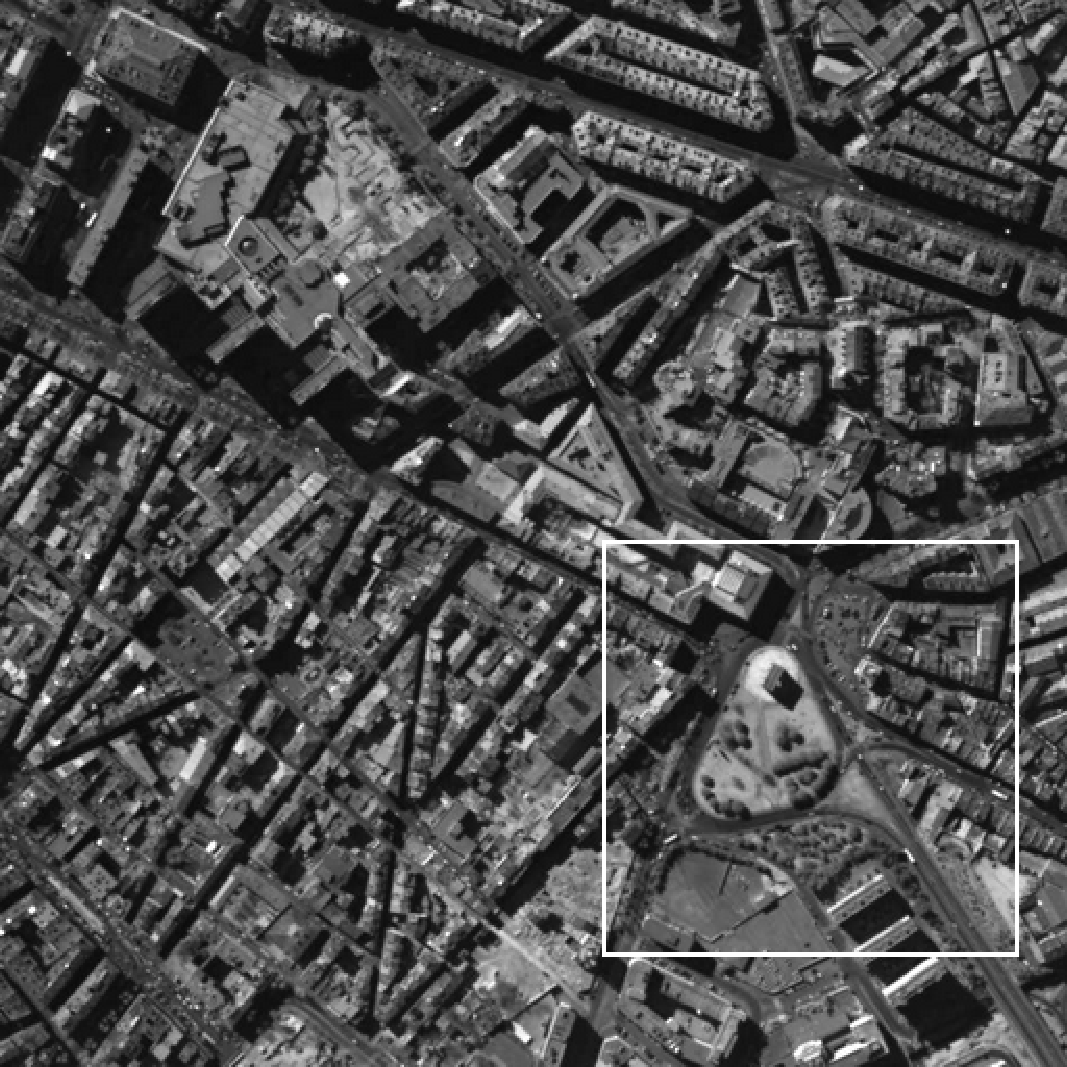} \subcaption{} \label{satellite1-box}
    \end{minipage}
    \caption{The original texture1 image and satellite1 image, the
    marked squares are marked for zooming.} \label{box}
\end{figure}

\begin{figure}[!htbp]
    \centering
    \begin{minipage}[t]{.2\linewidth}
        \centering
        \includegraphics[width=\textwidth]{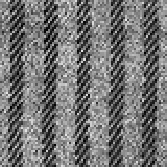}
        \subcaption{Original} \label{texture1-original-zoom}
    \end{minipage} 
    \begin{minipage}[t]{.2\linewidth}
        \centering
        \includegraphics[width=\textwidth]{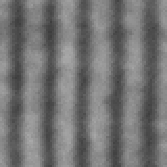}
        \subcaption{Blurred} \label{texture1-blurred-zoom}
    \end{minipage}
    \begin{minipage}[t]{.2\linewidth}
        \centering
        \includegraphics[width=\textwidth]{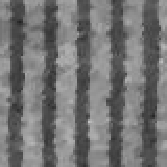}
        \subcaption{FastTV} \label{texture1-FastTV-zoom}
    \end{minipage}
    \begin{minipage}[t]{.2\linewidth}
        \centering
        \includegraphics[width=\textwidth]{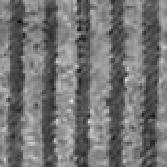}
        \subcaption{NLTV} \label{texture1-NLTV-zoom}
    \end{minipage}
    
    \begin{minipage}[b]{.2\linewidth}
        \centering
        \includegraphics[width=\textwidth]{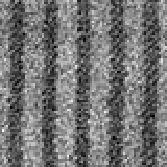}
        \subcaption{NLABH} \label{texture1-NLABH-zoom}
    \end{minipage}
    \begin{minipage}[b]{.2\linewidth}
        \centering
        \includegraphics[width=\textwidth]{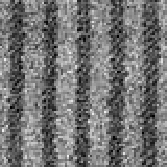}
        \subcaption{NFD} \label{texture1-NFD-zoom}
    \end{minipage}
    \begin{minipage}[b]{.2\linewidth}
        \centering
        \includegraphics[width=\textwidth]{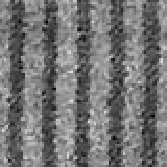}
        \subcaption{PLRPM} \label{texture1-PLRPM-zoom}
    \end{minipage}
    \begin{minipage}[b]{.2\linewidth}
        \centering
        \includegraphics[width=\textwidth]{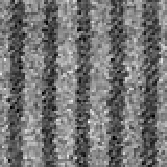}
        \subcaption{Ours} \label{texture1-Ours-zoom}
    \end{minipage}
    \caption{Zoomed-in images corresponding to the results shown in Fig. \ref{texture1}.} \label{texture1-zoom}
\end{figure}

\begin{figure}[htbp]
    \centering
    \begin{minipage}[t]{.2\linewidth}
        \centering
        \includegraphics[width=\textwidth]{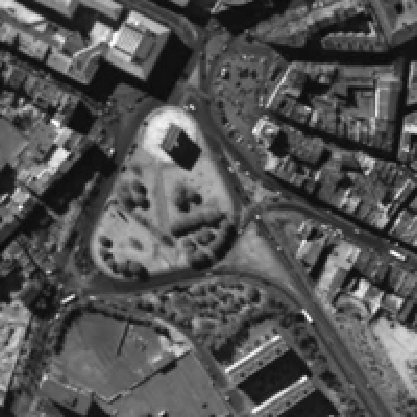}
        \subcaption{Original} \label{satellite1-original-zoom}
    \end{minipage} 
    \begin{minipage}[t]{.2\linewidth}
        \centering
        \includegraphics[width=\textwidth]{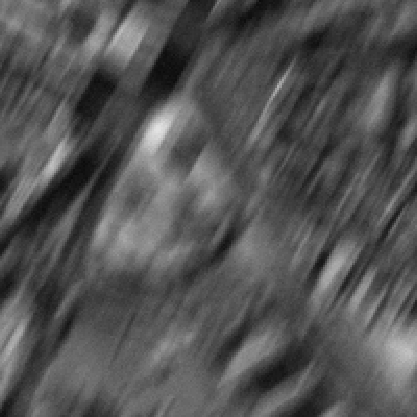}
        \subcaption{Blurred} \label{satellite1-blurred-zoom}
    \end{minipage}
    \begin{minipage}[t]{.2\linewidth}
        \centering
        \includegraphics[width=\textwidth]{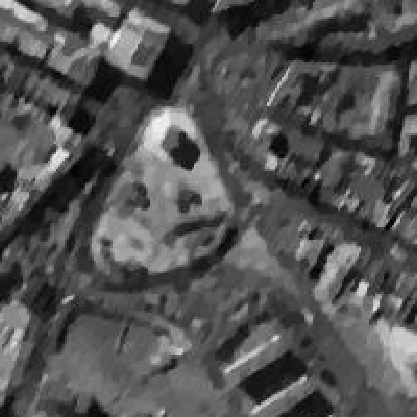}
        \subcaption{FastTV} \label{satellite1-FastTV-zoom}
    \end{minipage}
    \begin{minipage}[t]{.2\linewidth}
        \centering
        \includegraphics[width=\textwidth]{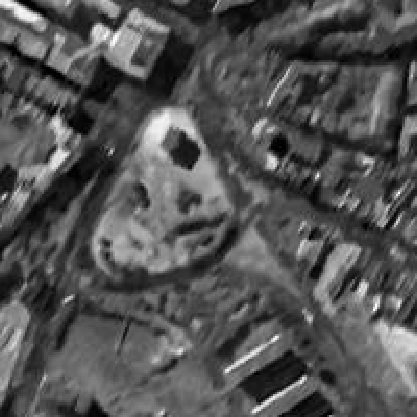}
        \subcaption{NLTV} \label{satellite1-NLTV-zoom}
    \end{minipage}
    
    \begin{minipage}[b]{.2\linewidth}
        \centering
        \includegraphics[width=\textwidth]{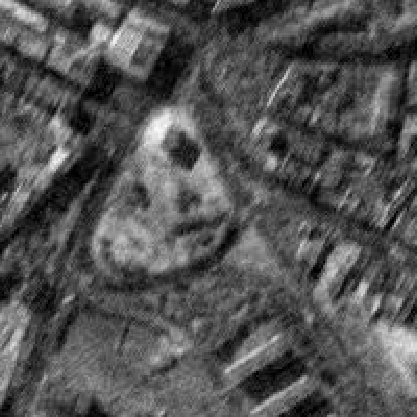}
        \subcaption{NLABH} \label{satellite1-NLABH-zoom}
    \end{minipage}
    \begin{minipage}[b]{.2\linewidth}
        \centering
        \includegraphics[width=\textwidth]{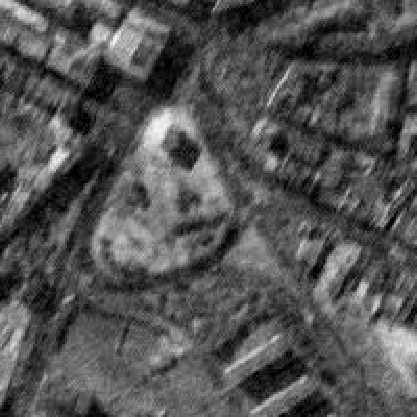}
        \subcaption{NFD} \label{satellite1-NFD-zoom}
    \end{minipage}
    \begin{minipage}[b]{.2\linewidth}
        \centering
        \includegraphics[width=\textwidth]{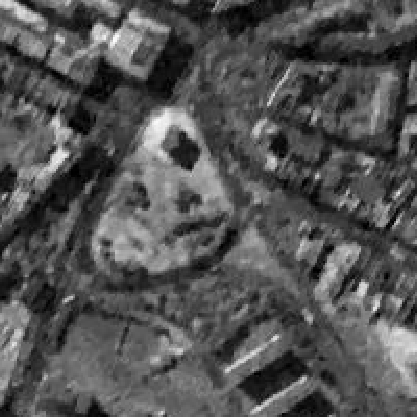}
        \subcaption{PLRPM} \label{satellite1-PLRPM-zoom}
    \end{minipage}
    \begin{minipage}[b]{.2\linewidth}
        \centering
        \includegraphics[width=\textwidth]{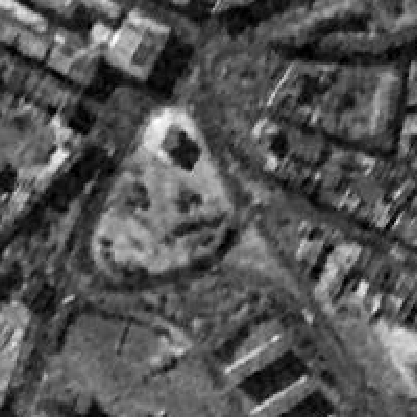}
        \subcaption{Ours} \label{satellite1-Ours-zoom}
    \end{minipage}
    \caption{Zoomed-in images corresponding to the results shown in Fig. \ref{satellite1}.} \label{satellite1-zoom}
\end{figure}

For the restoration results of the hybrid image, some new phenomena emerge, 
as illustrated in Fig. \ref{hybrid}. The central area of the hybrid image 
is rich in texture, while the surrounding area in smooth, 
as shown in Fig. \ref{hybrid-original}. The hybrid image is blurred by a motion kernel, 
posing a challenge for all models, see Fig. \ref{hybrid-blurred}.
For this experiment, we select $\lambda=15$ in the proposed model. 
It is observed from Fig. \ref{hybrid-FastTV} that the restoration 
result of FastTV contains fewer noise and artifacts
in smooth areas, but it fails to preserve texture information. NLABH and NFD 
noticeably amplify noise in smooth regions and produce obvious artifacts, as 
seen in Figs. \ref{hybrid-NLABH} and \ref{hybrid-NFD}. In the restoration 
results of the proposed model, noise in smooth areas is slightly more pronounced 
compared to NLTV and PLRPM, but texture preservation is better than PLRPM, and 
there are no artifacts in smooth regions as seen in the restoration results of NLTV, 
see Figs. \ref{hybrid-NLTV}, \ref{hybrid-PLRPM}, and \ref{hybrid-Ours}.

Finally, the restoration results for three satellite images are shown in Figs. 
\ref{satellite1}-\ref{satellite3}. The three satellite images are respectively 
blurred by motion kernel, disk kernel, and average kernel, as shown in 
Figs. \ref{satellite1-blurred}, \ref{satellite2-blurred}, and \ref{satellite3-blurred}.
We select $\lambda=15, 15, 10$ in the proposed model respectively 
for these three experiments. Similar to the previous experiments, the restoration 
results of FastTV and NLTV are relatively smooth, and the noise levels in the 
restoration results of NLABH and NFD are increased. Overall, the methods with better 
visual restoration results are NLTV, PLRPM, and the proposed method. Our method slightly 
reduces visual quality on satellite3 due to increased noise in smooth regions, 
but performs best on satellite1 and satellite2, as shown in 
Figs. \ref{satellite1-Ours}, \ref{satellite2-Ours}, and \ref{satellite3-Ours}. 
{Fig. \ref{satellite1-zoom} shows the zoomed-in restoration results 
corresponding to the selected region of the satellite1 image in Fig. \ref{satellite1-box}.}

{\section{Final remarks} \label{sec6}
The local well-posedness of the proposed model and the regularity 
of local solutions have been investigated. We hope that our work helps fill a 
gap in the theoretical analysis of similar image processing problems 
and may offer useful insight for future studies. 
Establishing global well-posedness conditions and the characterization of 
the steady-state limit for the general case are open and challenging problems
that need further study. 
We point out that the considered model does not possess a variational structure, 
and whether it can be explained from an energy point of view remains to be studied. 
It should also be emphasized that the theoretical analysis in this paper does not 
cover the case $\inf_{x \in \Omega} f(x) = 0$, which introduces extra degeneracy
and prevents the use of maximal regularity theory. Addressing this 
care may require brand-new ideas and techniques. Finally, numerical experiments 
have verified the effectiveness of the proposed model, while the design of fast algorithms 
that allow larger time steps represents another interesting direction for future work.}

\end{document}